\documentclass[11pt]{amsart}
\usepackage{amssymb}
\usepackage[mathscr]{eucal}
\usepackage{hyperref}
\usepackage{graphicx}
\usepackage{young}
\usepackage[vcentermath]{youngtab}

\usepackage{color}

\newtheorem{theorem}{Theorem}[section]
\newtheorem{corollary}[theorem]{Corollary}
\newtheorem{lemma}[theorem]{Lemma}

\theoremstyle{definition}
\newtheorem{definition}[theorem]{Definition}

\newtheorem{claim}[theorem]{Claim}

\theoremstyle{remark}
\newtheorem{remark}[theorem]{Remark}

\def\hhh{\mathfrak{h}}
\def\ppp{\mathfrak{p}}
\def\ggg{\mathfrak{g}}
\def\mmm{\mathfrak{m}}
\def\gl{\mathfrak{gl}}

\def\bz{{\bar{\bar 0}}}
\def\bo{{\bar{\bar 1}}}
\def\bc{{\mathbf{c}}}

\def\bw{{\bigwedge}}

\def\bfV{\mathbf{V}}
\def\bfG{{\mathbf{G}}}

\def\bfone{{\mathbf{1}}}

\def\bfa{{\mathbf{a}}}
\def\bfb{{\mathbf{b}}}
\def\bfc{{\mathbf{c}}}
\def\bfd{{\mathbf{d}}}
\def\bfi{\mathbf{i}}
\def\bfj{\mathbf{j}}
\def\bfr{\mathbf{r}}
\def\bfs{\mathbf{s}}
\def\bft{\mathbf{t}}
\def\bfk{\mathbf{k}}
\def\bfh{\mathbf{h}}

\def\calh{\mathcal{H}}
\def\caln{\mathcal{N}}
\def\calg{\mathcal{G}}

\def\frakG{\mathfrak{G}}

\def\rrtimes{{\scriptstyle{\rtimes\!\!\!\!\!\bigcirc}}}

\def\grsf{{\mathsf{gr}}}

\def\co{\mathcal{O}}
\def\GL{\text{GL}}
\def\Mat{\text{Mat}}

\def\fraksd{\mathfrak{S}_d}

\def\Ad{\mathsf{Ad}}

\def\ad{\mathsf{ad}}
\def\ev{\mathsf{ev}}

\def\str{\mathsf{str}}
\def\det{\mathsf{det}}
\def\op{\mathsf{op}}
\def\bfvotd{{{\textbf{V}}^{\otimes d}}}
\def\Votd{{V^{\otimes d}}}
\def\sfr{\textsf{r}}
\def\End{\text{End}}

\def\Dist{\text{Dist}}
\def\Lie{\text{Lie}}

\def\im{\text{Im}}

\def\col{\text{col}}
\def\row{\text{row}}

\def\bs{\textbf{s}}
\def\bfv{\textbf{V}}

\def\SHd{H_d}
\def\SHl{H_\ell}

\def\bbc{\mathbb{C}}
\def\bbz{\mathbb{Z}}

\def\ca{\mathcal{A}}

\newcommand{\red}[1]{{\color{red}#1}}

\numberwithin{equation}{section}
\begin{document}
	
\title[Super Vust theorem and higher Schur-Sergeev duality]{Super Vust theorem
and Schur-Sergeev duality for principal finite $W$-superalgebras}
\author{Changjie Cheng}
\address{School of Mathematical Sciences,   East China Normal University, No. 500 Dongchuan Rd., Shanghai 200241, China}
\email{cjcheng@math.ecnu.edu.cn}
\author{Bin Shu}
\address{School of Mathematical Sciences, Ministry of Education Key Laboratory of Mathematics and Engineering Applications \& Shanghai Key Laboratory of PMMP,  East China Normal University, No. 500 Dongchuan Rd., Shanghai 200241, China} \email{bshu@math.ecnu.edu.cn}
\author{Yang Zeng}
\address{School of Mathematics, Nanjing Audit University, No. 86 West Yushan Rd., Nanjing, Jiangsu Province 211815, China}
\email{zengyang@nau.edu.cn}
\thanks{This work is partially supported by the National Natural Science Foundation of China (Nos. 12071136,
12271345, 12461005), and by Science and Technology Commission of Shanghai Municipality (No. 22DZ2229014).}


\begin{abstract}
    Considering the general linear Lie superalgebra $\gl(m|n)=\gl(m|n)_\bz\oplus \gl(m|n)_\bo$ over $\bbc$,  we first formulate a super version of Vust theorem associated with a principal nilpotent element $e\in \gl(m|n)_\bz$.  As an application of this theorem, we then obtain  a Schur-Sergeev duality for principal finite $W$-superalgebras which is partially a super version of Brundan-Kleshchev's higher level Schur-Weyl duality established in \cite{BKl}.
\end{abstract}
\maketitle
	
\setcounter{tocdepth}{1}\tableofcontents	
\setcounter{section}{-1}
\section{Introduction}

The purpose of the present paper is to prove a super Vust theorem,
   and then, as an application, to extend Brundan-Kleshchev's higher Schur-Weyl duality to the super
case. We successfully achieve a super
version of Brundan-Kleshchev's result when $e$ is principal nilpotent (see Definition \ref{def:  principal nilpotent}).

\subsection{}\label{se:1}     We first recall Brundan-Kleshchev's work. Let $G=\text{GL}(V)$ and $\ggg=\gl(V)$ be the general linear Lie
 group and its Lie algebra on $V$ over $\bbc$ with $\dim_\bbc V=n$ respectively, and $e$ a
 nilpotent element of $\ggg$.  Set $\ggg_e=\{X\in\ggg\mid \ad{e}(X)=0\}$, the
 centralizer of $e$ in $\ggg$. There is a natural action $\phi$ of $\ggg_e$ on
 the tensor product $V^{\otimes d}$. The Vust theorem generalizes the classical Schur-Weyl duality in the setup of $\ggg_e$ (see \cite{KP}).  We briefly give an account of  it below, in connection with the finite $W$-algebras and the degenerate cyclotomic affine Hecke algebras.


Recall that the degenerate affine Hecke algebra $H_d$ is an associative algebra which is as a vector space, equal to the tensor product $\bbc[x_1,\ldots,x_d]\otimes \bbc S_d$ of a polynomial algebra and the group algebra of the symmetric group $S_d$.
{If we identify $\bbc[x_1,\dots,x_d]$ and
$\bbc S_d$ with the subspaces $\bbc[x_1,\dots,x_d] \otimes 1$ and
$1 \otimes \bbc S_d$ of $H_d$ respectively, then the associative algebra $H_d$ is generated by the elements $\{x_i,\; s_j=(j\; j+1)\in S_d\mid i=1,\ldots, d; j=1,\ldots,d-1\}$, and the multiplication is defined so that
$\bbc[x_1,\dots,x_d]$ and $\bbc S_d$ are subalgebras of $H_d$,
$s_ix_j=x_js_i$ if $j\ne i,i+1$, and $s_ix_{i+1}=x_is_{i}+1$.}

 On the other hand, associated with $e$ there
 is a twisted truncated polynomial algebra $\bbc_l[x_{1},\ldots,x_{d}]$, 
 where $l$ is a positive integer determined by the
 Young diagram of the adjoint orbit $G.e$, and $\bbc_l[x_{1},\ldots,x_{d}]$ is a  quotient of the polynomial algebra $\bbc[x_1,\ldots,x_{d}]$ by the ideal
 generated by  $x_i^l$, $i=1,\ldots,d$. So $x_i^l=0$  in $\bbc_l[x_1,\ldots,x_d]$ for $i=1,\ldots,d$.
 {Extend the usual right action of $S_d$ on $V^{\otimes d}$ to an action of
the twisted tensor product $\bbc_l[x_{1},\ldots,x_{d}]\rrtimes \bbc\fraksd$ (which has the ground space $\bbc_l[x_{1},\ldots,x_{d}]\otimes \bbc\fraksd$)
by defining the action of $x_i$
to be as the endomorphism
$1^{\otimes(i-1)}\otimes e \otimes 1^{\otimes (d-i)}$.
This commutes with the natural action of $\mathfrak{g}_e$,
making $V^{\otimes d}$ into a $(U(\mathfrak{g}_e),
\bbc_l[x_{1},\ldots,x_{d}]\rrtimes \bbc\fraksd)$-bimodule.}
 {We have defined a homomorphism $\psi$ from $\mathbb{C}_{l}[x_{1},\ldots,x_{d} ]\rrtimes
\mathbb{C}{\fraksd}$ to $\text{End}_{U(\ggg_{e})}(V^{\otimes d})^{\operatorname{\op}}$.
The classical Vust's theorem says
\begin{align}	\label{form: vust}
\psi(\mathbb{C}_{l}[x_{1},\ldots,x_{d} ]\rrtimes
\mathbb{C}{\fraksd})=\text{End}_{U(\ggg_{e})}(V^{\otimes d})^{\operatorname{\op}}.
\end{align}}
Then a generalized Schur-Weyl duality, associated with $e$ further asserts
  that $\ggg_e$ and $\bbc_l[x_{1},\ldots,x_{d}]\rrtimes \bbc\fraksd$ are
  double-centralizers in $\End_\bbc(\Votd)$. This is to say, addition to (\ref{form: vust}), the natural representation $\phi:U(\ggg_e)\rightarrow \End_\bbc(\Votd)$ satisfies
\begin{align}\label{form: vust+}	
\text{End}_{\mathbb{C}_{l}[x_{1},\ldots,x_{d} ]\rrtimes
\mathbb{C}{\fraksd}}(V^{\otimes d}) = \phi(U(\ggg_{e})).
\end{align}
Starting from here, Brundan and Kleshchev went further, considering the filtered
deformation of the above double-centralizers. Consider the finite  $W$-algebra
$W_{\chi}$ with $\chi$ being the linear dual of $e$ via the Killing form on
$\ggg$. According to \cite{P2}, a well-known result says that  $W_\chi$ is actually a
filtered-deformation of $U(\ggg_e)$.     There is a left action of $W_\chi$ on $V^{\otimes d}$, denoted by $\Phi$. On the other side, there is a degenerate cyclotomic Hecke algebra $H_d(\Lambda)$ which is some remarkable finite-dimensional quotient of $H_d$ by the two-sided ideal parameterized by a dominant weight $\Lambda$ of level $l$ (for the root system of type $A_\infty$) related to $e$. Brundan-Kleshchev defined a nontrivial right action of $H_d(\Lambda)$ on $V^{\otimes d}$, denoted by $\Psi$,
such that $(\Phi,\Psi)$  is a filtered deformation of $(\phi,\psi)$ on $V^{\otimes d}$.
They obtained the following higher
Schur-Weyl duality:
{\begin{align*}
	\text{End}_{H_{d}(\Lambda)}(\Votd) &= \Phi(W_{\chi}),\cr
\Psi(H_{d}(\Lambda))&=\text{End}_{W_{\chi}}(\Votd)^{\operatorname{\op}}.
\end{align*}}
\subsection{} Recall that Sergeev's super version of Schur-Weyl duality (which is called Schur-Sergeev
duality in the present paper) describes the double centralizer property of
general linear Lie superalgebras and symmetric groups happening on the tensor
products of  super spaces (see Theorem \ref{ss dual} precisely).

 Consider $V=
\mathbb{C}^{m|n}$,  $\ggg =\gl(m|n)$ and  the natural action of   $\ggg$ on
$\Votd$ is defined via
\begin{align}\label{eq: superaction0}
\rho(X)(v_{1}\otimes v_{2}\otimes \cdots \otimes v_ {d})
=\sum_{i=1}^{d}(-1)^{|X|(|v_{1}|+|v_{2}|+\cdots +|v_{i-1}|) }v_{1}\otimes
v_{2}\otimes \cdots \otimes Xv_{i}\otimes \cdots\otimes v_{d}
 \end{align}
 for any $\bbz_2$-homogeneous $X \in \ggg$  and $v_i\in V$ ($i=1,\ldots,d$) with
 an appointment  $v_0=0$.
 Here and further  $|w|$ always stands for the parity of $w\in W_{|w|}$ when we talk about a superspace $W=W_\bz\oplus W_\bo$.

\subsubsection{Super Vust theorem}\label{sec: super Vth 0}
Let $e$ be any given principal nilpotent element in $\ggg_{\bz}$.
 As to the twisted action {$\psi_d$} of the symmetric group  ${\fraksd}$ generated by
 transpositions $\{s_{i} = (i\;i+1)\mid i = 1,2,\ldots, d-1\}$, the action of
 $s_i$ is defined as below:
\begin{align}\label{super symaction}
\psi_{d}(s_{i}): v_{1}\otimes \cdots \otimes v_{d}\mapsto v_{1}\otimes
\cdots\otimes \tau (v_{i}\otimes v_{i+1 }) \otimes\cdots\otimes v_{d},
\end{align}
where $\tau: V\otimes V\rightarrow V\otimes V$  sends $v\otimes w$ to
$(-1)^{|w||v|}w\otimes v$.

As one of our main results, we have  the following super Vust theorem.

\begin{theorem}\label{thm Vust0}  Keep the notations and assumptions as above. In particular, $\ggg=\gl(m|n)$, $e\in\ggg_\bz$ is a principal nilpotent element.
    Let $\mathscr{A}$
denote the subalgebra generated by {$\psi_{d}({\fraksd})$} in
$\End_\bbc(V^{\otimes d})$. Then $\text{End}_{\rho({\ggg_{e}})}(V^{\otimes d}) =
\mathscr{A}$.
\end{theorem}
     The proof will be given in \S \ref{sec 4}.

\subsubsection{} For the other side of the desired double centralizers, a careful inspection and thorough argument involving parities shows that Brundan-Kleshchev's original strategy as in \cite{BKl} can be carried out here. Roughly speaking, our strategy is that we first extend  the  Schur-Sergeev
duality for $\gl(m|n)$ to the level of $\gl(m|n)_e$ for a principal nilpotent
$e\in\gl(m|n)_\bz$, where the general Lie superalgebra  $\gl(m|n)$ has
$\bbz_2$-superspace decomposition $\gl(m|n)=\gl(m|n)_\bz\oplus \gl(m|n)_\bo$ for
$\bbz_2=\{\bz,\bo\}$, which says that $U(\gl(m|n)_e)$ and
$\bbc_n[x_{1},\ldots,x_{d}]\rrtimes \bbc\fraksd$ (assume that $m\leq n$) play the
role of the double centralizers in $\End_\bbc(V^{\otimes d})$. Here the
algebra $\bbc_n[x_{1},\ldots,x_{d}]\rrtimes \bbc\fraksd$ has the same meaning as defined
previously. However, the action of the generators  $s_i$ on $V^{\otimes d}$ has
the super meaning (see  (\ref{super symaction})).  Then we exploit Brundan-Kleshchev's arguments as  in \cite{BKl} to establish the
Schur-Sergeev duality for the corresponding finite $W$-superalgebra. This is
carried out  for principal nilpotent elements
 in the present paper.
We accomplish this by a couple of steps.



 Owing to \cite{BBG}, associated
with a principal nilpotent element $e\in \gl(m|n)_{\bz}$ the centralizer
$\gl(m|n)_e$  is described clearly. Consequently, $U(\ggg_e)$ and its filtered
deformation---finite $W$-superalgebra $W_\chi$ are well investigated in \cite{BBG}.
Those results  are the start-point of our arguments.

 With a super version of Vust theorem already established,  we then
develop some counterpart theory in the super case as some necessary preparations, i.e., the super version of Premet's theorem on the gradation
isomorphism of finite $W$-algebras as in \cite{P2}.

\subsubsection{Filtrations (in ``Dynkin" sense/Kazhdan sense) on $U(\ggg)$ and then on $W_\chi$ for a basic classical Lie superalgebra $\ggg$ and a nilpotent $e\in\ggg_\bz$} \label{sec finite W}
For any basic classical Lie superalgebra $\ggg$
over $\bbc$ and any given nilpotent element $e\in\ggg_\bz$, fix an
$\mathfrak{sl}_2$-triple $f,h,e\in\ggg_\bz$, and keep the notation
$\mathfrak{g}_e=\text{Ker}(\ad e)$ in $\mathfrak{g}$. {If we set ${\ggg}(i)=\{x\in{\ggg}\mid[h,x]=2ix\}$, then} the linear
operator ad\,$h$ defines a $\frac{\mathbb{Z}}{2}$-grading
$\mathfrak{g}=\bigoplus_{i\in\frac{\mathbb{Z}}{2}}\mathfrak{g}(i)$ 
on $\mathfrak{g}$, {which is called a Dynkin grading (In fact, the original Dynkin grading is defined as $\mathfrak{g}=\bigoplus_{i\in\mathbb{Z}}\mathfrak{g}(i)'$ with ${\ggg}(i)'=\{x\in{\ggg}\mid[h,x]=ix\}$. To be compatible with our settings in \S\ref{sec: skryabin}, a little change is made here, which of course will not cause any confusion).} Define the Kazhdan degree on $\mathfrak{g}$ by declaring
$x\in\mathfrak{g}(j)$ to be $(j+1)$.     Consequently, one can define Dynkin filtration/Kazhdan filtration both on $U(\ggg)$ and $W_\chi$ respectively.
Recall that in \cite{ZS} a finite $W$-superalgebra is introduced for any
basic classical Lie superalgebra $\ggg$ over $\bbc$ and any given nilpotent
element $e\in\ggg_\bz$ (see  Definition
\ref{W-C} for the definition of $W_\chi$ corresponding to $e$,  where
$\chi=(e,-)$ is the linear dual of $e$ associated with the
non-degenerate bilinear form defined on $\ggg_\bz$), and
the PBW structure theorem for the
finite $W$-superalgebra $W_\chi$ is presented.
It is shown that the
construction of finite $W$-superalgebras $W_\chi$   can be divided into two
cases, in virtue of the parity of the so-called ``judging number" $\sfr$ which is equal to the dimension of $\ggg(-\frac{1}{2})_{\bo}$ (see \cite[Theorem 4.5]{ZS}). Moreover, under the Kazhdan
grading $W_\chi$ is a filtered deformation of $S(\ggg_e)$ or $S(\ggg_{e})\otimes \mathbb{C}[\Theta]$ with $\mathbb{C}[\Theta]$ being an exterior algebra generated by an element $\Theta$, depending on the parity of $\sfr$ we discussed above (see Theorem \ref{ZW pbw}). In the present paper, we will denote by $\grsf'$ for the graded algebra of  $U(\ggg)$ and also its subalgebras and quotients under the Kazhdan grading.

    Another important ingredient in our arguments is to prove an analogue of Premet's result
    \cite[Proposition 2.1]{P2}, which says that when $\sfr$ is even, $W_\chi$ is indeed a filtrated
    deformation of $U(\ggg_e)$, associated with the filtration
    of $W_\chi$ arising from the Dynkin grading on the above $(\ggg, e)$
    provided by $\mathfrak{sl}_2$-triple $\{e,h,f\}$. In the following, the associated grading is
    denoted by $\grsf$.

    \begin{theorem}(see Theorem \ref{cdd1})  \label{thm: good grading0}
        Let $\ggg$ be a basic classical Lie superalgebra over $\bbc$.
    For any given nilpotent $e\in\ggg_\bz$, the
    associated graded algebra
    $\grsf W_\chi$  is isomorphic to  $U(\ggg_e)$ if $\sfr$ is even.
\end{theorem}
The above result is new.  It is notable that under the Dynkin grading, we can always choose $\sfr$ to be an even number for any nilpotent element $e\in \ggg_\bz$ for $\ggg=\gl(m|n)$  (see Remark \ref{pbw1}). So the isomorphism in the above theorem always exists  for all nilpotent $e\in\ggg_\bz$ for $\ggg=\gl(m|n)$ (in the case for  a principal nilpotent $e\in\gl(m|n)_\bz$, this  was observed in \cite[Remark 4.6]{BBG}).

 However, for a general pair $(\ggg, e)$ the judging number $\sfr$ is not necessarily even, then in the above statement $W_\chi$ has to be replaced by a refined $W$-superalgebra $W'_\chi$ (see Theorem \ref{cdd1}). The refined
$W$-superalgebra is defined as
	$$W_{\chi}':= Q_{\chi}^{\ad\mmm'},$$
where ${\mmm}'={\mmm}$ if $\sfr$ is an even number, and ${\mmm}'$ is a subspace properly containing ${\mmm}$  when $\sfr$ is an
odd number (see \S\ref{kazhdan} for  details). Clearly, $W'_\chi$ is a subalgebra of $W_\chi$.
In fact, the PBW theorem of $W_\chi'$
is of much the same as that of Theorem \ref{pbw} with $\sfr$ being odd, just abandoning the related
topics on the element $\Theta_{l+q+1}$ (see \S\ref{4.2} for details).
The authors also showed that the graded algebras $\grsf'W_\chi'\cong
S(\ggg_e)$ as $\bbc$-algebras under the Kazhdan grading in \cite[Corollary
3.8]{ZS1}. We refer to \cite[Theorem 3.7]{ZS1} for
more details.

 \subsubsection{Filtered deformations and Schur-Sergeev duality for principal $W$-superalgebras} \label{sec: SS
 duality 0}
 Turn back to the case $\ggg=\gl(m|n)$ which is a general linear Lie superalgebra on the superspace $V=\bbc^{m|n}$ with basis $\{v_{\bar 1},\ldots, v_{\overline{m}}\in V_\bz\}\cup \{v_1,\ldots,v_n\in V_\bo\}$,  and $e=\sum_{i=1}^{m-1}e_{\bar i,\overline{i+1}}+\sum_{j=1}^{n-1}e_{j,j+1}$  for $I(m|n)=\{\bar 1,\ldots,\overline{m}; 1,\ldots,n\}$ with parities $|\bar i|=\bz$, and $|j|=\bo$ ($i=1,\ldots,m$ and $j=1,\ldots,n$).
 Associated with $e$, we define a pyramid as in (\ref{pyramid})  where we make an assumption $m\leq n$ without loss of generality.
 Define a $\bbz$-grading $\ggg=\bigoplus_{i\in\bbz}\ggg(i)$ on $\ggg$ with $\ggg(i)$ being spanned by $e_{s,t}$ satisfying $i=\text{deg}(e_{s,t}):=\text{col}(t)-\text{col}(s)$, where $\text{col}(t)$ and $\text{col}(s)$ respectively stand for the column coordinates of the positions with $t$ and $s$ lying in the pyramid (\ref{pyramid}). Set $\ppp=\bigoplus_{i\geq 0}\ggg(i)$, $\hhh=\ggg(0)$ and $\mmm=\bigoplus_{i<0}\ggg(i)$. By \cite{BBG}, $\ggg_e$ is a graded subalgebra of $\ppp$,  precisely as described there.
 Consider the
 natural representation 
 of $\ggg=\gl(m|n)$ on $V= \mathbb{C}^{m|n}$, and  the natural action of   $\ggg$ on $\Votd$. Then one has a representation
 $\phi_d$ of $\ggg_e$ on $\Votd$, and the action $\psi_{d}$ of the symmetric group
 ${\fraksd}$ on $V^{\otimes d}$ (see \S\ref{sec: super Vth 0}).  {Let  $ \bar\psi_d:\mathbb
C_n[x_1,\dots,x_d]\rrtimes \mathbb C {\fraksd}\rightarrow
\operatorname{End}_{\mathbb C}(V^{\otimes d})^{\op}$}
 be the homomorphism arising
from the
right action of   $\mathbb C_n [x_1,\dots,x_d]\rrtimes \mathbb C {\fraksd}$  on
$V^{\otimes d}$.

\begin{theorem}\label{doublecentralizerin}
Let  $e\in\ggg=\gl(m|n)_\bz$ be a principal nilpotent element. The following double centralizer property holds:
\begin{align*}
	\phi_{d}(U(\ggg_{e})) &= \text{End}_{\mathbb{C}_{n}[x_{1},\ldots,x_{d}
]\rrtimes \mathbb{C}{\fraksd}}(V^{\otimes d}),\\
	\text{End}_{U(\ggg_{e})}(V^{\otimes d})^{\op} &=\bar\psi_{d}(\mathbb{C}_{n
}[x_{1},\ldots,x_{d}]\rrtimes \mathbb{C}{\fraksd}).
\end{align*}
\end{theorem}		

The complete proof of the above theorem will be left till \S\ref{2.5}.

\subsection{}
Now we turn to the finite $W$-superalgebras. Thanks to \cite{BBG}, $W_\chi$ for a principal nilpotent $e\in\ggg_\bz$ can be described precisely as a subalgebra of $U(\ppp)$, and $W_\chi$ can also be considered as a subalgebra of $U(\hhh)$ via Miura transform (see Remark \ref{W realization}(3)). Under the filtration arising from the grading on $U(\ppp)$ as  an enveloping algebra,  $W_\chi$ is a filtered subalgebra of $U(\ppp)$, whose gradation is isomorphic to $U(\ggg_e)$.

For any given $\bc=(c_1,\ldots,c_n)\in \bbc^n$ and any $e_{i,j}\in \ppp$, there is a one-dimensional $U(\ppp)$-module $\bbc_\bc=\bbc 1_\bc$ with the action $
e_{i,j}. 1_\bc=\delta_{i,j}(-1)^{|i|} c_{\text{col}(i)}1_\bc$ on the generator $1_\bc$.
Correspondingly,  there is a tensor representation $\Phi_{d,\bc}$ of $W_\chi$ on  $V_\bc^{\otimes d}=\bbc_\bc\otimes V^{\otimes d}$.

Let $\Lambda_\bc =\sum_{i=1}^n\Lambda_{c_i}$ be an element of the free abelian group generated by symbols     $\{\Lambda_{c_i}\mid i=1,\ldots,n\}$. The corresponding degenerate cyclotomic
Hecke algebra ${\SHd}(\Lambda_\bc)$
is the quotient  of the degenerate affine Hecke algebra ${\SHd}$ (see Definition \ref{def: dasha}) by the two-sided ideal generated
by  $\prod_{i=1}^n(x - c_i)$. The operator
 $$ \Omega =\sum_{i,j\in I(m|n)} (-1)^{|j|} e_{i,j} \otimes e_{j,i}$$
  on $\bbc_\bc\otimes V^{\otimes d}$ will play a crucial role in the arguments.
 By exploiting the arguments as in \cite{BK0} and  \cite{BKl} to the present case,  we will introduce  the right action of ${\SHd}$ on
 $V_\bc^{\otimes d}$ via  \begin{equation*}
x_i:= \bfone^{d-i}\otimes x \otimes \bfone^{i-1},\qquad
s_j := \bfone^{d-j-1}\otimes s \otimes \bfone^{j-1}
\end{equation*}
for all $1 \leq i \leq d$,   $1 \leq j \leq d-1$ with $d\geq 1$, where the operators $x$ and $s$ are defined through $\Omega$ respectively (see \S\ref{6} for details).
So we finally have
a homomorphism
{\begin{equation*}\label{psid}
\Psi_{d, \bc}: {\SHd}(\Lambda_\bc) \rightarrow \operatorname{End}_{\mathbb C}(V_\bc^{\otimes
d})^{\operatorname{\op}}.
\end{equation*}}
Then we have the following maps:
\begin{equation*}
W_{\chi} \stackrel{\Phi_{d,\bc}}{\longrightarrow}
\operatorname{End}_{\mathbb C}(V_\bc^{\otimes d})
\stackrel{\Psi_{d,\bc}}{\longleftarrow}
{\SHd}(\Lambda_\bc).
\end{equation*}
 Owing to Theorem
 \ref{doublecentralizerin} and the gradation property of $W_\chi$, the images of the maps $\Phi_{d,\bc}$ and $\Psi_{d,\bc}$ turn out to satisfy the double centralizer
property (see Theorem \ref{thm: concluding}). 
 Take a special value of $\bc=(0,0,\ldots,0):=\bf{0}$ and $\Lambda=\Lambda_{\bf0}$. Let $(W_{\chi})_d$ denote the image of $\Phi_{d,\bf{0}}$, and ${\SHd}(\lambda)$ the image of $\Psi_{d,\bf{0}}$.
 We finally have the following result:
\begin{theorem} \label{thm:dcthm}
When an even element $e$ is a principal nilpotent element in $\ggg=\gl(m|n)$, the following double centralizer property hold:
		\begin{align*}
		(W_{\chi})_d &= \operatorname{End}_{{\SHd}(\lambda)}(V^{\otimes
d}),\cr
		\operatorname{End}_{(W_{\chi})_d}(V^{\otimes d})^{\operatorname{\op}}&=
{\SHd}(\lambda).
	\end{align*}	
\end{theorem}

The proof of the above theorem will be given in the concluding section \S\ref{pf
last}.  It is worthwhile mentioning that according to some anonymous  referee's statements, there is an alternative proof via some work on categorifications as in \cite{BLW} and \cite{BG}.

\subsection{The structure of this paper} The paper is organized as follows. In
\S\ref{sec 1} we  introduce some preliminary material.
	\S \ref{sec 4} is devoted to the proof of the super Vust theorem (Theorem
\ref{thm Vust0}). In \S\ref{sec 3}, most of the time we
specially take $e$ to be principal nilpotent. In order to prove the first equality
in Theorem \ref{doublecentralizerin}, we need to compute the centralizer of
$\mathbb{C}_{n}[x_{1},\ldots,x_{d} ]\rrtimes \mathbb{C}{\fraksd}$ in
$\End_\bbc(V^{\otimes d})$. For this, we will work with the distribution space of
$M_e$ ($M_e$ is the ground space of $\ggg_e$), and identify $U(\ggg_e)$ with
$\Dist(M_e)$. By analysing the comodule structure of $\bbc[M_e]$ on $V^{\otimes
d}$, we finally verify the first equality of Theorem \ref{doublecentralizerin}.
In \S\ref{sec 5}, we show that $U(\ggg_e)$ is actually a contraction of the
corresponding finite $W$-superalgebra.
In \S\ref{sec: skryabin}, with the aid of Skryabin's equivalence we will analyse the  representations of $W_\chi$ on $V_\bc^{\otimes d}$.
In the concluding section \S\ref{6}, we will introduce
the degenerate affine Hecke algebras, introduce the filtered deformation of
the double centralizer property, and finally accomplish the proof of Theorem
\ref{thm:dcthm}.

\subsection{Acknowledgments}   The authors thank Husileng Xiao for his valuable comments. They express their thanks to  anonymous referees for helpful comments which enable the authors to improve the manuscripts.

\section{Preliminaries}\label{sec 1}
Throughout the paper, we always assume  $V=\bbc^{m|n}$ unless otherwise specified,
which is a vector superspace with $V=V_\bz\oplus V_\bo$, and the even part and
odd part are isomorphic to $\bbc^m$ and $\bbc^n$ respectively. The general linear Lie
supergroup $\GL(V)$ and its Lie superalgebra $\gl(V)$ are simply written as
$\GL(m|n)$ and $\gl(m|n)$ respectively.  We will list some basic material on
them, for which the reader can refer to \cite{CCF},  \cite{Man} or
\cite[\S6]{M1}. 
Note that $\gl(m|n)\cong \gl(n|m)$ (see \cite[Remark 1.6]{CW}).  Throughout the
paper, we always assume that $m\leq n$ without any loss of generality.

Generally, for a $\bbc$-vector superspace $M=M_\bz\oplus M_\bo$, we denote by
$|m|\in\bbz_2$ the parity of a $\bbz_2$-homogeneous vector $m\in M_{|m|}$.
Throughout the paper, the terminology of ideas, subalgebras, modules, etc., of a
Lie superalgebra instead of superideals, subsuperalgebras, supermodules, etc., are
adopted.

\subsection{General linear Lie superalgebras and their (even) nilpotent
elements}\label{gl and nilpotents}
For $V=\bbc^{m|n}$, we have the $\bbz_2$-grading $V=V_\bz\oplus V_\bo$ with the
even part $V_\bz\cong \bbc^m$ and the odd part $V_\bo\cong \bbc^n$.  Choose
ordered bases for $V_\bz$ and $V_\bo$ that combine to a homogeneous ordered basis
of $V$. We will make it a convention to parameterize such a basis by the set
$I(m|n)=\{\bar 1,\ldots, \overline{m}; 1,\ldots,n\}$ with total order $\{\bar
1<\cdots<\overline{m}<1<\cdots<n\}$ (where $I(m|n)$ will be simply written as $I$
afterwards). The elementary matrices are accordingly denoted by
$e_{i,j}$ corresponding to $(i,j)$-entry ($i,j\in I(m|n)$), which stands for  the
matrix of all zeros in any entry with exception $(i,j)$-entry equal to
$1$. With respect to such an ordered basis, $\gl(V)$ can be  realized as $\gl(m|n)$ consisting of
$(m+n)\times (m+n)$ complex matrices of the block form
 $e=\begin{pmatrix} \bfa&\bfb\\ \bfc&\bfd\end{pmatrix}$, where $\bfa,\bfb,\bfc$ and $\bfd$ are
 respectively $m\times m$, $m\times n$, $n\times m$, and $n\times n$ matrices.
 Furthermore, $\gl(V)=\gl(V)_\bz\oplus \gl(V)_\bo$ with the even part
 $\gl(V)_\bz=\begin{pmatrix} \bfa&\bf0\\ \bf0&\bfd\end{pmatrix}$, and the odd part
 $\gl(V)_\bo=\begin{pmatrix} \bf0&\bfb\\ \bfc&\bf0\end{pmatrix}$. In particular,
 $\gl(V)_\bz\cong \gl(m)\oplus \gl(n)$.

The general linear Lie supergroup $\GL(m|n)$ can be defined, via the setting-up of affine supergroup schemes,  which is  a functor sending a commutative $\bbc$-superalgebra $A$ to the group
$\GL(m|n;A)$ of all invertible $(m+n)\times (m+n)$ matrices of the form
\begin{align}\label{super gp sch}
g=\begin{pmatrix} \bfa&\bfb\\ \bfc&\bfd\end{pmatrix},
 \end{align}
   where $\bfa$ and $\bfd$ are $m\times m$ and $n\times n$ matrices,
   respectively, with entries in $A_\bz$; $\bfb$ and $\bfc$ are $m\times n$ and
   $n\times m$ matrices, respectively, with entries in $A_\bo$. It is well-known
   (cf. for example, \cite{FioGav} or \cite{Man}) that $g$ is invertible if and
   only if both $\bfa$ and $\bfd$ are invertible. The Lie superalgebra of
   $\GL(m|n)$  is just  $\gl(m|n)$ (cf. \cite{FioGav} or \cite{Man}).

Let $\Mat(m|n)$ be the affine superscheme with $\Mat(m|n;A)$ consisting of all
$(m+n)\times (m+n)$ matrices of the form (\ref{super gp sch}). Then the affine
coordinate superalgebra $\bbc[\GL(m|n)]$ is the localization of $\bbc[\Mat(m|n)]$
at  the function  $\det: g\mapsto \det(\bfa)\det(\bfd)$. A closed subgroup of
$\GL(m|n)$ is an affine supergroup scheme with coordinate superalgebra being a
quotient of $\bbc[\GL(m|n)]$  by a Hopf ideal $J$. In particular, the underling
purely even group of $\GL(m|n)$, denoted by $\GL(m|n)_\ev$, corresponds to the
Hopf ideal $\bbc[\GL(m|n)]\bbc[\GL(m|n)]_\bo$. That is to say,  $\GL(m|n)_\ev$ is
completely determined by its coordinate superalgebra $\bbc[\GL(m|n)]\slash
\bbc[\GL(m|n)]\bbc[\GL(m|n)]_\bo$. It is well-known that $\GL(m|n)_\ev\cong
\GL(m)\times \GL(n)$ (cf. \cite{FioGav} or \cite{Man}).

\subsubsection{Nilpotent orbits}\label{nilp orb}
From now on, we set $\ggg=\gl(m|n)$ and denote  by $G$  the corresponding Lie supergroup $\GL(m|n)$. For an (even) nilpotent element $e\in \ggg_\bz$, and any
of its conjugate element $e'=\ad(g)e$ under the adjoint action of $g\in G_\ev$, we have $\ggg_e\cong \ggg_{e'}$.
Both of them share the same double centralizer property in $\End_{\bbc}(\Votd)$.
So we can choose nilpotent matrices of Jordan standard forms, as typical
representatives in their $G_\ev$-adjoint orbits for the study of the double
centralizer property we are concerned.

Recall that all nilpotent orbits under $G_\ev$ in $\ggg_\bz$ are parameterized by
all partition $(\lambda,\mu)$ of $(m|n)$, which means that both sequences
$\lambda$ and $\mu$ of positive integers with
$\lambda=(\lambda_1,\lambda_2,\ldots,\lambda_s)$ for a certain $s$, and
$\mu=(\mu_1,\mu_2,\ldots,\mu_t)$ for a certain $t$ satisfy the condition that
$\lambda_1\geq \lambda_2\geq \cdots \geq\lambda_s >0$, and
$\mu_1\geq\mu_2\geq\cdots \geq\mu_t>0$ along with
$|\lambda|:=\sum_{i}\lambda_i=m$ and $|\mu|:=\sum_{j}\mu_j=n$.
(The order of a partition may be defined in a reversed one, for which the components go
increasing.)
\subsubsection{Principal  nilpotent elements}
\begin{definition}\label{def: principal nilpotent}
 A nilpotent element	 $e\in \ggg_\bz$ is said to be principal or regular (see \cite{Jan})  if
the adjoint orbit $G_{\ev}. e$ has the largest  dimension among all
$G_{\ev}$-orbits of the nilpotent cone which consists of all nilpotent elements in $\ggg_\bz$.
\end{definition}
Recall that all principal  nilpotent elements in $\ggg$ are  in the same
$G_\ev$-orbit (see, e.g. the remark below \cite[Theorem 4.1.6]{CM}). We choose a standard representative $e$ of principal nilpotent
elements as below (which is called a canonical principal nilpotent element):
$$e = \sum_{i=1}^{m-1}e_{\bar{i},\overline{i+1}}+\sum_{j=1}^{n-1}e_{j ,j+1}.$$

\subsection{Schur-Sergeev duality}
 For a natural $\ggg$-module $V=\mathbb{C}^{m|n}$,  the tensor product module
 $V^{\otimes d}$ over $U(\ggg)$ (represented by $\phi_{d}$)  is  defined as in
 (\ref{eq: superaction0}).
The symmetric group  ${\fraksd}$  acting on  $V^{\otimes d}$ is defined as in
(\ref{super symaction}).

\begin{theorem} \label{ss dual}(the Schur-Sergeev duality,  see \cite{Ser} or \cite{CW}) The
images of $\phi_d$ and $\psi_d$ satisfy the double centralizer property, i.e.,
\begin{align*}
	\phi_{d}(U(\ggg)) &= \text{End}_{\fraksd}(V^{\otimes d}),\\
	\text{End}_{U(\ggg)}(V^{\otimes d})&= \psi_{d}({\fraksd}).
\end{align*}
\end{theorem}

\subsection{Degenerate affine Hecke algebras in the super case} Now we introduce the so-called degenerate affine  Hecke  algebra (dAHA for short) in the super case,  which are actually the same as  defined in  \cite{BKl}. However, the action on the tensor product of superspaces is different from that one from \cite{BKl}, where the former  is introduced by  Sergeev. Below we will make use the notation $\SHl$ for such an algebra, instead of the usual one.

\begin{definition}\label{def: dasha}
The degenerate affine Hecke algebra $\SHl$  is an associative algebra with generators $x_i$ $(i
 =1,\ldots, \ell)$   and  $s_i$ $(i=1,\ldots, \ell-1)$, subjected to relations:
\begin{eqnarray*}
s_i^2 =1,\quad s_i s_{i+1} s_i &=& s_{i+1} s_i s_{i+1}, \nonumber\\
s_i s_j &=& s_j s_i,\quad |i-j|>1, \nonumber \\
x_j s_i &=& s_i x_j, \quad (j \neq i,i+1), \nonumber \\
x_{i+1} s_i -s_i x_i &=& 1, \label{Hecke} \\
x_i x_j &=& x_j x_i, \quad (i \neq j). \nonumber
\end{eqnarray*}
\end{definition}

\vskip10pt
\section{Super Vust theorem}\label{sec 4}
Keep the notation $V=\bbc^{m|n}$. For   $A\in
\text{End}_\bbc(V)$ whose centralizer subgroup in $G=\GL(m| n)$ is denoted
by $G_{A}$,  the classical Vust  theorem says as introduced in \S\ref{se:1}, that
on $\Votd$, the action of $G_A$ and the twisted action of symmetric group
$\fraksd$ presented as the algebra $\bbc_l[x_{1},\ldots,x_{d}]\rrtimes
\bbc\fraksd$  are double-centralizers in $\End_\bbc(V^{\otimes d})$ (see
\cite{Vust} or \cite{KP}).

In this section, we will prove the super Vust theorem (Theorem \ref{thm Vust0}),
taking the same strategy as in the classical case (see \cite[\S6]{KP}).

\subsection{Preparation for the proof of Theorem \ref{thm Vust0}}

\subsubsection{Adjoint orbits}
 Generally, for an affine algebraic supergroup $G$, the adjoint action $\Ad$ of $G$ on $\ggg:=\text{Lie}(G)$ can be naturally defined (for example, see  \cite[\S11.5]{CCF} or \cite[\S6]{M1}  ). Denote by $\Ad$ as usual the adjoint action. Then $\Ad$ is defined as the natural
transformation $\Ad: G\rightarrow \GL(\ggg)$.

In particular, take $G:=\GL(m|n)$ and $\ggg=\gl(m|n)$ over $\bbc$. Let $e$ be any given nilpotent element with $e\in\ggg_{\bz}$ (keep in mind that all the nilpotent elements considered in the paper are supposed to be even). Under the action of $\Ad$, $G.e$ can be regarded a super scheme of $\ggg$ defined as
\begin{align}\label{eq: orbit map}
 G.e(R)=\{\Ad(g)(e)=geg^{-1}\mid g\in G(R) \}
\end{align}
for any commutative $\bbc$-superalgebra $R$.

\subsubsection{}\label{sec: general supersch}
     Let us roughly recall some basic material on superschemes.  The reader may refer to \cite{CCF} or \cite{Man} for details.

     A superspace $S$ is a super ringed space $(\shortmid S\shortmid, \co_S)$, i.e., a
topological space $\shortmid S\shortmid$ endowed with a sheaf of commutative super rings $\co_S$ (the
structure sheaf of $S$),  with the property that the
stalk $\co_{S;x}$ is a local super ring for all $x\in \shortmid S\shortmid$.
 A superscheme can be equivalently defined as a superspace $X=(\shortmid X \shortmid,\co_X)$ such that $(\shortmid X\shortmid,\co_{\shortmid X\shortmid}:=(\co_X)_\bz)$ is an ordinary  scheme  and $(\co_{X})_\bo$ is a quasi-coherent sheaf of $(\co_{X})_\bz$ (see \cite[\S3.3]{CCF}).
 More precisely,  the sheaf of superalgebras $\co_X$ is a sheaf of $(\co_X)_\bz$-modules, where $(\co_X)_\bz$ is the sheaf
over $\shortmid X\shortmid$ defined as $\co_X(U)_\bz$ for all $U$ open in $\shortmid X\shortmid$. The
sheaf $(\co_X)_\bo$ defined as $(\co_X)_\bo(U)=\co_X(U)_\bo$, is also a sheaf of $(\co_X)_\bz$-modules.
 Set $\bbc[X]:=\co_{X}(\shortmid X\shortmid)$.

  Recall that a superscheme $Y$ is a closed subscheme of a superscheme  $X$
       if $\shortmid Y\shortmid$ is a closed subset of $\shortmid X \shortmid$ and  $\co_Y=\co_X\slash \mathcal{I}$ where $\mathcal{I}$ is a quasi-coherent sheaf of ideals in $\co_X$.

   Now let us turn to nilpotent orbits of supergroups.
  A general setting-up for  adjoint orbits can be referred to   \cite[\S6.2]{M1}.
The orbit closure  $\overline{G.e}$   can be defined as below.
The orbit map $g\mapsto g.e:=\Ad(g)e$ as in (\ref{eq: orbit map}) gives rise to a natural transformation of functors $\mu:G\rightarrow G.e\subseteq\ggg$. By Yoneda's lemma (see \cite[Lemma 3.4.3]{CCF}), $\mu$ induces an algebra homomorphism $\mu^*:\bbc[G.e]\rightarrow \bbc[G]$, and $\overline{G.e}$ is the closed subscheme of $\ggg$ defined
by the ideal $\ker \mu^*$.
By definition, there is an embedding $$\bbc[\overline{G.e}]\hookrightarrow \bbc[G.e].$$

  Note that $\shortmid G\shortmid= G_\ev\cong\GL(m)\times \GL(n)$ is a connected algebraic group. So both $G_\ev.e$ and $\overline{G_\ev.e}$ are irreducible varieties in $\ggg_\bz$. Denote by $\text{Frac}(\bbc[G_\ev.e])$ (resp. $\text{Frac}(\bbc[\overline{G_\ev.e}])$) the fractional field of $\bbc[G_\ev.e]$ (resp. $\bbc[\overline{G_\ev.e}]$). Then we have $\text{Frac}(\bbc[G_\ev.e])=\text{Frac}(\bbc[\overline{G_\ev.e}])$ since $G_\ev.e$ is an open subset in $\overline{G_\ev.e}$ (see \cite{CM} or \cite{Jan}).

   Now set $Y=G.e$ and $Z=\overline{G.e}$.  Thanks to \cite[Theorem 6.8]{M1}, $\dim \shortmid Z\shortmid=\dim \shortmid Y\shortmid=\dim\shortmid G.e\shortmid=\dim G_\ev.e$. By the above arguments,  we have the following consequences:
  \begin{align}\label{eq: closure frac}
  \text{Frac}(\bbc[\shortmid Z\shortmid])=\text{Frac}(\bbc[\overline{G_\ev.e}])=\text{Frac}(\bbc[G_\ev.e])=\text{Frac}(\bbc[\shortmid Y\shortmid]).
  \end{align}

 \subsubsection{Fractional algebras} We introduce the fractional superalgebras of $\bbc[Y]$ and of $\bbc[Z]$ for $Y=G.e$ and $Z=\overline{G.e}$ respectively.
    { By the arguments in \S\ref{sec: general supersch}, we can set}
\begin{align}\label{eq: frac}
&\text{Frac}(\bbc[Y]):=\text{Frac}(\bbc[\shortmid Y\shortmid])\otimes_{\bbc[ Y]_\bz}\bbc[Y];\cr
&\text{Frac}(\bbc[Z]):=\text{Frac}(\bbc[\shortmid Z\shortmid])\otimes_{\bbc[ Z]_\bz}\bbc[Z].
\end{align}

\begin{lemma}\label{lemma3}
	Let  $G= \GL(m|n)$ and $\ggg=\gl(m|n)$, and let $e\in \ggg_{\bz}$  be a nilpotent element. Then the following statements hold.
 \begin{itemize}
 \item[(1)] $\text{Frac}(\bbc[\overline{G.e}])=\text{Frac}(\bbc[{G.e}])$.
\item[(2)] $\mathbb{C}[\overline{G.e}]=\mathbb{C}[G.e]$.
\end{itemize}
\end{lemma}

\begin{proof} (1) Note that the superscheme $G.e$ has super dimension
$$\underline{\dim} G.e=(\dim G_\ev.e, \dim\ggg_\bo-\dim(\ggg_\bo)_e)$$ (see  \cite[Theorem 6.8]{M1}), and the superscheme $\overline{G.e}$ has the same dimension as the previous one. Keep the notations as in \eqref{eq: frac}.
 Note that $\bbc[Z]$ is a subalgebra of $\bbc[Y]$. By definition, $\bbc[Z]_\bz=\bbc[\shortmid Z\shortmid]$, $\bbc[Y]_\bz=\bbc[\shortmid Y\shortmid]$.
 Moreover, both $\bbc[Z]_\bo$ and $\bbc[Y]_\bo$ are finite-dimensional, with the same dimension. Thus $\bbc[Z]_\bo=\bbc[Y]_\bo$. Now the statement follows from \eqref{eq: closure frac}.

  (2) According to the nilpotent orbit theory of general linear Lie algebras, the
closure of the nilpotent $G_{\ev}.e$  is normal (see \cite{Jan} or \cite{KP}). So
$\mathbb{C}[\overline{G_\ev.e}]=\mathbb{C}[G_\ev.e]$.
By definition, $\shortmid \overline{G.e}\shortmid$ is an irreducible closed subset in $\ggg_\bz$.
    Now the superscheme $\overline{G.e}$ has super dimension $(\dim G_\ev.e, \dim\ggg_\bo-\dim(\ggg_\bo)_e)$, containing  a closed subset $\overline{G_\ev.e}$.  It is deduced that $\shortmid\overline{G.e}\shortmid=\shortmid\overline{G_\ev.e}\shortmid$.
So, $\bbc[Z]_\bz=\mathbb{C}[\overline{G_\ev.e}]=\mathbb{C}[G_\ev.e]=\bbc[Y]_\bz$.
It is already known above that $\bbc[\overline{G.e}]_\bo=\bbc[{G.e}]_\bo$. Consequently, $\mathbb{C}[G.e] =
\mathbb{C}[\overline{G.e}]$.
\end{proof}

\subsubsection{} For a given $e\in\ggg_\bz$, consider the centralizer supergroup $G_e$ of $e$ in $G$, as a closed
subgroup of $G$ (see \cite[\S11.8]{CCF}, or see \cite[\S{I}.2.6]{Jan0} as an analogue of group schemes). The following fact
is clear, as an analogue of the classical result (see \cite[Proposition 10.6]{Hum}).
\begin{lemma}\label{lie of centralizer} Keep the notations as in Lemma
\ref{lemma3}. Then $\text{Lie}(G_e)=\ggg_e$.
\end{lemma}

\subsection{}\label{sec Vustproof}
    From now on till  \S\ref{end of pf th 4.1},  we will prove Theorem  \ref{thm
    Vust0}.  First of all, by  a direct computation it is easy to verify that
\begin{align}\label{lefthandside}
\text{End}_{\rho({\ggg_{e}})}(V^{\otimes d}) \supseteq\mathscr{A}.
\end{align}
We only need to prove the other direction.
    For the simplicity of the subsequent arguments, we will adopt some technique of formulations of supergroups which could be more familiar to some physicists  (see \cite{GZ}).

   \subsubsection{}
     Associated with an $N$-dimensional vector space over $\bbc$, one has the
  Grassmannian algebra which is only dependent on the degree $N$ up to
  isomorphism, denoted by $\bigwedge (N)$.
 Set $\bigwedge = {\varinjlim} \bigwedge (N)$, which becomes an
 infinite-dimensional commutative superalgebra over $\mathbb{C}$.

Set $\bfv: = V\otimes_\bbc \bigwedge$, $\bfv^{*} =
\text{Hom}_{\bigwedge}(\bfV, \bigwedge)$.
The general linear Lie supergroup on $\bfv$ can be defined as:
$$\GL(\bfv) \cong \{g\in\text{End}_{\bigwedge}(\bfv)_\bz\mid  g \text{ is invertible}
\}.$$
The algebraic Lie supergroup $\GL(m|n)$ will be identified with  $\bfG:=\GL(\bfv)$ in our arguments.

    Let $\widetilde{\gl(\bfv)}=\gl(V)\otimes_\bbc\bigwedge$ which is regarded as a
Lie superalgebra over $\bigwedge$ in the sense that there is a
$\bigwedge$-bilinear Lie bracket  on $\widetilde{\gl(\bfv)}$,  via $[X\otimes \lambda_1,
Y\otimes \lambda_2]=[X,Y]\otimes (-1)^{|\lambda_1||Y|}\lambda_1\lambda_2$ for $\bbz_2$-homogeneous $X,Y\in
\gl(\bfv)$ and $\lambda_1, \lambda_2\in \bigwedge$. Then
    $\gl(\bfv)=(\gl(V)\otimes_\bbc\bigwedge)_\bz$ forms a Lie subalgebra of
$\widetilde{\gl(\bfv)}$ over $\bigwedge_\bz$, which can be referred to as the Lie
algebra of $\GL(\bfv)$. The action of $\gl(\bfv)$ on $\bfv$ is exactly the restriction of
the one of $\widetilde{\gl(\bfv)}$ on $\bfv$, the latter of which is clear by its
definition.

    In this new formulation, we have the following basic fact.
    \begin{lemma}(\cite[Lemma 2.6]{LZ})\label{2.8.2}
The following hold:
\begin{itemize}
	\item[(1)] For any $X \in \mathfrak{gl}(\bfv)$, let  $\text{exp}(X): =
\sum_{r=0}^{\infty}\frac{X^ {r}}{r!}$. Then  $\text{exp}(X)$  is  well-defined as
an automorphism of $\bfv$ which lies in  $\GL(\bfv)$. Hence, there is a mapping
	$$\text{exp}:\mathfrak{gl}(\bfv)\rightarrow \GL(\bfv), X\longmapsto
\text{exp}(X).$$
	
	\item[(2)] The image of $\mathfrak{gl}(\bfv)$ under the map $\text{exp}$ is  $\GL(\bfv)$.
\end{itemize}
\end{lemma}

 \subsubsection{} In the next subsubsections, we investigate  $\mathscr{B}:=\text{End}_{\rho({G_{e}})}(\bfv^{\otimes d})$
(for simplicity we do not distinguish the notations of $\gl(\bfv)$-action and
$\GL(\bfv)$-action on $\bfvotd$),  which is exactly $\text{End}_{\rho
({\ggg_{e}})}(V^{\otimes d})\otimes \bigwedge$ by Lemma \ref{lie of
centralizer}. Note that $\mathscr{B}$ contains the subalgebra
$\text{End}_{\bigwedge}(\bfvotd)_\bz^{\bfG}$, the latter of which by
Schur-Sergeev duality,  is just the image {$\psi_d({\fraksd})$} in
$\text{End}_{\bigwedge}(\bfvotd)_\bz$ by base change. Then $\mathscr{B}$
contains {$\psi_d({\fraksd})$} and $e$, thereby contains the subalgebra generated by them. We need to decide the exact generators of   $\mathscr{B}$. For this, we first recall a general fact on the group action. Suppose $H$ is a group, and $X_1, X_2$ are $H$-sets. If the two points $x_1\in X_1$ and $x_2\in X_2$ satisfy the condition concerning the stabilizer subgroups $\text{Stab}_H(x_1)\subseteq \text{Stab}_H(x_2)$,
then one can define an $H$-equivariant map from $H.x_1\subseteq X_1$ to $H.x_2\subseteq X_2$ which sends $x_1$ onto $x_2$.

 We first have the following basic observation.

\begin{lemma}\label{lem: general extension of GP aciton} Suppose $H$ is a connected reductive algebraic group over $\bbc$, $e\in \text{Lie}(H)$ is a principal nilpotent element. Suppose $(\rho,M)$ is a finite-dimensional  $H$-module with $b\in \End_\bbc(M)$ satisfying $\text{Stab}_H(e)\subseteq \text{Stab}_H(b)$. Then there exists an $H$-equivariant morphism $\varphi$ from $\text{Lie}(H)$ to $\End_\bbc(M)$ such that $\varphi(e)=b$.
\end{lemma}
\begin{proof} Write $\calh=\text{Lie}(H)$. By the arguments above the lemma, there is an $H$-equivariant morphism $\varphi$ from $H.e\subseteq \calh$ to $H.b\subseteq \End_\bbc(M)$.
    Note that $e$ is a principal nilpotent, thereby $H.e$ is an open dense set in the nilpotent cone $\caln(\calh)$ consisting all nilpotent elements in $\calh$. We proceed with arguments by steps.

    (i) We first claim that $\varphi$ can be extended to be an $H$-equivariant morphism from $\caln(\calh)$ to $\End_\bbc(M)$.

In general, for $x\in\calh$ and $\co:=H.x$,  the orbit closure $\overline\co$ is $H$-stable
and we have an $H$-module structure  on $\bbc[\overline\co]$. Note that $\co$ is dense in $\overline\co$, and the restriction of functions
from $\overline\co$ to $\co$ is  $H$-equivariant and injective.
In particular, for principal nilpotent $x=e$ we have that $\overline\co=\caln(\calh)$ (see \cite[\S6.3]{Jan}) and $\overline\co$ is a normal variety (\cite[\S8.5]{Jan}). Furthermore, $\bbc[\overline\co]=\bbc[\co]$ (\cite[\S8.3]{Jan}). Hence $\bbc[\co]^H\otimes_\bbc \End_\bbc(M)=\bbc[\overline\co]^H\otimes_\bbc \End_\bbc(M)$.  On the other hand, if we denote by $\textsf{Mor}_H(\co,\End_\bbc(M))$ the space of all $H$-equivariant morphism from $\co$ to $\End_\bbc(M)$, then $$\textsf{Mor}_H(\co,\End_\bbc(M))=(\bbc[\co]\otimes_\bbc \End_\bbc(M))^H$$
(see the forthcoming Lemma \ref{lem: forthcoming}). The same thing happens for $\caln(\calh)=\overline\co$.
This shows that $\varphi$ becomes an $H$-equivariant morphism from $\caln(\calh)$ to $\End_\bbc(M)$.  Hence the claim is proved.

    (ii) On the other hand, $\caln(\calh)$ is an $H$-stable closed subvariety of the affine space $\calh$. Note that $\calh$ is an $H$-variety, and $\End_\bbc(M)$ is an $H$-module. Thanks to \cite[Lemma 6.2.1]{KP}, there exists an $H$-equivariant morphism from $\calh$ to $\End_\bbc(M)$  extending  the $H$-equivariant morphism $\varphi: \caln(\calh)\rightarrow \End_\bbc(M)$. Such an extended $H$-equivariant morphism is the desired one.

    The proof is completed.
\end{proof}
\begin{lemma}\label{lem: forthcoming} Keep the notations as in the proof of Lemma \ref{lem: general extension of GP aciton}. We have
\begin{align}
\textsf{Mor}_H(\co,\End_\bbc(M))=(\bbc[\co]\otimes_\bbc \End_\bbc(M))^H,\cr
\textsf{Mor}_H(\overline\co,\End_\bbc(M))=(\bbc[\overline\co]\otimes_\bbc \End_\bbc(M))^H.
\end{align}
\end{lemma}

\begin{proof} These facts follow from  the definition of $H$-equivariant morphisms.
\end{proof}

\subsubsection{} Now we turn back to the case of the action of  supergroups.

Recall that the category of supergroups is equivalent to the category of super Harish-Chandra pairs (see \cite[\S7.4]{CCF}). For example, $(G_\ev\cong\GL(m)\times \GL(n), \gl(m|n))$ is the super Harish-Chandra pair corresponding to $\GL(m|n)$. Generally, for a super Harish-Chandra pair $(\frakG_\ev,\calg)$ associated with an algebraic supergroup $\frakG$ and  $\calg=\Lie(\frakG)$, one can define a module structure on a vector superspace $M$ over $(\frakG_\ev,\calg)$, which means that $M$ is endowed with two structures of the rational $\frakG$-module and of the $\calg$-module, and that both are compatible. This is to say, for any $m\in M$, $g\in\frakG_\bz$, and $X\in\calg$, we have $g. (X.m)=\Ad g(X).m$.

\begin{lemma}\label{lemma4-(1)}
Let $e\in \ggg_{\bz}$  be a principal nilpotent element, and $\bfG=\GL(m|n)$.
For any  $b\in \text{End}_{\rho({\ggg_{e}})}(V^{\otimes d})\subseteq \mathscr{B}= \text{End}_{\rho({\ggg_{e}})}(\bfvotd)$, there is a  $\bfG$-equivariant morphism
	$$\Phi: \text{End}_{\bigwedge}(\bfV)\rightarrow
\text{End}_{\bigwedge}(\bfV^{\otimes d}) = \text{End}_{\bigwedge} (\bfV)^{\otimes d}$$
	such that  $\Phi(e) = b$.
\end{lemma}
\begin{proof}  Note that $\bfG$ acts on $\text{End}_{\bigwedge}(\bfV)$ and $\text{End}_{\bigwedge}(\bfV^{\otimes d})$ by conjugation respectively (as a realization given in Lemma \ref{2.8.2}).
 From the assumption $b\in (\text{End}_{\bigwedge}\bfvotd)^{\bfG_e}$, it follows that
 \begin{align}\label{eq: stab rel}
 \text{Stab}_G(e)\subseteq \text{Stab}_G(b).
  \end{align}
  Now we proceed with the arguments in steps.

 (i) Note that the purely-even subgroup of $\bfG=\GL(m|n)$ is isomorphic to $H:=\GL(m)\times \GL(n)$ which is reductive,  identified with  $\shortmid \bfG\shortmid$. Still set $\calh:=\Lie(H)$. Then $\calh=\gl(m)\oplus\gl(n)$. For $e\in \calh$, from (\ref{eq: stab rel}) it follows that   $\text{Stab}_H(e)\subseteq\text{Stab}_H(b)$.
  Consider $M=V^{\otimes d}$, then $\bfvotd=M\otimes_\bbc \bigwedge$. Note that $b\in \text{End}_{\bbc}(M)$, 
  and $\text{End}_{\bbc}(\bfvotd)$ is a base change of $\text{End}_{\bbc}(M)$.  By Lemma \ref{lem: general extension of GP aciton}, there exists an $H$-equivariant morphism $\Phi$ from $\calh$ to $\text{End}_{\bbc}(M)$ such that $\Phi(e)=b$. Since $\text{End}_{\bigwedge}(\bfV)=\text{End}_\bbc(V)\otimes \bigwedge$, then  this $\Phi$ can be regarded, by base change,  as an $H$-equivariant morphism from $\text{End}_{\bigwedge}(\bfV)$ to $\text{End}_{\bigwedge}(\bfV^{\otimes d}) = \text{End}_{\bigwedge} (\bfV)^{\otimes d}$.

(ii) We claim that the above $\Phi$ is a morphism of the super Harish-Chandra pair $(H,\gl(m|n))$. This is to say, $\Phi$ is a homomorphism of  $\gl(m|n)$-modules, compatible with the $H$-action. This claim follows from the facts that both $\text{End}_{\bigwedge}(\bfV)$ and $\text{End}_{\bigwedge}(\bfV^{\otimes d})$  are natural $\gl(m|n)$-modules and $H$-modules, respectively, and that $H$ has conjugation action on $\gl(m|n)$ compatible with those module structures.

 (iii) We claim that the above $\Phi$ can be extended to a $\bfG$-equivariant morphism from $ \text{End}_{\bigwedge}(\bfV)$ to $\text{End}_{\bigwedge} (\bfV)^{\otimes d}$. Note that the distribution algebra $\Dist(\bfG)$ of $\bfG$ is isomorphic to $U(\ggg)$ for $\ggg=\gl(m|n)$ (see \cite[\S5.1]{MS}, or \cite[Chapter I.7]{Jan0}), and  $\Phi$ is already a module homomorphism for the super Harish-Chandra pair $(H,\ggg)$. Hence it is a homomorphism of $(\Dist(\bfG),T)$-modules for the standard maximal torus $T$ of $\GL(m)\times\GL(n)$ (see \cite[\S2]{SW}, or \cite[Chapter II.1.20]{Jan0}). Note that the category of $G$-modules is equivalent to the category of locally-finite $(\Dist(\bfG),T)$-modules (see \cite[Theorem 2.8]{SW}). Hence, $\Phi$ can be extended to a $\bfG$-equivariant morphism.


 Summing up, we accomplish the proof.
 \end{proof}

\begin{remark} We can generally consider an algebraic supergroup $\frakG$ over $\bbc$ and the corresponding super Harish-Chandra pair $(\frakG_\ev,\calg)$ with $\calg=\Lie(\frakG)$. Suppose $M$ and $N$ are modules over $(\frakG_\ev,\calg)$,  then both $M$ and $N$ are $\frakG$-modules (see for example,  \cite[Proposition 8.3.3]{CCF}).
%
We give a general claim of (iii) in the above proof.

\begin{claim}
If $\Phi:M\rightarrow N$ is a morphism of the super Harish-Chandra pair, then $\Phi$ is  a $\frakG$-equivariant morphism.
\end{claim}
%

The arguments for this claim can be done in the same spirit of the arguments as in \cite[Proposition 8.3.3]{CCF} which deals with the equivalence of the action of super Harish-Chandra pairs and of supergroups. Actually, $\Phi$ is already a homomorphism of $\calg$-modules. It suffices to show that $\Phi$ is a homomorphism of comodules over $\bbc[\frakG]$. Making use of some arguments with  Hopf algebra techniques (see for example \cite[Proposition 18(2)]{Ma}), one can carry it out.
\end{remark}

\subsection{} Note that there is an isomorphism
	\begin{align*}
	\pi :((\bfV^*)^{\otimes d}\otimes \bfV^{\otimes d})^{*}\cong
\text{End}_{\bigwedge}(\bfV^{\otimes d})=(\End_{\bigwedge}(\bfV))^{\otimes d}
	\end{align*}
	determined  by the non-degenerate bilinear form defined as follows: for
$\Omega\in (\End_{\bigwedge}\bfV^{\otimes d})$,  $v_1\otimes v_2\otimes\cdots \otimes v_d\in
\bfV^{\otimes d}$ and $f_1\otimes f_2\otimes\cdots\otimes f_d\in (\bfV^*)^{\otimes d}$, we set
	\begin{align*}
	&\langle\Omega,\; f_1\otimes f_2\otimes\cdots\otimes f_d\otimes
v_{1}\otimes\cdots \otimes v_{d}\rangle\\
	=&\langle f_1\otimes f_2\otimes\cdots\otimes f_d, \;
\Omega(v_{1}\otimes\cdots \otimes v_{d})\rangle.
	\end{align*}
Here and further, we use the notation $\langle -,- \rangle$ to stand for the valuation of the specific  pair of linear duals.

 Now we {have the following results, which will be used in the  proof of the forthcoming Lemma \ref{lemma4-(2)}.} 
\begin{lemma}\label{3.3.1}     Set $\bfG=\GL(\bfV)$. The following statements hold.
\begin{itemize}
	\item[(1)] The $\bfG$-invariants of $\bigwedge$-function ring in
$\text{End}_{\bigwedge}(\bfV)^{q}$  are generated by the supertrace functions at
$(X_1,\ldots,X_q)\in \text{End}_{\bigwedge}(\bfV)^q$ (the space of $q$-tuples of elements of $\text{End}_{\bigwedge}(\bfV)$):
	$$\str(X_{i_{1}}X_{i_{2}}\cdots X_{i_{j}}),\;
i_1,i_2,\ldots,i_j\in\{1,2,\ldots,q\}.$$

	\item[(2)] Furthermore, the $\bfG$-invariants in $\text{End}_{\bigwedge}(\bfV)^{\otimes
d}$
are  $\bigwedge$-linear combinations of the following basic invariants:
	$$\theta_\sigma: h_1\otimes h_2\otimes\cdots\otimes h_d\otimes
v_{1}\otimes\cdots \otimes v_{d}\mapsto (-1)^{J (h,v, \sigma)}\prod_{j}\langle
h_{\sigma(j)}, v_{j}\rangle,$$
	where $\sigma\in {\fraksd}$, ${J(h,v, \sigma)}:=\sum_{i=1}^d t_i$ with
$t_i=|v_i|\sum_{j=i+1}^d |h_{\sigma(j)}|$.
	
	\item[(3)] Under the identification between $(\bfV^{*})^{\otimes d}\otimes \bfV^{\otimes d}$ and
$\text{End}_{\bigwedge}(\bfV)^{\otimes d}$, we have
$$\theta_\sigma(X_{1}\otimes\cdots \otimes X_{d}) = \str(X_{i_{1}}\cdots
X_{i_{r}})\str(X_{j_{1}}\cdots X_{j_{s}})\cdots\str(X_{k_{1}}\cdots X_{k_{t}}),$$
where  $X_{1}, X_{2}, \ldots,X_{d}\in \text{End} _{\bigwedge}(\bfV)$, and
$\sigma = (i_{1}i_{2}\cdots i_{r})(j_{1}j_{2}\cdots
j_{s })\cdots(k_{1}k_{2}\cdots k_{t})\in {\fraksd}$ is an expression in product
of cycles not intersecting mutually.	
\end{itemize}
\end{lemma}

\begin{proof} For Parts (1) and (3), the proof is the same as in \cite[Theorem 1.3]{Pr}, and we omit it here.  In the following, we only deal with the second part.

	(2) By \cite[Theorem 3.6]{LZ} we know that  $\theta_\sigma$  is
$\bfG$-invariant, which can be achieved by the same arguments as in the proof of \cite[Theorem 1.1]{Pr}.
According to the Schur-Sergeev duality (see \cite{Ser}), $\text{End}_{\bfG}{(\bfV^{\otimes
d})}$  is generated by the twisted action of the  generators $(i\; i+1)\in
{\fraksd}$ as below
	\begin{align*}
	&(i\; i+1)\cdot v_{1}\otimes \cdots \otimes v_{i}\otimes v_{i+1}\otimes \cdots
\otimes v_{d}\\
	&=(-1)^{|v_{i}|\cdot |v_{i+1}|}v_{1}\otimes \cdots \otimes v_{i+1}\otimes
v_{i}\otimes \cdots \otimes v_{d},\quad 1\leq i\leq d-1.
	\end{align*}
	Recall  that there is an isomorphism as presented previously
	\begin{align*}
	\pi :((\bfV^*)^{\otimes d}\otimes \bfV^{\otimes d})^{*}\cong \text{End}_\bbc(\bfV^{\otimes
d})=(\End_\bbc(\bfV))^{\otimes d}.
	\end{align*}
	Hence, the generators of invariants in the Schur-Sergeev duality become
$\theta_\sigma$ for $\sigma\in {\fraksd}$ which are defined via
	\begin{align*}
	\langle \theta_\sigma, \;h_1\otimes h_2\otimes\cdots\otimes h_d\otimes
v_{1}\otimes\cdots \otimes v_{d}\rangle
	=(-1)^{J(h,v, \sigma)}\prod_{j}\langle h_{\sigma(j)}, v_{j}\rangle. 	
\end{align*}
\end{proof}

\begin{lemma}\label{lemma4-(2)}
Keep the notations and assumptions as in Lemma \ref{lemma4-(1)}.
 Let $\mathscr{V}
\triangleq\{\psi:\text{End}_{\bigwedge}(\bfV)\rightarrow\text{End}_{\bigwedge}(\bfV)^{\otimes
d }\mid\psi\circ g =g\circ\psi,\;\; \forall g\in \bfG \}$.
Then  $\mathscr{V}$  is a $\bigwedge$-module generated by the elements as follows:
	$$X\longmapsto \sigma\cdot(X^{h_{1}}\otimes X^{h_{2}}\otimes\cdots\otimes
X^{h_{d}}),$$
	where $h_{i}\in\bbz_+$ for $1\leq i\leq d$, and $\sigma\in {\fraksd}$ gives rise to a
permutation on the positions of the homogeneous tensor parts.
\end{lemma}

\begin{proof} Recall that
$\mathscr{V}=\{\psi:\text{End}_{\bigwedge}(\bfV)\rightarrow\text{End}_{\bigwedge}(\bfV)^{\otimes
d}\mid\psi\circ g =g\circ\psi,  \forall g\in \bfG  \}$, where
$g\cdot(X_{1}\otimes\cdots\otimes X_{d}) = {\Ad}g(X_{1})\otimes\cdots\otimes
{\Ad}g(X_{d})$.
For a given   $\varPhi \in \mathscr{V}$, we define a $\bigwedge$-function
$\bar{\varPhi}$ on $(\End_{\bigwedge}(\bfV))^{\otimes (d+1)}$, via
$$\bar{\varPhi}: X\otimes Y_{1}\otimes\cdots \otimes Y_{d}\longmapsto
\str(\varPhi(X)\circ(Y_{1}\otimes\cdots \otimes Y_{d})),$$
which is $\bfG$-invariant, and $\bigwedge$-linear in $Y_1, Y_2,\ldots,Y_d$. Denote by
$\mathcal{R}$ the $\bigwedge$-ring of all $\bfG$-invariant functions on
$(\End_{\bigwedge}(\bfV))^{\otimes(d+1)}$.
By the non-degenerate property of the supertrace we see that the mapping
$\varPhi \rightarrow \bar{\varPhi}$  is injective from  $\mathscr{V}$  to
$\mathcal{R}$.

On the other hand, Lemma \ref{3.3.1}(1) entails that all elements in  $\mathcal{R}$  are of the form
$$\sum a_X\str(X^{q_{1}}Y_{i_{1}}X^{q_{2}}Y_{i_{2}}\cdots)
\str(X^{p_{1}}Y_{j_{1}}X^{p_{2}}Y_{j_{2}}\cdots )\cdots
\str(X^{s_{1}}Y_{t_{1}}X^{s_{2}}Y_{t_{2}}\cdots ),$$
where $a_X\in \bigwedge$, $q_i,p_j,s_k\in \bbz_+$ and
$\sigma=(i_1i_2\cdots )(j_1j_2\cdots)\cdots(t_1t_2\cdots)\in \fraksd$ in the
expression is a product of some cycles not intersecting mutually. 
In addition, from Lemma  \ref{3.3.1}(2)-(3) it follows that  these elements are
of the form
$$\str(\Phi(X)\circ(Y_{1}\otimes Y_{2}\otimes \cdots \otimes Y_{d})),$$
where $\Phi(X)$  is  $\mathcal{R}$-spanned by the mapping
$X\longmapsto \sigma \cdot (X^{h_{1}}\otimes X^{h_{2}}\otimes \cdots \otimes
X^{h_{d}})$.

 The injective property of the mapping sending $\varPhi$  to  $\bar{\varPhi}$  entails
 that the second statement of the lemma.  The proof is completed.
\end{proof}

\subsection{Summation for the proof of Theorem \ref{thm Vust0}}\label{end of
pf th 4.1}  Let us turn back to 
the proof of Theorem \ref{thm
Vust0}.
       Let $\mathscr{B}_\bbc = \text{End}_{\rho({\ggg_{e}})}(V^{\otimes
d})$,  a subalgebra of $\text{End}_{\bbc}(V^{\otimes
d})$. By our analysis at the beginning of \S\ref{sec Vustproof}, we need to
decide  the generators of $\mathscr{B}_\bbc$.
One can easily reformulate the $\bbc$-versions of Lemmas \ref{lemma4-(1)}, \ref{3.3.1} and \ref{lemma4-(2)}, just replacing $\bigwedge$ with $\bbc$ in the original statements.

By Lemma \ref{lemma4-(1)}, for any  $b\in
\mathscr{B}_\bbc$ there is a
$G$-equivariant morphism $\Phi: \text{End}_{\bbc}(V)\rightarrow
\text{End}_{\bbc}(V^{\otimes d}) = \text{End}_{\bbc} (V)^{\otimes d}$ such that  $\Phi(e) = b$. By Lemma \ref{lemma4-(2)}, such an $\Phi$ is
generated by some elements of the form:	$X\mapsto \sigma\cdot(X^{h_{1}}\otimes
X^{h_{2}}\otimes\cdots\otimes X^{h_{d}})$ for $\sigma\in \fraksd$. Hence the theorem follows.

\section{Distribution superalgebras and double centralizers}\label{sec 3}

As a consequence of Theorem \ref{thm Vust0}, The second equality of Theorem \ref{doublecentralizerin} has been already proved. In this section, we will accomplish the remaining proof of Theorem \ref{doublecentralizerin}
by the same strategy as in \cite{BKl}.

Maintain the notations as in \S\ref{sec 1}. In particular, from now on we recover the notation $V=\bbc^{m|n}$. Still set $\ggg=\gl(V)$ and $e\in\ggg_\bz$ is a principal nilpotent element.
Recall that as in the introduction, we have the action $\bar\psi_d$ of  $\mathbb
C_n[x_1,\dots,x_d]\rrtimes \mathbb C {\fraksd}$ on $V^{\otimes d}$, and the action
$\phi_d$ of $U(\ggg_e)$ on $V^{\otimes d}$.
In order to prove that the centralizer of  $\bar\psi_d(\mathbb
C_n[x_1,\dots,x_d]\rrtimes \mathbb C {\fraksd})$ is exactly $\phi_d(U(\ggg_e))$,
we will realize the latter by the action of  $\bbc[M_e]^*_d$ on $V^{\otimes d}$,
where $\bbc[M_e]^*_d$ is  the homogeneous dual space of polynomial degree $d$ of
$\bbc[M_e]$ (see \S\ref{sec distri}). To carry this out, we will make use
of the description of $\ggg_e$ via ``pyramid" expression technique, along with
the
description of $U(\ggg_e)$ by the distribution algebra of $M_e$.

\subsection{Pyramids and a description of $\ggg_e$ with pyramids}\label{pyramid}
We take a standard principal nilpotent element $e:=\sum_{i=1}^{m-1}e_{\bar i,
\overline{i+1}}+\sum_{j=1}^{n-1}e_{j,j+1}\in \ggg_\bz$.
Accordingly, we have a pyramid associated with $(m|n)$. Roughly speaking, a
pyramid is a graph consisting of
connected horizontal lines of boxes. Corresponding to $(m|n)$, we have the
following pyramid
\begin{equation}\label{pyramid}
\Xi=
{\begin{picture}(90, 35)
	\put(15,-10){\line(1,0){105}}
	\put(15,5){\line(1,0){105}}
	\put(15,20){\line(1,0){75}}
	\put(15,-10){\line(0,1){30}}
	\put(30,-10){\line(0,1){30}}
	\put(45,-10){\line(0,1){30}}
	\put(60,-10){\line(0,1){30}}
	\put(75,-10){\line(0,1){30}}
	\put(90,-10){\line(0,1){30}}
	\put(105,-10){\line(0,1){15}}
	\put(120,-10){\line(0,1){15}}
	
\put(18,-5){$1$}\put(33,-5){$2$}\put(48,-5){$3$}\put(63,-5){$\cdot$}\put(78,-5){$\cdot$}\put(93,-5){$\cdot$}\put(108,-5){$n$}
	
\put(18,10){$\bo$}\put(33,10){$\cdot$}\put(48,10){$\cdot$}\put(63,10){$\cdot$}\put(78,10){$\overline{m}$}
	\end{picture}}\\[3mm]
\end{equation}

Let us define the coordinates for all $\bar i$ and $j$  appearing in the pyramid
$\Xi$.
Define the coordinate of $\bar i$ to be $(\bar 1, \bar i)$ for $i=1,\ldots,m$, and the
coordinate of $j$ to be $(1, j)$ for $j=1,\ldots,n$. Then we have the
following sets:
\begin{align*}
I:=&I(m|n) = \{\bar1,\ldots,\overline{m};1,\ldots,n\},\cr
J:= &\{(i,j)\in I\times I\mid\text{col}(i)\leq\text{col}(j), j=\overline m~~~\text{or}~~~n\}, \cr
K:= &\{(\bar1,\bar1,0),\ldots,(\bar1,\bar1,m-1);(\bar1,1,n-m),\ldots,
(\bar1,1,n-1);\cr
&(1,\bar1,0),\ldots,(1,\bar1,m-1); (1,1,0),\ldots,(1,1,n-1)\}.
\end{align*}
We define a map from  $J$ to $K$ by
\begin{align}\label{2.0000}
\upsilon: (i,j)\longmapsto (\text{row}(i),\text{row}(j),\text{col}(j)-\text{col}(i)),
\end{align}
which is  bijective.
 There is  a well-known  explicit description of the centralizer of any nilpotent
 element in a general linear Lie algebra associated with a pyramid (see \cite[Lemma
 7.3]{BK0} and \cite[IV.1.6]{SS}). By the same discussion as in \cite[Lemma 4.2]{BBG}, the same  description is still true for
 our principal nilpotent $e\in \gl(m|n)$ associated with $\Xi$.

\begin{lemma}\label{lem: basis}
 For $(i,j,r)\in K$, set  
\begin{equation}\label{hunny}
	e_{i,j;r} = \sum_{\substack{h,k \in I \\ \text{row}(h) = i, \text{row}(k) = j
\\
			\text{col}(k) - \text{col}(h) = r}} e_{h,k}.
	\end{equation}
Then $\{e_{i,j,r}\mid (i,j,r) \in K\}$ constitute a basis of $\ggg_e$, and all of
them are $\bbz_2$-homogeneous.

Alternatively, for $(i,j)\in J$ we set
\begin{align}\label{2.5555}
\xi_{i,j}: =e_{i,j;{\text{col}(j)-\text{col}(i)}}.
\end{align}
Then  $\{\xi_{i,j}\mid (i,j) \in J\}$ constitutes a basis of $\ggg_e$.

\end{lemma}

\begin{proof}  Note that $e\in \ggg_\bz$. The centralizer of $e$ in
$\ggg=\gl(m|n)$ is the same as a vector space as the centralizer of $e$ viewed as
an element of $\gl(m+n)$.  The lemma follows from \cite[IV.1.6]{SS} or
\cite[Lemma 7.3]{BK0}.
\end{proof}

\subsection{Some special notations for tensor products}\label{sec morenota}
We need some other notations.
\subsubsection{}\label{3.2.1}
Let $I^{d}$ and $J^{d}$ denote respectively, a
collection of
all $d$-tuples $\mathbf{i}=(i_{1},\ldots,i_{d})$ and  a collection of all pairs
$(\mathbf{i},\mathbf{j})$ with $\mathbf{i} = (i_{1},\ldots,i_{d})$, $\mathbf{j} =
(j_{1},\ldots,j_{d})$ satisfying  $(i_{k},j_{k})\in J$ for $1\leq k\leq d$.
Denote by $K^{d}$  a collection of all triples
$(\mathbf{i},\mathbf{j},\mathbf{r})$ with
$\mathbf{i} = (i_{1},\ldots ,i_{d})$, $\mathbf{j} = (j_{1},\ldots,j_{d})$,
$\mathbf{r} = (r_{1},\ldots,r_{d})$ satisfying each $(i_{k},j_{k},r_{k})\in K$ for $1\leq k\leq d$.

As a generalization of (\ref{2.0000}), there is a bijective map
$$\Upsilon: J^{d}\longrightarrow K^{d}$$
defined as
$$(\mathbf i,\mathbf j)\longmapsto (\row(\mathbf i),\row(\mathbf j),\col(\mathbf
j)-\col(\mathbf i)),$$
where $\text{row}(\mathbf i) = (\text{row}(i_{1}), \ldots, \text{row}(i_{d}))$,
and $\text{col} (\mathbf i)= (\text{col}(i_{1}), \ldots, \text{col}(i_{d}))$.


\subsubsection{} For $\bfi=(i_1,\ldots,i_d)\in I^d$, we set $\epsilon_\bfi:=(|i_1|,\ldots,|i_d|)\in \bbz_2^d$, and write $|\epsilon_\bfi|:= |i_1|+\cdots+|i_d|$. Furthermore, following \cite{BKu} we adopt some more notations as below. For $\epsilon=(\epsilon_1,\epsilon_2,\ldots,\epsilon_d)$ and $\delta=(\delta_1,\ldots,\delta_d)$ in $\bbz_2^d$, set
$$\alpha(\epsilon,\delta):=\prod_{1\leq s<t\leq d}(-1)^{\delta_s\epsilon_t}.  $$
Then we have $(X_1\otimes X_2\otimes \cdots\otimes X_d)(w_{1}\otimes\cdots\otimes w_{d})=\alpha((|X_1|,\ldots,|X_d|), (|w_1|,\ldots,|w_d|)) X_1(w_1)\otimes\cdots \otimes X_d(w_d)$ for $\bbz_2$-homogeneous elements $X_i\in \End_\bbc(V)_{|X_i|}$ and $w_j\in V_{|w_j|}$ for $i,j\in I$, which enables us to identify $\End_\bbc(V)^{\otimes d}$ with $\End_\bbc(V^{\otimes d})$.

\subsubsection{} Note that any $\sigma\in {\fraksd}$ acts on the right on $I^d$ by permutation; this is to say,
$\bfi.\sigma=(i_{\sigma(1)},\ldots, i_{\sigma(d)})$ for
$\bfi=(i_1,\ldots,i_d)\in I^d$. For $\epsilon\in \bbz_2^d$ and $\sigma\in \fraksd$, following \cite{BKu} again we set
$$ \nu(\epsilon,\sigma):=\prod_{\overset{1\leq s<t\leq d}{\sigma^{-1}(s)>\sigma^{-1}(t)}}(-1)^{\epsilon_s\epsilon_t}.$$

Arising from the map $\psi_d$ in (\ref{super symaction}), the action of $\fraksd$ on the right on $V^{\otimes d}$ becomes
$$(w_1\otimes\cdots\otimes w_d).\sigma=\nu((|w_1|,\ldots,|w_d|), \sigma) w_{\sigma(1)}\otimes \cdots\otimes w_{\sigma(d)}.$$

We can further define the action of $\fraksd$ on the right on $J^d$
diagonally, which means  $(\bfi,\bfj).\sigma :=(\bfi. \sigma,\bfj.\sigma)$ for any
$\sigma\in \fraksd$. This conjugation is denoted by $(\bfi,\bfj){\overset
\sigma\curvearrowright} (\bf{h},\bf{k})$ for $(\bf{h},\bf{k})= (\bfi,\bfj).\sigma$. Similarly,
$\fraksd$ acts diagonally on $K^d$. Choose the sets of orbit representatives $J^d\slash
{\fraksd}$ and $K^d\slash {\fraksd}$. Recall that for $\bfi\in I^d$, we write
$\row(\bfi)=(\row(i_1),\ldots,\row(i_d))$ and $\col(\bfi)= (\text{col}(i_{1}), \ldots, \text{col}(i_{d}))$.
In the same sense, set $\col(\bfj)-\col(\bfi)=(\col(j_1)-\col(i_1),\ldots,\col(j_d)-\col(j_d))$.
Then the
map $\Upsilon:J^d\rightarrow K^d$, $(\bfi,\bfj)\mapsto
(\row(\bfi),\row(\bfj),\col(\bfj)-\col(\bfi))$ is  ${\fraksd}$-equivariant.

We can write
\begin{align}\label{Upsilon}
\Upsilon(\bfi,\bfj){\overset
\sigma\curvearrowright} \Upsilon(\bfs,\bft)
 \end{align}
 for  $(\row(\bfi),\row(\bfj), \col(\bfj)-\col(\bfi)).\sigma= (\row(\bfs),\row(\bft), \col(\bft)-\col(\bfs))$.

\subsection{The distribution superalgebra associated with $e$} \label{sec distri}
Let $M_{e}$ be the
superalgebra consisting of all matrices in $\Mat(m|n;\bbc)$ that commute
with $e$. It is obvious that  $M_{e}$ is equal to  $ \ggg_{e}$ as vector spaces.
There is a $\mathbb{Z}_{2}$-graded structure on the coordinate  ring
$\mathbb{C}[M_{e}]$.  
Hence
\begin{align}\label{tonggou}
\mathbb{C}[M_{e}]\cong S((\ggg_e)_\bz^*)\otimes \bigwedge((\ggg_e)^*_{\bo}),
\end{align}
    which is a commutative superalgebra,  and  linearly spanned by canonical monomial elements with description as below (modulo the order)
\begin{align}\label{base}\prod_{(i,j)\in J}
x_{\row(i),\row(j);\col(j)-\col(i)}^{s_{i, j}}
\end{align}
where $s_{i,j}\in
\bbz_+$ if $|\row(i)|+|\row(j)|=\bar{\bar 0}$, and
$s_{i, j}\in\bbz_2$ if $|\row(i)|+|\row(j)|=\bar{\bar 1}$, and the factors $\{x_{h,k;t}\in M_{e}^*\mid (h,k,t)\in K\}$ are
the linear duals of  $\{e_{i,j,r}\in M_e\mid (i,j,r)\in K\}$ as
in Lemma \ref{lem: basis}, such that $\langle e_{i,j;r},x_{h,k;t}\rangle=\delta_{(i,j;r),(h,k;t)}(-1)^{|i|+|j|}$ with $\delta_{(i,j;r),(h,k;t)}:=\delta_{i,h}\delta_{j,k}\delta_{r,t}$.

The commutative superalgebra $\bbc[M_e]$ can be endowed with a Hopf superalgebra structure via
$$\Delta(x_{i,j,r})=\sum_{r_1+r_2=r;~(i,h,r_1), (h,j,r_2)\in K}(-1)^{(|i|+|h|)(|h|+|j|)}x_{i,h,r_1}\otimes x_{h,j,r_2},$$
and $\varepsilon(x_{i,j,r})=\delta_{i,j}\delta_{r,0}$.

Define the distribution space $\text{Dist}(M_{e}) = \{f\in
\mathbb{C}[M_{e}]^{*}\mid f((\ker\varepsilon)^{m}) = 0  \text{ for } m\gg0\}$, 
where $\varepsilon$ is the unit map of $\bbc[M_e]$.
There is a multiplication on $\bbc[M_e]^*$ dual to the comultiplication on $\bbc[M_e]$, defined via $(xy)f=(x\bar\otimes y)\Delta(f)$ for $x,y\in \bbc[M_e]^*$ and $f\in \bbc[M_e]$.  Here and further the notation $\bar\otimes$ means the operation     obeying the rule of signs as below:
$$(x\bar\otimes y) (f\otimes g)=(-1)^{|y||f|}x(f)y(g).$$
Furthermore, $\text{Dist}(M_e)$  can be endowed with a Hopf
algebra structure, which  is isomorphic to $U(\ggg_e)$ as Hopf superalgebras (see \cite[\S2.2]{BKl} or \cite[\S2.2]{SW}).

\subsubsection{} Set the degree of each polynomial generator to be  $1$.
 Then $\mathbb{C}[M_{e}]$ is endowed with a graded algebra structure:
$$\mathbb{C}[M_{e}] = \bigoplus_{i\geq 0} \mathbb{C}[M_{e}]_{i}.$$
Moreover, each graded dual subspace $\mathbb{C}[M_{e}]_{i}^{*}$ of the dual
algebra $\mathbb{C}[M_{e}]^*$ is finite-dimensional.

Recall that (\ref{base}) provides a  basis of the coordinate superalgebra
$\bbc[M_e]$. Then $\mathbb{C}[M_{e}]_{d}^{*}$ has the corresponding dual basis.
Now we write a set of basis elements for $\mathbb{C}[M_{e}]_{d}^{*}$
specifically. It follows from (\ref{tonggou}) that the base of
$\mathbb{C}[M_{e}]_{d}$  is of the form $x_{\mathbf{i},\mathbf{j},\mathbf{k}} =
x_{i_{1}, j_{1}, k_{1}} \cdots x_{i_{d}, j_{d}, k_{d}}$ for $(\bfi,\bfj,\bfk)\in
K^d\slash {\fraksd}$.
We introduce a set of  basis elements
$\{\xi_{\bfi,\bfj,\bfk}\mid (\bfi,\bfj,\bfk)\in K^d\slash
 {\fraksd}\}$ of $\mathbb{C}[M_{e}]^{*}_{d}$ dual to
 $\{x_{\bfi,\bfj,\bfk}\mid (\bfi,\bfj,\bfk)\in K^d\slash  {\fraksd}\}$. Note that the bijective map  $\Upsilon: J^d \rightarrow K^d$ satisfies the $\fraksd$-equivariance and  the property   $\nu(\epsilon_\bfi+\epsilon_\bfj,\sigma)
 =\nu(\epsilon_{\row(\bfi)}+\epsilon_{\row(\bfj)},\sigma)$ because $\epsilon_\bfi=\epsilon_{\row(\bfi)}$ for any $\bfi\in I^d$.
  Then for any $(\bfi,\bfj)\in J^d$ we have
\begin{align*}
\langle\xi_{\Upsilon(\bfi,\bfj)}, x_{\Upsilon(\bfi,\bfj)}\rangle
=\alpha(\epsilon_\bfi+\epsilon_\bfj, \epsilon_\bfi+\epsilon_\bfj)
\end{align*}
and
\begin{align} \label{angle value}
&\langle \xi_{\text{row}(\mathbf{i}),\text{row}(\mathbf{j});
\text{col}(\mathbf{j})-\text{col}(\mathbf{i})}, x_{\bfs,\bft,\text{col}(\mathbf{t})-\text{col}(\mathbf{s})}\rangle\cr
=&\begin{cases} \alpha(\epsilon_\bfi+\epsilon_\bfj, \epsilon_\bfi+\epsilon_\bfj)\nu(\bfi+\bfj,\sigma)  &\mbox{if there exists }\sigma\in \fraksd \mbox{ such that } \Upsilon(\bfi,\bfj){\overset
\sigma\curvearrowright} \Upsilon(\bfs,\bft),\cr
0  &\mbox{otherwise.}
\end{cases}
\end{align}


\subsection{}
 By the above arguments, we can identify  $U(\ggg_{e})$  with
$\text{Dist}(M_{e})$.
Let us define an algebra homomorphism:
$$\pi_{d}: U(\ggg_{e})\rightarrow \mathbb{C}[M_{e}]_{d}^{*}$$
such that for all  $u\in U(\ggg_{e}), x\in \mathbb{C}[M_{e}]_{d}$, we have
$\pi_{d}( u)(x) =\langle u, x\rangle$.
\begin{lemma}\label{flu1}
	$\pi_{d}$  is surjective.
\end{lemma}
\begin{proof} The arguments are the same as in the Lie algebra case (see
\cite[Lemma 2.1]{BKl}).
	Suppose the contrary. Then we can find  $0\neq x\in\mathbb{C}[M_{e}]_{d}$
such that  $\pi_{d}(u)(x)=0$ for all $u\in U(\ggg_{e})$.
By identifying  $U(\ggg_e)$ with $\Dist(M_e)$, for each
$m\geq 1$ we have  $u(x)=0$ all $u\in \mathbb{C}[M_{ e}]^{*}$ with $u((\text{ker
}\varepsilon)^{m}) =0$.
    So  $x\in (\text{ker }\varepsilon)^{m}$ by definition.
However,  we already have $\bigcap _{m\geq 1}(\text{ker }\varepsilon)^{m} = 0$,  which is a
contradiction.
\end{proof}

\subsection{The comodule structure on $V^{\otimes d}$}

   Let us first observe that a $\ggg_{e}$-module $V$  can be endowed with a natural
   right $\mathbb{C}[M_{e}]$-comodule structure as follows.

Denote by $\rho$ the action of $\ggg$ on the natural
module $V$.
Fix basis elements $v_{\bar1}, \ldots, v_{\overline{m}}, v_{1}, \ldots, v_{n}$
of $V$ such that $v_{\bar1},\ldots, v_{\overline m}\in V_\bz$ and $v_1,\ldots,
v_{n}\in V_\bo$.  We can write
$$\rho(x)v_{i}=\sum_{j=\bar1}^{\overline{m}}\rho_{ji}(x)v_{j}+\sum_{j=
1}^{n}\rho_{ji}(x)v_{j}, \quad \rho_{ji}\in\ggg^*.$$
Note that by definition $\rho$ is  even. So $\rho_{ji}\in \ggg^*_{|j|+|i|}$ where
$|i|=\bz$ if $i\in\{\bar 1,\ldots,\overline m\}$, and $|i|=\bo$ if
$i\in\{1,\ldots,n\}$.
Consequently, the structure of comodule on $V^{\otimes d}$ can be described via
\begin{align*}
\Delta^{(1)}(v_{i}) = \sum_{j=\bar1}^{\overline{m}}
v_{j}\otimes \rho_{ji}+\sum_{j=1}^{n}(-1)^{|j|+|i|}v_{j}\otimes
\rho_{ji}.
\end{align*}
Turning  to the $\ggg_e$-module $V$, we first have
\begin{lemma}\label{lemma: comodule}
The comodule structure on $V$ can be given as follows: 	
\begin{equation}\label{equ: comodule}
  \begin{array}{lcll}
\Delta^{(1)}:&V&\rightarrow&V\otimes \mathbb{C}[M_{e}]\\
&v_{i}&\mapsto&\sum_{j\in I; (j,i)\in J}(-1)^{|j|(|j|+|i|)}
v_{j}\otimes x_{\upsilon(j,i)}.
\end{array}
\end{equation}
%
%
\end{lemma}
\begin{proof} It follows by a straightforward computation.
\end{proof}
\begin{remark} In \cite[\S 10]{BKl1} Brundan-Kleshchev gave some similar arguments for
Lie superalgebras of type $Q$.
\end{remark}

\subsubsection{} Generally, for a given bi-superalgebra $\ca$, and a (super) comodule
$C$ over $\ca$ with defining map $\Delta^{(1)}:C\rightarrow C\otimes \ca$ of the image
$\Delta^{(1)}(c)=\sum c_0\otimes c_1$, using $\ca$-comodule structure on $C$ one can  define an
$\ca$-comodule on $C\otimes C$ 
as
below (see \cite{Abe} for the details):
\begin{align}\label{2.10}
\Delta^{(2)}(c\otimes d ) = \sum c_{0}\otimes d_{0}\otimes c_{
1}d_1,\quad\forall  c, d \in C.
\end{align}
Inductively, one can endow $C^{\otimes d}$ with an $\ca$-comodule structure for any
positive integer $d$, of which the defining map is
\begin{align}\label{2.10'}
\Delta^{(d)}: C^{\otimes d}\rightarrow C^{\otimes d}\otimes \ca.
\end{align}

 Furthermore, $\ca^*$ is also a superalgebra. One can define  a super $\ca^*$-module
 on  $C$ as below
\begin{align}\label{qiuhe}
f.c = \sum\langle f, c_1\rangle c_0,~\text{where}~f\in \ca^{*}, c\in C.
\end{align}
Inductively, one can endow $C^{\otimes d}$ with $\ca^*$-module structure for any
positive integer $d$ arising from (\ref{2.10'}), with the initial step
(\ref{qiuhe}).

\subsubsection{} Let us return to the case of $\bbc[M_e]$.
Recall that $\mathbb{C}[M_{e}]$ is  a bi-superalgebra, and there is a (super)
comodule structure on $V^{\otimes d}$, with comodule map
$\Delta^{(d)}:\bbc[M_e]\rightarrow V^{\otimes d}\otimes \bbc[M_e]$, which is
defined inductively  with the initial step from  (\ref{equ: comodule}). To be precise, for any $\bft=(t_1,\ldots,t_d)\in I^d$ and any monomial basis vector
$v_{\mathbf{t}}:=v_{t_{1}}\otimes
v_{t_{2}}\otimes\cdots\otimes v_{t_{d}}\in V^{\otimes d}$, combining Lemma
\ref{lemma: comodule} along with (\ref{2.10})-(\ref{2.10'}), and   the induction
initiated from (\ref{equ: comodule}), we have that
\begin{align}\label{comodule precise}
\Delta^{(d)}(v_\bft) = \sum_{\overset{\bfs\in I^d}{\text{with }(\bfs,\bft)\in J^d}}(-1)^{|\epsilon_\bfs|(|\epsilon_\bfs|+|\epsilon_\bft|)}
\alpha(\epsilon_\bfs+\epsilon_\bft,\epsilon_\bfs)v_\bfs\otimes x_{\Upsilon(\bfs,\bft)},
\end{align}
where the requirement of the parameters in the sum implies that $\mathbf{s}=(s_1,s_2,\ldots,s_d)\in I^d$ satisfies $(s_q,t_q)\in J$ for
every $q=1,\ldots,d$, and
$$x_{\Upsilon(\mathbf{s},\bft)}=x_{s_1,t_1;\col(t_1)-\col(s_1)}x_{s_2,t_2;\col(t_2)-\col(s_2)}\cdots
x_{s_d,t_d;\col(t_q)-\col(s_d)}.$$

 Consequently, $V^{\otimes d}$
becomes a $\mathbb{C}[M_e]^*$-module. This representation of $\mathbb{C}[M_e]^*$  on
$V^{\otimes d}$ is denoted by $\omega_d$ which will play a critical role in the sequel.

\subsection{} For $(\bfi, \bfj)\in J^d$ (see
 \S\ref{sec morenota}), set
\begin{align}\label{equ: xi bfi bfj}
\Theta_{\mathbf{i},\mathbf{j}}:=
\omega_{d}(\xi_{\text{row}(\mathbf{i}),\text{row}(\mathbf{
j});\text{col}(\mathbf{j})-\text{col}(\mathbf{i})}).
\end{align}
Recall that $V$ has a basis $\{v_i\mid i\in I\}$. Then $V^{\otimes d}$ admits a basis $\{v_{\bfi}:=v_{i_1}\otimes\cdots\otimes v_{i_d}\mid \bfi=(i_1,\ldots,i_d) \in I^d\}$.

Recall in (\ref{2.5555}), we have defined $\xi_{s,t}$ for $(s,t)\in J$. For $(s_i,t_i)\in J$ with $i=1,\ldots,d$ and $\bfs=(s_1,\ldots,s_d)$, $\bft=(t_1,\ldots,t_d)$, we continue to define   $\xi_{\bfs,\bft}:=\xi_{s_1,t_1}\otimes\cdots\otimes \xi_{s_d,t_d}\in \End_\bbc(V)^{\otimes d}\cong \End_\bbc(V^{\otimes d})$.
Then we can precisely formulate  $\Theta_{\bfs,\bft}$ as below.

\begin{lemma}\label{Theta precise} Keep the notations as above and in (\ref{Upsilon}). Then
$$\Theta_{\bfi,\bfj}=\alpha(\epsilon_\bfi+\epsilon_\bfj, \epsilon_\bfi+\epsilon_\bfj)\sum_{\overset{\sigma\in \fraksd}{\Upsilon(\bfi,\bfj){\overset{\sigma}{\curvearrowright}}
\Upsilon(\bf{s},\bf{t})}}\nu(\bfi+\bfj;\sigma)\xi_{\bfs,\bft}.$$
\end{lemma}
\begin{proof} For any basis element $v_{\bft}\in V^{\otimes d}$ with $\bft\in I^d$, by (\ref{comodule precise}) we have
\begin{align*}
\Theta_{\bfi,\bfj} v_\bft&=(1\bar\otimes \xi_{\text{row}(\mathbf{i}),\text{row}(\mathbf{
j});\text{col}(\mathbf{j})-\text{col}(\mathbf{i})})\Delta^{(d)}(v_\bft)\cr
&= \sum_{\overset{\bfs\in I^d}{\text{with }(\bfs,\bft)\in J^d}}(-1)^{|\epsilon_\bfs|(|\epsilon_\bfi|+|\epsilon_\bfj|)+|\epsilon_\bfs|
(|\epsilon_\bfs|+|\epsilon_\bft|)}
\alpha(\epsilon_\bfs+\epsilon_\bft,\epsilon_\bfs)
\langle \xi_{\text{row}(\mathbf{i}),\text{row}(\mathbf{
j});\text{col}(\mathbf{j})-\text{col}(\mathbf{i})}, x_{\Upsilon(\bfs,\bft)}\rangle
v_\bfs.
\end{align*}
Note that
$\langle\xi_{\text{row}(\mathbf{i}),
\text{row}(\mathbf{j});\text{col}(\mathbf{j})-\text{col}(\mathbf{i})}, x_{\Upsilon(\bfs,\bft)}\rangle$
is nonzero if and only if there exists $\sigma \in \fraksd$ such that $\Upsilon(\bfi,\bfj)\overset{\sigma}{\curvearrowright}
\Upsilon(\bf{s},\bf{t})$. In this case, by (\ref{angle value}) we have $\langle\xi_{\text{row}(\mathbf{i}),
\text{row}(\mathbf{j});\text{col}(\mathbf{j})-\text{col}(\mathbf{i})}, x_{\Upsilon(\bfs,\bft)}\rangle=\alpha(\epsilon_\bfi+\epsilon_\bfj, \epsilon_\bfi+\epsilon_\bfj)\nu(\bfi+\bfj,\sigma)$.
So
  \begin{align*}
  \Theta_{\bfi,\bfj} v_\bft&=\alpha(\epsilon_\bfi+\epsilon_\bfj, \epsilon_\bfi+\epsilon_\bfj)\sum_{\overset{\sigma\in \fraksd}{\Upsilon(\bfi,\bfj){\overset{\sigma}{\curvearrowright}}
\Upsilon(\bf{s},\bf{t})}}
\alpha(\epsilon_\bfs+\epsilon_\bft,\epsilon_\bfs)\nu(\bfi+\bfj,\sigma)
v_\bfs.
\end{align*}
Keep in mind that
\begin{align}\label{multimatrix action}
e_{\bfh,\bfk}v_\bft=\delta_{\bfk,\bft}
\alpha(\epsilon_\bfh+\epsilon_\bfk,\epsilon_\bft)v_\bfh
\end{align}
for $e_{\bfh,\bfk}:=e_{h_1,k_1}\otimes\cdots\otimes e_{h_d,k_d}$ with
$\bfh=(h_1,\ldots,h_d)$ and $\bfk=(k_1,\ldots,k_d)$.
By a straightforward computation, we have
$$\xi_{\bfh,\bfk} v_\bft=\delta_{\bfk,\bft}\alpha(\epsilon_\bfh+\epsilon_\bfk,\epsilon_\bft)v_\bfh.$$
This proves  the formula for $\Theta_{\bfi,\bfj}$.
\end{proof}

\begin{lemma}\label{flu2} The representation $\omega_d$ is faithful.
\end{lemma}

\begin{proof} By Lemma \ref{Theta precise}, it follows that $\omega_d$ maps the
basis elements
     $\xi_{\text{row}(\mathbf{i}),\text{row}(\mathbf{j});
     \text{col}(\mathbf{j})-\text{col}(\mathbf{i})}$ of $\mathbb C[M_e]_d^*$ to
     linearly independent elements.  
\end{proof}

Note that $\End_\bbc(V^{\otimes d})\cong \End_\bbc(V)^{\otimes d}$. Then we will identify
both of them.  For any $E\in
\End_\bbc(V^{\otimes d})$, we can write $E=\sum_{(\bfi,\bfj)\in
I^d\times I^d}a_{\bfi,\bfj}e_{\bfi,\bfj}$ for $a_{\bfi,\bfj}\in \bbc$ and
$e_{\bfi,\bfj}:=e_{i_1,j_1}\otimes\cdots\otimes e_{i_d,j_d}$ with
$\bfi=(i_1,\ldots,i_d)$
and $\bfj=(j_1,\ldots,j_d)$.
Recall that we have already defined the algebra
$\mathbb{C}_{n}[x_{1},\ldots,x_{d}]\rrtimes\mathbb{C}{\fraksd}$, and its action on
$V^{\otimes d}$ in the introduction.

\begin{lemma}\label{flu3} Keep the notations as above. Suppose
$E=\sum_{(\bfi,\bfj)\in I^d\times I^d} a_{\bfi,\bfj}e_{\bfi,\bfj}\in
\End_\bbc(V^{\otimes
d})$.
If $E$ commutes with the action of  $\mathbb{C}_n[x_{1},\ldots,x_{d}]\rrtimes
\mathbb{C}{\fraksd}$, then those  coefficients $a_{\bfi,\bfj}$ satisfy the
following conditions:
\begin{itemize}
	\item[(1)] $a_{\bfi,\bfj} = 0$ if there exit $k\in \{1,\ldots,d\}$ such that
$\operatorname{col}(i_k) > \operatorname{col}(j_k)$. Moreover, the
nonzero terms come from the parameters $(\bfi,\bfj)\in J^d$.
	\item[(2)]
$a_{\mathbf{i},\mathbf{j}}=\nu(\epsilon_\bfi+\epsilon_\bfj, \sigma) a_{\bfi.\sigma,\bfj.\sigma}$
for all $\sigma\in {\fraksd}$.
	\item[(3)] $a_{\bfi,\bfj}=\nu(\epsilon_\bfi+\epsilon_\bfj, \sigma)a_{\bfs,\bft}$ if there exists some $\sigma\in \fraksd$ such that $\Upsilon(\bfi,\bfj){\overset{\sigma}{\curvearrowright}}
\Upsilon(\bf{s},\bf{t})$.
\end{itemize}
\end{lemma}

\begin{proof} 
As to Part (2), first note that the image of $v_\bfj$ under $E$ is $\sum_{\bfi\in I^d} a_{\bfi,\bfj}v_\bfi$. Applying $\sigma\in {\fraksd}$, we get
$\sum_{\bfi\in I^d}\nu(\epsilon_\bfi+\epsilon_\bfj, \sigma) a_{\bfi,\bfj}v_{\bfi.\sigma}$. This must equal the image of $v_{\bfi.\sigma}$ under $E$, i.e.,
$\sum_{\bfi\in I^d} a_{\bfi.\sigma,\bfj.\sigma}v_{\bfi.\sigma}$. Equating coefficients gives $a_{\mathbf{i},\mathbf{j}}=\nu(\epsilon_\bfi+\epsilon_\bfj, \sigma) a_{\bfi.\sigma,\bfj.\sigma}$. This proves the statement in (2).

Now we proceed with arguments for Parts (1) and (3).  Take $z_k=1\otimes\cdots \otimes \overset{k\text{th}}{e}\otimes \cdots\otimes 1\in \End_\bbc(V)^{\otimes d}$,  of which all the components in the tensor product are the identity except the $k$th component equal
to $e$. Then $z_k$ is naturally an even element in $\End_\bbc(V^{\otimes d})$.
From the assumption $z_k\circ E=E\circ z_k$, we verify the statement in (3).

For this, we first adopt some notations. For $\ell\in I$, we set
\begin{align*}
e_{\ell-\imath,\ell}:=\begin{cases} 0, &\mbox{ if } \ell=\bar 1, \mbox{ or } 1;\cr
e_{\overline{j-1},\bar j} &\mbox{ if } \ell=\bar j\in \{\bar 2,\ldots,\overline{m}\};\cr
e_{j-1,j} &\mbox{ if } \ell=j\in\{2,\ldots,n\}.
\end{cases}
\end{align*}
    In the above, we are implicitly using the notation $\ell-\imath$. In the same spirit, we can talk about $\ell+\imath$.
Furthermore, for any $\bfj=(j_1,\ldots,j_d)\in I^d$ and $k\in\{1,\ldots,d\}$, we set
\begin{align*}
e_{\bfj-\imath_k,\bfj}
:=e_{j_1,j_1}\otimes e_{j_2,j_2}\otimes \cdots \otimes e_{j_{k-1},j_{k-1}}\otimes e_{j_k-\imath,j_k}\otimes e_{j_{k+1},j_{k+1}}\otimes\cdots\otimes e_{j_d,j_d},
\end{align*}
and
\begin{align}\label{eq: imath}
v_{\bfj-\imath_k}
:=\begin{cases} 0, &\mbox{ if } j_k=\bar 1\mbox{ or } 1;\cr
v_{j_1}\otimes v_{j_2}\otimes \cdots \otimes v_{j_{k-1}}\otimes v_{j_k-\imath}\otimes v_{j_{k+1}}\otimes\cdots\otimes v_{j_d}, &\mbox{ otherwise}.
\end{cases}
\end{align}
Using (\ref{multimatrix action}), we have $z_k(v_\bfj)=\alpha(\epsilon_{\bfj-\imath_k}+\epsilon_\bfj,\epsilon_\bfj)v_{\bfj-\imath_k}$. Note that $\alpha(\epsilon_{\bfj-\imath_k}+\epsilon_\bfj,\epsilon_\bfj)=1$, and consequently
\begin{align}\label{zk action}
z_k(v_\bfj)=v_{\bfj-\imath_k}.
\end{align}

Next, for any given  $v_\bfj\in V^{\otimes d}$, the commutation concerning $E$ assures that
$E\circ z_k(v_\bfj)=z_k\circ E(v_\bfj)$.
By (\ref{zk action}) and \eqref{multimatrix action} we have
\begin{align*}
E\circ z_k(v_\bfj)&=\sum_{\bfs,\bft\in I^d}a_{\bfs,\bft}e_{\bfs,\bft}z_k(v_\bfj) \cr
&=\sum_{\bfs,\bft\in I^d}a_{\bfs,\bft}e_{\bfs,\bft} v_{\bfj-\imath_k}\cr
&=\sum_{\bfs\in I^d}a_{\bfs,\bfj-\imath_k} \alpha(\epsilon_\bfs+\epsilon_{\bfj-\imath_k},\epsilon_{\bfj-\imath_k})v_\bfs.
\end{align*}
On the other hand, by (\ref{multimatrix action}) again we have
\begin{align*}
z_k\circ E(v_\bfj)&=z_k(\sum_{\bfs\in I^d}a_{\bfs,\bfj}\alpha(\epsilon_\bfs+\epsilon_\bfj,\epsilon_\bfj)v_\bfs) \cr
&=\sum_{\bfs\in I^d}a_{\bfs,\bfj}\alpha(\epsilon_\bfs+\epsilon_\bfj,\epsilon_\bfj)z_k(v_\bfs)\cr
&=\sum_{\bfs\in I^d}a_{\bfs,\bfj}\alpha(\epsilon_\bfs+\epsilon_\bfj,\epsilon_\bfj) v_{\bfs-\imath_k}.
\end{align*}
Comparing  $E\circ z_k(v_\bfj)$ with $z_k\circ E(v_\bfj)$, we have
$$a_{\bfs,\bfj-\imath_k}\alpha(\epsilon_\bfs+\epsilon_{\bfj-\imath_k},\epsilon_{\bfj-\imath_k})
=a_{\bfs+\imath_k,\bfj}\alpha(\epsilon_{\bfs+\imath_k}+\epsilon_\bfj,\epsilon_\bfj).$$
Note that  $\alpha(\epsilon_{\bfs+\imath_k}+\epsilon_\bfj,\epsilon_\bfj)=\alpha(\epsilon_\bfs+\epsilon_\bfj,\epsilon_\bfj)=\alpha(\epsilon_\bfs+\epsilon_{\bfj-\imath_k},\epsilon_{\bfj-\imath_k})$.
So we finally have  $a_{\bfs,\bfj-\imath_k}=a_{\bfs+\imath_k,\bfj}$ for any $\bfs, \bfj\in I^d$ and $k\in\{1,\ldots, d\}$  whenever $\bfs+\imath_k$ and $\bfj-\imath_k$ make sense. Therefore, $a_{\bfs,\bfj}=a_{\bfs-\imath_k,\bfj-\imath_k}$ for any $\bfs, \bfj\in I^d$ and $k\in\{1,\ldots, d\}$ whenever $\bfs-\imath_k$ and $\bfj-\imath_k$ make sense, and $a_{\bfs,\bfj}=0$ otherwise.
Then Parts (1) and (3) follow easily from this and the statement in (2).

The proof is completed.
\end{proof}

As an immediate consequence, we have the following result.
\begin{corollary} If $E\in \End_\bbc(V^{\otimes d})$ commutes with the action of
$\mathbb{C}_{n}[x_{1},\ldots,x_{d}]\rrtimes
\mathbb{C}{\fraksd}$, then $E$ is a linear combination of $\Theta_{\bfi,\bfj}$ with
$(\bfi,\bfj)\in J^d$ in (\ref{equ: xi bfi bfj}).
\end{corollary}

\subsection{Double centralizer property and the proof of Theorem
\ref{doublecentralizerin}}\label{2.5}
Let us  proceed with the proof of double centralizer property. Keep in
mind that $V=\bbc^{m|n}$, $\ggg=\gl(V)$ and $e\in\ggg_\bz$ is a principal nilpotent
element. Recall that $\phi_d:U(\ggg_e) \rightarrow
\operatorname{End}_{\mathbb C}(V^{\otimes d})$ is  a representation arising from
the natural action induced by  $\ggg_e$  on the tensor space. By definition,
$\phi_d$ is
precisely the composite map $\omega_d \circ \pi_d$. By virtue of the arguments
in the proof of Lemmas  \ref{flu1}  and  \ref{flu2}, the image $\im\phi_d$ is
isomorphic to  $\mathbb C[M_e]_d^*$. Let  $ \bar\psi_d:\mathbb
C_n[x_1,\dots,x_d]\rrtimes \mathbb C {\fraksd}\rightarrow
\operatorname{End}_{\mathbb C}(V^{\otimes d})^{\op}$ be the homomorphism arising
from the
right action of   $\mathbb C_n [x_1,\dots,x_d]\rrtimes \mathbb C {\fraksd}$  on
$V^{\otimes d}$. Thus, we have the following  maps
\begin{equation*}\label{grgrpre}
U(\ggg_e) \stackrel{\phi_d}{\longrightarrow}
\operatorname{End}_{\mathbb C}(V^{\otimes d})
\stackrel{\bar\psi_d}{\longleftarrow}
\mathbb C_n[x_1,\dots,x_d]\rrtimes \mathbb C {\fraksd}.
\end{equation*}

The Dynkin $\mathbb{Z}$-grading on $\ggg$ (see \S\ref{sec finite W})
 extends to a grading
$U(\ggg) = \bigoplus_{r \in \mathbb{Z}} U (\ggg)(r)$
on its universal enveloping algebra, and $U(\ggg_e)$
is a graded subalgebra.
Consider $V$ as a graded module
by declaring
that each $v_i$ is of degree $(n- \operatorname{col}(i))$.
There are induced gradings
on $V^{\otimes d}$ and its endomorphism algebra
 $\operatorname{End}_{\mathbb C}(V^{\otimes d})$,
so that the map $\phi_d$ is then a homomorphism of graded algebras.
Also define a grading on  $\mathbb C_n[x_1,\dots,x_d]\rrtimes \mathbb C {\fraksd}
$
by declaring that each $x_i$ is of degree $1$ and each $w \in {\fraksd}$ is of
degree
$0$.
The map $\bar\psi_d$ is then a homomorphism of graded algebras as well.

Now we are in the position to prove Theorem \ref{doublecentralizerin}.
\begin{proof} Recall that $e\in \ggg_\bz$ is principal
nilpotent.
	The second equality is based on Theorem  \ref{thm Vust0}. Owing to Lemma
\ref{flu3} and its corollary, any element of
	  $\operatorname{End}_{\mathbb C_n[x_1,\dots,x_d] \rrtimes \bbc
{\fraksd}}(V^{\otimes d})$ is a linear combination of
	the elements $\Theta_{\bfi,\bfj}$ with
$(\bfi,\bfj)\in J^d$ in (\ref{equ: xi bfi bfj}).
	These elements belong to the image of $\phi_d$
	by Lemma \ref{flu1}. Thus, the first equality
follows.
\end{proof}

\section{$U(\ggg_e)$ as a contraction of finite $W$-superalgebra for any nilpotent
$e$}\label{sec 5}
In this section, we assume that $\ggg$ is any basic classical  Lie superalgebra over
$\bbc$ as defined in \cite{Kac}, \cite{CW},  \cite{ZS4} or \cite{ZS1}, and $e$ is any
nilpotent element in $\ggg_\bz$.     In this section, we will prove Theorem \ref{cdd1}. A  consequence of this result is important to the final proof of Theorem \ref{thm:dcthm}.  This result shows that under the Dynkin grading  which is  denoted by $\grsf$, there is an isomorphism between $\grsf W_\chi$ and $U(\ggg_e)$ for the finite $W$-superalgebra associated with $\chi=(e,-)\in\ggg^*_\bz$ when the ``judging number" $\sfr$ (see \S\ref{sec finite W}) is even. This result is a super version of \cite[Theorem 3.8]{BGK} and \cite[Proposition 2.1]{P2}.

Recall that in \cite{P3} Premet introduced and developed the theory of finite  $W$-algebra
$U(\ggg,e)$ (in the present paper,  we instead use the notation $W_{\chi}$ for
$\chi=(e,-)$, see \S\ref{sec finite W}). Partial counterpart theory
in the super case has been established in \cite{Zeng}, \cite{ZS}, \cite{ZS1} and \cite{ZS5}.

\subsection{Kazhdan filtration}\label{kazhdan}
First recall some results in \cite{ZS}. For any given nilpotent element
$e\in\ggg_\bz$, let $\{e,f,h\}$ be an $\mathfrak{sl}_2$-triple in $\ggg$, then under
the action of ad\,$h$ we have $\ggg=\bigoplus_{i\in\frac{\mathbb{Z}}{2}}{\ggg}(i)$
{with ${\ggg}(i)=\{x\in{\ggg}\mid[h,x]=2ix\}$, which is called a Dynkin grading.} Moreover, the definition of Dynkin grading on $\ggg$ can be generalized to the ones on $U(\ggg)$ and its subalgebras under the action of ad\,$h$.  It follows from
\cite[Proposition 2.1]{ZS} that there exists an even non-degenerate
supersymmetric invariant bilinear form $(-,-)$ such that $(e,f)=1$. For
any $x\in\ggg$, define $\chi(x)=(e,x)$.
As the even part $\ggg_{\bz}$ of $\ggg$ is a reductive Lie algebra, it is easily known
 that $(\mathfrak{g}_{h})_{\bz}\cap(\ggg_{e})_{\bz}$ is a Levi  subalgebra
of $(\ggg_{e})_{\bz}$. Pick a maximal toral subalgebra $\mathfrak{t}_{e}\subseteq
\ggg(0)_{\bz}$ of this Levi  subalgebra, and a Cartan subalgebra  $\mathfrak{t}$
of $\ggg$ containing $\mathfrak{t}_{e}$ and $h$.

There exists a symplectic (resp. symmetric) bilinear form
$\langle-,-\rangle'$ on the ${\bbz}_2$-graded subspace ${\ggg}(-\frac{1}{2})_{\bz}$
(resp. ${\ggg}(-\frac{1}{2})_{\bo}$) given by $\langle x,y\rangle':=(e,[x,y])=\chi([x,y])$
for all $x,y\in{\ggg}(-\frac{1}{2})_{\bz}~(\text{resp.}\,x,y\in{\ggg}(-\frac{1}{2})_{\bo})$.
There exist bases $\{u_1,\ldots,u_{2s}\}$ of ${\ggg}(-\frac{1}{2})_{\bz}$ and
$\{v_1,\ldots,v_\sfr\}$ of ${\ggg}(-\frac{1}{2})_{\bo}$  such that $\langle u_i, u_j\rangle'
=i^*\delta_{i+j,2s+1}$ for $1\leq i,j\leq 2s$, where
$i^*=\left\{\begin{array}{ll}-1&\text{if}~1\leq i\leq s;\\
1&\text{if}~s+1\leq i\leq 2s\end{array}\right.$, and $\langle
v_i,v_j\rangle'=\delta_{i+j,\sfr+1}$ for $1\leq i,j\leq\sfr$.

Set $\mathfrak{m}:=\bigoplus_{i\leq -1}{\ggg}(i)\oplus{\ggg}(-\frac{1}{2})^{\prime}$
with
${\ggg}(-\frac{1}{2})^{\prime}={\ggg}(-\frac{1}{2})^{\prime}_{\bz}\oplus{\ggg}(-\frac{1}{2})^{\prime}_{\bo}$,
where ${\ggg}(-\frac{1}{2})^{\prime}_{\bz}$ is the ${\bbc}$-span of $u_{s+1},\ldots,u_{2s}$
and
${\ggg}(-\frac{1}{2})^{\prime}_{\bo}$ is the ${\bbc}$-span of
$v_{\frac{\sfr}{2}+1},\ldots,v_\sfr$ (resp. $v_{\frac{\sfr+3}{2}},\ldots,v_\sfr$)
when $\sfr:=\text{dim}\,{\ggg}(-\frac{1}{2})_{\bo}$ is even (resp. odd), then $\chi$
vanishes on the derived subalgebra of $\mathfrak{m}$. Define
\begin{equation}\label{defp}
  \mathfrak{p}:=\bigoplus_{i\geq
  0}{\ggg}(i),\quad\mathfrak{m}^{\prime}:=
  \left\{\begin{array}{ll}\mathfrak{m}&\text{if}~\sfr~\text{is
  even;}\\
\mathfrak{m}\oplus {\bbc}v_{\frac{\sfr+1}{2}}&\text{if}~\sfr~\text{is
odd.}\end{array}\right.\end{equation}

As in \cite[\S2]{ZS1} we can choose a basis
$x_1,\ldots,x_l,x_{l+1},\ldots,x_{m'}\in\mathfrak{p}
_{\bz}, y_1,\ldots, y_q, y_{q+1}, \ldots,y_{n'}\in{\mathfrak{p}}_{\bo}$ of the free
$\mathbb{C}$-module ${\mathfrak{p}}$ such that

(a) $x_i\in{\ggg}(k_i)_{\bz}, y_j\in{\ggg}(k'_j)_{\bo}$, where
$k_i,k'_j\in\frac{{\bbz}_+}{2}$ with $1\leq i\leq m'$ and $1\leq j\leq
n'$;

(b) $x_1,\ldots,x_l$ is a basis of $({\ggg}_e)_{\bz}$ and $y_1,\ldots,y_q$ is a
basis of $({\ggg}_e)_{\bo}$;

(c) $x_{l+1},\ldots,x_{m'}\in[f,{\ggg}_{\bz}]$ and $
y_{q+1},\ldots,y_{n'}\in[f,{\ggg}_{\bo}]$.

Recall that a generalized Gelfand-Graev ${\ggg}$-module associated with $\chi$ is defined by
\begin{align}\label{eq: Q chi}
Q_\chi:=U({\ggg})\otimes_{U(\mathfrak{m})}{\bbc}_\chi,
\end{align}
 where ${\bbc}_\chi={\bbc}1_\chi$ is a one-dimensional  $\mathfrak{m}$-module
 such that $x.1_\chi=\chi(x)1_\chi$ for all $x\in\mathfrak{m}$. For
 $k\in\mathbb{Z}_+$, define
 \begin{equation*}
 \begin{array}{llllll}
 \mathbb{Z}_+^k&:=&\{(i_1,\ldots,i_k)\mid i_j\in\mathbb{Z}_+\},&
 \mathbb{Z}_2^k:=&\{(i_1,\ldots,i_k)\mid i_j\in\{0,1\}\}
 \end{array}
 \end{equation*}with $1\leq j\leq k$. For $\mathbf{i}=(i_1,\ldots,i_k)$
 in $\mathbb{Z}_+^k$ or $\mathbb{Z}_2^k$, set $|\mathbf{i}|=i_1+\cdots+i_k$. For any
 real number $a\in\mathbb{R}$, let $\lceil a\rceil$ denote the largest integer
 lower bound of $a$, and $\lfloor a\rfloor$ the least integer upper bound of $a$.
 Given
 $(\mathbf{a},\mathbf{b},\mathbf{c},\mathbf{d})\in{\bbz}^{m'}_+\times\mathbb{Z}_2^{n'}\times{\bbz}^s_+\times\mathbb{Z}_2^t$
 (where
 $t=\lfloor\frac{\sfr}{2}\rfloor=\lfloor\frac{\text{dim}\,{\ggg}(-\frac{1}{2})_{\bo}}{2}\rfloor$),
 let $x^\mathbf{a}y^\mathbf{b}u^\mathbf{c}v^\mathbf{d}$ denote the monomial
 $x_1^{a_1}\cdots x_{m'}^{a_{m'}}y_1^{b_1}\cdots y_{n'}^{b_{n'}}u_1^{c_1}\cdots
 u_s^{c_s}v_1^{d_1}\cdots v_t^{d_t}$ in $U({\ggg})$, 
 {denote by $\text{wt}(x^{\mathbf{a}}y^\mathbf{b}u^\mathbf{c}v^\mathbf{d})=2\Big(\sum\limits_{i=1}^{m'}k_ia_i\Big)+2\Big(\sum\limits_{i=1}^{n'}k'_ib_i\Big)-|\mathbf{c}|-|\mathbf{d}|$\, the weight of $x^{\mathbf{a}}y^\mathbf{b}u^\mathbf{c}v^\mathbf{d}$, and
 set
\begin{equation}\label{e-degree}
|(\mathbf{a},\mathbf{b},\mathbf{c},\mathbf{d})|_e:=2\sum_{i=1}^{m'}a_i(k_i+1)+2\sum_{i=1}^{n'}b_i(k'_i+1)+\sum_{i=1}^sc_i+\sum_{i=1}^td_i.
\end{equation}}
\begin{definition}\label{W-C}
Define the finite $W$-superalgebra over $\mathbb{C}$ by
$$W_\chi:=(\text{End}_{\ggg}Q_{\chi})^{\operatorname{\op}}.$$
\end{definition}

Let $N_\chi$ denote the $\mathbb{Z}_2$-graded ideal of codimension one in
$U({\mmm})$ generated by all $x-\chi(x)$ with $x\in{\mmm}$. Set
\begin{equation}\label{Ichi}
I_\chi:=U({\ggg})N_\chi.
\end{equation}
Then $Q_\chi\cong U({\ggg})/I_\chi$ as ${\ggg}$-modules. The fixed point space
$(U({\ggg})/I_\chi)^{\text{ad}\,{\mmm}}$ carries with  a natural algebra structure
given by $$(x+I_\chi)\cdot(y+I_\chi):=(xy+I_\chi)$$ for all $x,y\in U({\ggg})$ such that $[a,x],[a,y]\in I_\chi$, for all $a\in\mmm$.
By \cite[Theorem 2.12]{ZS}, we have
$(\text{End}_{\ggg}Q_{\chi})^{\operatorname{\op}}\cong Q_{\chi}^{\text{ad}\,{\mmm}}$
as $\mathbb{C}$-algebras.

Let $w_1,\ldots, w_c$ be a basis of $\ggg$ over $\bbc$. Let $U({\ggg})=\bigcup_{i\in{\frac{\bbz}{2}}}U^i({\ggg})$ be a filtration of $U({\ggg})$, where $U^i({\ggg})$ is the ${\bbc}$-span of all $w_{i_1}\cdots w_{i_k}$ with $w_{i_1}\in{\ggg}(j_1),\ldots,w_{i_k}\in{\ggg}(j_k)$ and $(j_1+1)+\cdots+(j_k+1)\leqslant  i$. This filtration is called {\sl Kazhdan filtration}.  The Kazhdan filtration on $Q_{\chi}$ is defined by $Q^i_{\chi}:=\pi(U^i({\ggg}))$ with $\pi:U({\ggg})\twoheadrightarrow U({\ggg})/I_\chi$ being the canonical homomorphism, which makes $Q_{\chi}$ into a filtered $U({\ggg})$-module. By virtue of \eqref{e-degree}, we have that $Q^i_{\chi}$ is the linear combination of all
$x^\mathbf{a}y^\mathbf{b}u^\mathbf{c}v^\mathbf{d}\otimes1_\chi$ with $|(\mathbf{a},\mathbf{b},\mathbf{c},\mathbf{d})|_e\leq 2i$.
Then there is an induced Kazhdan filtration
$W^i_{\chi}$ on the subspace $W_{\chi}=Q_{\chi}^{\text{ad}\,\mmm}$ of $Q_{\chi}$ such that $W^j_{\chi}=0$
unless $j\geqslant0$.
%
%
We see that  $W_{\chi}^{i}\cdot W_{\chi}^{j} \subseteq W_{\chi}^{i+j} $ for any  $i,
j\in\frac{{\bbz}_+}{2}$.
We will write $\grsf'W_{\chi}$ for the graded algebra of $W_{\chi}$ under the Kazhdan grading.

\subsection{}\label{4.2}
Recall that $\{x_1,\ldots,x_l\}$ and $\{y_1,\ldots,y_q\}$ are bases of
${\ggg}^e_\bz$ and ${\ggg}^e_{\bo}$, respectively. Set
\[Y_i:=\left\{
\begin{array}{ll}
x_i&\text{if}~1\leq  i\leq  l;\\
y_{i-l}&\text{if}~l+1\leq  i\leq  l+q;\\
v_{\frac{\textsf{r}+1}{2}}&\text{if}~i=l+q+1 \mbox{ whenever }\textsf{r}\mbox{ is
odd}.
\end{array}
\right.
\]
By assumption it is easily seen that $Y_i\in{\ggg}_e$ for $1\leq  i\leq
l+q$. The term $Y_{l+q+1}\notin{\ggg}_e$ occurs only when
$\textsf{r}=\dim\ggg(-\frac{1}{2})_{\bo}$ is odd. In the subsequent arguments, we
set
$$q':=\begin{cases} q &\mbox{ if }\textsf{r} \mbox{ is even; }\cr
q+1 & \mbox{ if }\textsf{r}\mbox{ is odd}.\end{cases}$$ Assume that the above $Y_i$ belongs to
${\ggg}(m_i)$ for $1\leq  i\leq  l+q'$.
Let $\tilde{\Theta}_i$ denote the image of $\Theta_i\in W_{\chi}$ in
the graded algebra $\grsf'W_{\chi}$ under the Kazhdan grading.

\begin{theorem}(\cite[Theorem 4.5]{ZS}, \cite[Theorem 1.4, Proposition 1.10]{ZS5})\label{pbw}
The following PBW structural statements for the finite $W$-superalgebra
$W_{\chi}$ hold, corresponding to the cases when  $\textsf{r}$  is even and when
$\textsf{r}$ is odd, respectively.
	\begin{itemize}
		\item[(1)] There exist homogeneous elements
$\Theta_1,\ldots,\Theta_{l}\in (W_{\chi})_{\bz}$ and
$\Theta_{l+1},\ldots,\Theta_{l+q'}\in (W_{\chi})_{\bo}$ such that
		\begin{equation}\label{gen}
		\begin{array}{ll}
		&\Theta_k(1_\chi)\\=&(Y_k+\sum\limits_{\mbox{\tiny
$\begin{array}{c}|\mathbf{a},\mathbf{b},\mathbf{c},\mathbf{d}|_e=m_k+1,\\|\mathbf{a}|
				+|\mathbf{b}|+|\mathbf{c}|+|\mathbf{d}|\geq
2\end{array}$}}\lambda^k_{\mathbf{a},\mathbf{b},\mathbf{c},\mathbf{d}}x^{\mathbf{a}}
		
y^{\mathbf{b}}u^{\mathbf{c}}v^{\mathbf{d}}\\+&\sum\limits_{|\mathbf{a},\mathbf{b},\mathbf{c},\mathbf{d}|_e<m_k+1}\lambda^k_{\mathbf{a},\mathbf{b},\mathbf{c},\mathbf{d}}x^{\mathbf{a}}
		y^{\mathbf{b}}u^{\mathbf{c}}v^{\mathbf{d}})\otimes1_\chi
		\end{array}
		\end{equation}
		for $1\leq  k\leq  l+q$, where
$\lambda^k_{\mathbf{a},\mathbf{b},\mathbf{c},\mathbf{d}}\in\mathbb{Q}$, and
$\lambda^k_{\mathbf{a},\mathbf{b},\mathbf{c},\mathbf{d}}=0$ if
$a_{l+1}=\cdots=a_{m'}=b_{q+1}=\cdots=b_{n'}=c_1=\cdots=c_s=
		d_1=\cdots=d_{\lceil\frac{\textsf{r}}{2}\rceil}=0$.
		
		Additionally  set
$\Theta_{l+q+1}(1_\chi)=v_{\frac{\textsf{r}+1}{2}}\otimes1_\chi$  when
$\textsf{r}$ is odd.

		\item[(2)] The monomials
$\Theta_1^{a_1}\cdots\Theta_l^{a_l}\Theta_{l+1}^{b_1}\cdots\Theta_{l+q'}^{b_{q'}}$
with $a_i\in\mathbb{Z}_+, b_j\in\mathbb{Z}_2$ for $1\leq i\leq l$ and
$1\leq j\leq q'$ form a basis of $W_{\chi}$ over $\mathbb{C}$.
		
		\item[(3)] For $1\leq  i\leq  l+q'$,  the elements
$\tilde{\Theta}_i=\Theta_i+W_{\chi}^{m_i+\frac{1}{2}}\in\grsf'W_{\chi}$ are
algebraically independent and generate $\grsf'W_{\chi}$. In particular,
$\grsf'W_{\chi}$ is a graded polynomial superalgebra with homogeneous
generators of degrees $m_1+1,\ldots,m_{l+q'}+1$.
		
		\item[(4)] For  $1\leq  i,j\leq  l+q'$, we have
		$$[\Theta_i,\Theta_j]\in W_\chi^{m_i+m_j+1}.$$
		Moreover, if the elements $Y_i, Y_j\in {\ggg}_e$ for $1\leq
i,j\leq  l+q$ satisfy $[Y_i,Y_j]=\sum\limits_{k=1}^{l+q}\alpha_{ij}^kY_k$
in ${\ggg}_e$, then
		\begin{equation}\label{Thetaa2}
		\begin{array}{llllll}
		&[\Theta_i,\Theta_j]&\equiv&\sum\limits_{k=1}^{l+q}
\alpha_{ij}^k\Theta_k+q_{ij}(\Theta_1,\ldots,\Theta_{l+q'})&(\text{mod }
W_\chi^{m_i+m_j+\frac{1}{2}}),
		\end{array}
		\end{equation}where $q_{ij}$ is a super-polynomial in $l+q'$ variables
in $\mathbb{Q}$ whose constant term and linear part are zero.
		
		\item[(5)] When $\textsf{r}$ is odd, for $1\leq i,j\leq l+q$  we have 
$$[\Theta_i,\Theta_{l+q+1}]=[\Theta_{l+q+1},\Theta_j]=0.$$		
%
		Moreover, if $i=j=l+q+1$, then \begin{equation}\label{idid}
		[\Theta_{l+q+1},\Theta_{l+q+1}]=\sf{id}.
		\end{equation}
	\end{itemize}
\end{theorem}

By virtue of Theorem \ref{pbw}, we see that the Kazhdan degree of the monomial
$\Theta_1^{a_1}\cdots\Theta_l^{a_l}\Theta_{l+1}^{b_1}\cdots\Theta_{l+q'}^{b_{q'}}$ for the PBW basis of $W_\chi$
is $\sum_{i=1}^{l}a_{i}(m_{i}+1)+\sum_{i=1}^{q'}b_{i}(m_{l+i}+1)$. Then it follows from the discussion as in \S\ref{kazhdan} and \eqref{gen} that
\begin{equation}\label{degeTheta}
\text{deg}_e(\Theta_1^{a_1}\cdots\Theta_l^{a_l}\Theta_{l+1}^{b_1}\cdots\Theta_{l+q'}^{b_{q'}})=2\bigg(\sum_{i=1}^{l}a_{i}(m_{i}+1)+\sum_{i=1}^{q'}b_{i}(m_{l+i}+1)\bigg).
\end{equation} Moreover, the corresponding graded algebra
 $\grsf'W_{\chi} = \sum_{k\geq -\frac{1}{2}}W_{\chi}^{k+\frac{1}{2}}/W_{\chi}^{k}$ (we set $W_{\chi}^{-\frac{1}{2}}=\{0\}$ here) is a commutative superalgebra, i.e.,
\begin{theorem}(\cite[Theorem 0.1]{ZS}, \cite[Theorem 1.7]{ZS5}, \cite[Corollary 3.9]{SX})\label{ZW pbw}
For any basic classical Lie superalgebra  $\ggg$, we have
\begin{itemize}
\item[(1)] $\grsf'W_{\chi} \cong S(\ggg_{e})$ when  $\textsf{r}$ is even;
\item[(2)] $\grsf'W_{\chi} \cong S(\ggg_{e})\otimes \mathbb{C}[\Theta]$ when
    $\textsf{r}$ is odd,
 \end{itemize}
where  $S(\ggg_{e})$ denotes the supersymmetric algebra on  $\ggg_{e}$,  and
$\mathbb{C}[\Theta]$  is an external algebra generated by an element  $\Theta$.
\end{theorem}

\begin{remark}\label{pbw1}
(1) It is notable that for the case with $\ggg=\mathfrak{gl}(m|n)$ and a nilpotent element $e\in\ggg_{\bz}$, we can always find an appropriate element $h\in\ggg_{\bz}$ for the $\mathfrak{sl}_2$-triple $\{e,f,h\}$ such that
$\textsf{r}$ is an even number under the Dynkin grading
(see \cite[\S3.2]{WZ} for more
 details).

(2) In the case $\ggg=\gl(m|n)$ with $e\in \ggg_\bz$ being principal nilpotent, the  first statement of Theorem \ref{ZW pbw} was formulated in \cite[Theorem 4.1]{BBG}.
\end{remark}

 The following observation is basic (see for example, \cite[Appendix I, Lemma I.1]{Zeng}). 
\begin{lemma}\label{3.3.2}
Each generator $\Theta_k$ of the finite $W$-superalgebra  $W_{\chi}$ can be
chosen to be a weight vector for $\mathfrak{t}_{e}$ of the same weight as $Y_k$.
\end{lemma}

Consider the linear mapping
\begin{equation}
\begin{split}
\Theta:\quad&\ggg_e\rightarrow W_{\chi}\cr
&x\mapsto \Theta_{x}
\end{split}
\end{equation}
such that $\Theta_{Y_i}:=\Theta_i$ for all $i$. Thanks to Lemma \ref{3.3.2},
$\Theta$ is an injective homomorphism of $\mathfrak{t}_{e}$-module. Although
$\Theta$ is not a Lie superalgebra homomorphism, in general, we have
\begin{lemma}\label{4.6}For $Y_i, Y_j\in {\ggg}_e$ with $1\leq
i,j\leq  l+q$, we have
\begin{equation}\label{comm}
\begin{array}{lllll}
[\Theta_{i},\Theta_{j}]&\equiv&
	\Theta_{[Y_{i}, Y_{j}]}+q_{ij}(\Theta_1,\ldots,\Theta_{l+q'})&(\text{mod }
W_\chi^{m_i+m_j}),
	\end{array}
	\end{equation}    where $q_{ij}$ is a polynomial superalgebra in $l+q'$ variables with
initial form of total degree $\geq2$.
\end{lemma}
\begin{proof}
For any $x\in\mathfrak{g}(i)$, set $\sigma(x)=(-1)^{2i}x$. Since $\sigma$ preserves
$I_\chi$ and $\mathfrak{m}$, it acts on $W_\chi\cong
Q_\chi^{\text{ad}\,\mathfrak{m}}$ as algebra automorphisms.
For any monomial $x^\mathbf{a}y^\mathbf{b}u^\mathbf{c}v^\mathbf{d}\otimes1_\chi$
in $Q_\chi$ with
$(\mathbf{a},\mathbf{b},\mathbf{c},\mathbf{d})\in\mathbb{Z}_+^{m'}\times\mathbb{Z}_2^{n'}\times\mathbb{Z}_+^s\times\mathbb{Z}_2^t$,
note that
\begin{equation}\label{sigma}
\sigma(x^\mathbf{a}y^\mathbf{b}u^\mathbf{c}v^\mathbf{d}\otimes1_\chi)=(-1)^{|(\mathbf{a},\mathbf{b},\mathbf{c},\mathbf{d})|_e}
x^\mathbf{a}y^\mathbf{b}u^\mathbf{c}v^\mathbf{d}\otimes1_\chi.
\end{equation}
Now the lemma follows from Theorem \ref{pbw}(4), Lemma \ref{3.3.2} and \eqref{sigma}.
\end{proof}

\subsection{}	    The second and third authors introduced in \cite[Definition 4.8]{ZS} the refined
$W$-superalgebra
	$$W_{\chi}':= Q_{\chi}^{\ad\mmm'},$$ which is a subalgebra of the finite
$W$-superalgebra $W_{\chi}$.
When $\sfr$ is even, since ${\mmm}'={\mmm}$ by definition, we have
$W_\chi=W'_\chi$. However, the situation changes in the case when $\sfr$ is
odd. Since ${\mmm}$ is a proper subalgebra of ${\mmm}'$, it follows that
$W'_\chi$ is a subalgebra of $W_\chi$.
In fact, one can formulate the PBW theorem of $W_\chi'$, which is the same as  Theorem \ref{pbw} with $\sfr$ being odd, just abandoning the related
topics on the element $\Theta_{l+q+1}$.   It is  also shown that $\grsf'W_\chi'\cong
S(\ggg_e)$ as $\bbc$-algebras under the Kazhdan grading in \cite[Corollary
3.8]{ZS1}. We refer to \cite[Theorem 3.7]{ZS1} for
more details.

Now we are in a position to introduce the main results of this section.
\begin{theorem}\label{3.3.4}
There exists an associative $\mathbb{C}[t]$-superalgebra $\mathcal{H}_{\chi}$
free as a module over  $\mathbb{C}[t]$
	 such that
	$$\mathcal{H}_{\chi}/(t-\lambda)\mathcal{H}_{\chi}\cong
\left\{\begin{array}{lcl}
	W'_{\chi} & \mbox{if} & \lambda\neq 0;
	\\U(\ggg_e) &\mbox{if} & \lambda = 0.	
	\end{array}
	\right.$$
In other words, the enveloping algebra $U(\ggg_e)$ is a contraction of $W'_{\chi}$.
\end{theorem}

\begin{proof}   The strategy of our arguments is the same as in \cite[Proposition 2.1]{P2}. For the readers' convenience, we write the arguments out here.
 	Consider the superalgebra $H(R)=R\otimes W'_{\chi}$ over the ring of Laurent
polynomials  $R=\mathbb{C}[t,t^{-1}]$ obtained from  $W'_{\chi}$ by extension of
scalars, and identify  $W'_{\chi}$ with the subspace $\mathbb{C}\otimes W'_\chi$ of hyperspace $H(R)$. Set the elements in $R$ to be even. Define an invertible
$R$-linear transformation $\pi$ on $H(R)$ by setting
	
\begin{align}\label{keyeq}
\pi(\Theta_1^{a_1}\cdots\Theta_l^{a_l}\Theta_{l+1}^{b_1}
\cdots\Theta_{l+q}^{b_{q}})
=t^{2(a_1m_1+
\cdots+a_lm_l+b_1m_{l+1}+\cdots+b_qm_{l+q})}
\Theta_1^{a_1}\cdots\Theta_l^{a_l}\Theta_{l+1}^{b_1}\cdots\Theta_{l+q}^{b_{q}}
\end{align}
and extending to $H(R)$ by  $R$-linearity. We view $\pi$  as an isomorphism from
$H(R)$ onto a
	new $R$-superalgebra $H(R,\pi)$ with underlying $R$-module  $R\otimes
	W'_\chi$  and with associative product
	given by  $(x\cdot y)_\pi:=\,
	\pi^{-1}\big(\pi(x)\cdot\pi(y)\big)$, for all $x,y\in R\otimes W'_\chi$. We
denote by  ${ \mathcal H}_\chi$ the free
	 $\mathbb{C}[t]$-submodule of $H(R, \pi)$ generated by
$\Theta_1^{a_1}\cdots\Theta_l^{a_l}\Theta_{l+1}^{b_1}
\cdots\Theta_{l+q}^{b_{q}}$, $\forall\,
	a_i\in\mathbb{Z}_+, b_j\in\mathbb{Z}_{2}$. Recall in \eqref{degeTheta} we have
$$\deg_e\big(\Theta_1^{a_1}\cdots\Theta_l^{a_l}\Theta_{l+1}^{b_1}\cdots\Theta_{l+q}^{b_{q}}\big
)\,=\,
	2\bigg(\sum_{i=1}^l a_{i}m_{i}+\sum_{i=1}^q b_{i}m_{l+i}+\sum_{i=1}^l a_i+
\sum_{i=1}^q b_{i}\bigg).$$

Note that the initial form of
$q_{ij}$  has total degree  $\ge 2$  and  $\deg_e
	q_{ij}(\Theta_1,\ldots,\Theta_{l+q'}) =2(m_i+m_j+1)$ if $q_{ij}\neq0$.
By virtue of \eqref{comm} this yields
$$(\Theta_i\cdot\Theta_j-(-1)^{|\Theta_i||\Theta_j|}\Theta_j\cdot\Theta_i)_\pi\,=\,\pi^{-1}\big(t^{2(m_i+m_j)}
	[\Theta_i,\Theta_j]\big)\,\equiv\,\Theta_{[Y_i,Y_j]}\ \, \,
	\big({\text{mod}}\, \, t{\mathcal H}_{\chi}\big)$$
	    By induction on the Kazhdan degrees of
$\Theta_1^{a_1}\cdots\Theta_l^{a_l}\Theta_{l+1}^{ b_1}\cdots\Theta_{l+q}^{b_{q}}$
and by taking it into an account that
	  $\grsf'W'_\chi$ is a commutative superalgebra, we have
	  $(\Theta_i\cdot {\mathcal H}_\chi)_\pi
	\subseteq{\mathcal H}_\chi$  for $1\leq i\leq l+q$. So ${\mathcal H}_\chi$ is a
$\mathbb{ C}[t]$-subalgebra of
	 $H(R,\pi)$.
	
	If   $\lambda\ne 0$, then the homomorphism   $\bbc[t]\rightarrow\bbc$ taking
$t$ to $\lambda$ extends to a homomorphism $R\rightarrow\bbc$. The isomorphism
$\pi^{-1}$ injects   $(t-\lambda)H(R,\pi)$  onto
	$(t-\lambda)H(R)$. Because  ${\mathcal H}_\chi\cap
	(t-\lambda)H(R,\pi)=(t-\lambda){\mathcal H}_\chi$,
	$W'_\chi\cap(t-\lambda)H(R)=0$, we have
	$${\mathcal H}_\chi/(t-\lambda){\mathcal H}_\chi\cong
H(R,\pi)/(t-\lambda)H(R,\pi)
	\cong H(R)/(t-\lambda)H(R)\cong W'_\chi,$$by the theorem on isomorphism.

Now put $\bar{\mathcal H}_\chi:={\mathcal
		H}_\chi/t{\mathcal H}_\chi$, and identify the generators
$\Theta_i$ of   ${\mathcal H}_\chi$ with their images in $
\bar{\mathcal H}_\chi$. It follows from our earlier discussion that these
images satisfy the relations
	$[\Theta_{i},\Theta_{j}]=\Theta_{[Y_i,Y_j]}$
	for $1\leq i,j\leq l+q$.
By the universality
property of the enveloping superalgebra   $U(\ggg_e)$, there is an algebra
homomorphism  $\phi\colon\,U(\ggg_e)
	\twoheadrightarrow \bar{\mathcal H}_\chi$ with  $\phi(Y_i)=\Theta_i$
	for $1\leq i\leq l+q$.
Since   ${\mathcal H}_\chi$  is a free  $\mathbb{C}[t]$-module, the monomials
$\Theta_1^{a_1}\cdots\Theta_l^{a_l}\Theta_{l+1}^{b_1}\cdots\Theta_{l+q}^{b_{q}}$
are linearly independent in  $\bar{\mathcal H}_\chi$.  As a consequence,
	$\phi$ is an isomorphism.
\end{proof}

\begin{theorem}\label{cdd1}The following are true:
\begin{itemize}
\item[(1)] The graded algebra of the refined $W$-superalgebra is isomorphic to the universal enveloping
superalgebra of $\ggg_e$ under the Dynkin grading, i.e.,
\[\begin{array}{lcll}
\grsf:&\grsf W_\chi' &\rightarrow&U(\ggg_e)\\ &\grsf(\Theta_{i})&\mapsto&Y_{i}
\end{array}
\]for $1\leq i\leq l+q$.
\item[(2)] When $\textsf{r}$ is even, then $\grsf W_\chi\cong U(\ggg_e)$.
\end{itemize}
\end{theorem}
\begin{proof} The part (2) is a direct consequence of (1) because we have $W_\chi=W_\chi'$ when $\textsf{r}$ is even.

Now we consider the part (1). Pick a homogeneous basis $Y_1,\dots,Y_{l+q}$ for $\ggg_e$ as in \S\ref{4.2}, so
that the monomials $\Theta_1^{a_1}\cdots\Theta_l^{a_l}\cdot\Theta_{l+1}^{b_1}\cdots\Theta_{l+q}^{b_{q}}$
with $a_i\in\mathbb{Z}_+, b_j\in\mathbb{Z}_2$ for $1\leq i\leq l$ and
$1\leq j\leq q$ form a basis for $W_\chi'$. By virtue of \eqref{gen} and Lemma \ref{4.6}, one can easily conclude that \begin{equation}\label{natural}
\Theta_1^{a_1}\cdots\Theta_l^{a_l}\cdot\Theta_{l+1}^{b_1}\cdots\Theta_{l+q}^{b_{q}}
=Y_1^{a_1}\cdots Y_l^{a_l}\cdot Y_{l+1}^{b_1}\cdots Y_{l+q}^{b_{q}}+(\dag),
\end{equation}
where the first term on the right-hand side of \eqref{natural} lies in
$U(\mathfrak{p})(\sum_{i=1}^l a_{i}m_{i}+\sum_{i=1}^q b_{i}m_{l+i})$ and
the term $(\dag)$ lies in the sum of all strictly lower graded components. Hence the
monomials $Y_1^{a_1}\cdots Y_l^{a_l}\cdot Y_{l+1}^{b_1}\cdots Y_{l+q}^{b_{q}}$ constitute a homogeneous basis for $\grsf W_\chi'$. The same monomials constitute a
homogeneous basis for $U(\ggg_e)$ by the PBW theorem.	
\end{proof}

\begin{remark} In the case $\ggg=\gl(m|n)$ with $e\in\ggg_\bz$ being a principal nilpotent element,  the above theorem was formulated in \cite[Remark 4.6]{BBG}. In that case, $U(\ggg_e)$ can be realized precisely as a subalgebra of $U(\hhh)$ (see Remark \ref{W realization}(3) for more details).
\end{remark}

\section{Deformed action on $V^{\otimes d}$: as a module over finite
$W$-superalgebras}\label{sec: skryabin}

In this section we will discuss the filtered deformation of Vust's double
centralizer
property in the super case. To achieve this we need another formation of finite $W$-superalgebras.
\subsection{Pyramids and finite  $W$-superalgebras}\label{map for W}
 Now we turn to the case $\ggg=\gl(m|n)$.  Consider the principal nilpotent orbit with the standard representative  element
 \begin{align}\label{eq: e's expression}
 e=\sum_{i=1}^{m-1}e_{\bar i,
\overline{i+1}}+\sum_{j=1}^{n-1}e_{j,j+1}\in \ggg_\bz.
\end{align}
We have a pyramid associated with the pair $(m\leq n)$  as in (\ref{pyramid}).
Then $\ggg=\bigoplus_{r\in \bbz}\ggg(r)$ by declaring that $e_{i,j}$ is of degree
$$\text{deg}(e_{i,j})=\col(j)-\col(i).$$
{The above grading is a Dynkin grading on $\ggg$ under the adjoint action of $$h=-2\text{diag}(\text{col}(\bar1),\text{col}(\bar2),\ldots,\text{col}(\overline m),\text{col}(1),\text{col}(2),\ldots,\text{col}(n)),$$ such that ${\ggg}(r)=\{x\in{\ggg}\mid[h,x]=2rx\}$ with $r\in\bbz$.}
In this case, we take
\begin{equation}\label{eq: p h m}
\ppp:=\bigoplus_{r\geq 0}\ggg(r)=\bigoplus_{\col(i)\leq\col(j)}\bbc e_{i,j};\;\; \hhh:=\ggg(0)=\bigoplus_{\col(i)=\col(j)}\bbc e_{i,j};\;\; \mmm:=\bigoplus_{r<0}\ggg(r)=\bigoplus_{\col(i)>\col(j)}\bbc e_{i,j},
\end{equation}
Then the generalized Gelfand-Graev ${\ggg}$-module $Q_\chi$ and the finite $W$-superalgebra $W_\chi$ can be defined as in \eqref{eq: Q chi} and Definition \ref{W-C}, respectively.

Note that we have $U({\ggg})=U({\ppp})\oplus I_\chi$ by definition.
Let $\text{Pr}:
U({\ggg})\longrightarrow U({\mathfrak{p}})$ denote the corresponding linear
projection.  Then we have
\begin{lemma} (\cite[\S4]{BBG} or \cite[Remark 2.14]{ZS})\label{isoWchi} Keep the notations and assumptions as above. Then we
can identify $W_\chi$ with the following subalgebra in $U({\mathfrak{p}})$:
\begin{equation}\label{wchi}
\{u\in U({\mathfrak{p}})\mid\text{Pr}([x,u])=0~\text{for
any}~x\in{\mathfrak{m}}\}.
\end{equation}
Furthermore, with the above identification there is an isomorphism of $\mathbb{C}$-algebras
\begin{align}\label{thethirdeqdef}\varphi: W_\chi&\longrightarrow
Q_\chi^{\text{ad}\,{\mathfrak{m}}}\qquad u\mapsto u(1+I_\chi).
\end{align}
\end{lemma}

\begin{remark}\label{W realization}
(1) Up to isomorphisms, the finite $W$-superalgebra $W_\chi$ is only dependent on the set $\{m,n\}$ (see \cite[Remark 4.8]{BBG}).

(2) Here the definition is subjected to the condition ``$\text{Pr}\,([\mmm,u])=0$" with left-multiplication by $\mmm$,  which can be regarded as a ``left-handed" formulation. There is another so-called ``right-handed" formulation. Both of the definitions are equivalent because the two finite $W$-superalgebras defined in different ways are isomorphic as algebras (see \cite[Remark 4.7]{BBG}).

(3) Furthermore, $W_\chi$ can be precisely realized as a subalgebra of $U(\ppp)$ isomorphic to $U(\hhh)$ as below.  Associated with the Pyramid \eqref{pyramid}, one can define an automorphism $\eta^0$ of $U(\ppp)$ via the shift
$$ e_{i,j}\mapsto e_{i,j}-\delta_{i,j}d_i$$
for each $e_{i,j}\in \ppp$, where \
\begin{align*}
d_i=\begin{cases} 1,\;\;\;&\text{ if } \bar1\leq i\leq \overline{m};\cr -1,\;\;\;&\text{ if } 1\leq i\leq m;\cr
m-i,  \;\;\;&\text{ if } m<i\leq n.
\end{cases}
\end{align*}
Define a so-called Miura transform $\mu^0$ which is the composite of $\eta^0|_{W_\chi}$ and the projection of $U(\ppp)$ onto $U(\hhh)$ arising from the canonical projection $\ppp\twoheadrightarrow \hhh$. Then $\mu^0$ gives rise to an injection from $W_\chi$ into $U(\hhh)$ (cf.  \cite{BBG}).  However,  we are not going to take $\eta^0$ into our arguments  in the next section for  simplicity.
\end{remark}

\subsection{}\label{sec: bc and filtered} From this section on, {\sl{we will identify $W_\chi$ with the subalgebra of $U(\ppp)$ as described in \eqref{wchi}.}}

By \cite[Proposition 4.7]{ZS4}, there exists a one-dimensional module over $W_\chi$, which we denote by $\bbc_\bc$. Under the identification of $W_\chi$ as in Lemma \ref{isoWchi}, we can extend the one-dimensional $W_\chi$-module $\bbc_\bc$ to a $U(\ppp)$-module. In the following we will describe $U(\ppp)$-module $\bbc_\bc$ precisely.

From now on, fix $\bc = (c_1,c_2,\ldots,c_n)\in \bbc^n$. Let us introduce a one-dimensional $\ppp$-module $\bbc_\bc=\bbc 1_\bc$ by defining the action for any $e_{i,j}\in \ppp$ on the generator $1_\bc$ as
\begin{equation*}\label{def: bc}
e_{i,j} 1_\bc=\delta_{i,j}(-1)^{|i|} c_{\text{col}(i)}1_\bc.
\end{equation*}
Furthermore, we consider the mapping
\begin{equation*}\label{eq: eta bc}
\eta_\bc: U(\ppp)\rightarrow U(\ppp) \mbox{  with }e_{i,j}\mapsto e_{i,j}+\delta_{i,j}(-1)^{|i|}c_{\text{col}(i)}
\end{equation*}
for each $e_{i,j}\in \ppp$, which is an algebra automorphism. To be explicit, for any $e_{i,j}, e_{k,l}\in \ppp$   we have that
\begin{align}\label{etac}
\eta_\bc([e_{i,j},e_{k,l}])&=\eta_\bc(\delta_{j,k}e_{i,l}-(-1)^{(|i|+|j|)(|k|+|l|)}\delta_{i,l}e_{k,j})\cr
&=\delta_{j,k}e_{i,l}-(-1)^{(|i|+|j|)(|k|+|l|)}\delta_{i,l}e_{k,j}+\delta_{j,k}\delta_{i,l}(-1)^{|i|}c_{\text{col}(i)}\cr
&-\delta_{i,l}\delta_{j,k}(-1)^{(|i|+|j|)(|k|+|l|)+|k|}c_{\text{col}(k)}\cr
&=[\eta_\bc(e_{i,j}),\eta_\bc(e_{k,l})]+\delta_{i,l}\delta_{j,k}((-1)^{|i|}c_{\text{col}(i)}-(-1)^{(|i|+|j|)(|k|+|l|)+|k|}c_{\text{col}(k)}).
\end{align}
Now we consider the indices $i,j,k,l\in I$ in \eqref{etac}. On one hand, $e_{i,j}, e_{k,l}\in \ppp$
if and only if $\text{col}(i)\leq\text{col}(j)$ and $\text{col}(k)\leq\text{col}(l)$ by definition. On the other hand,  $\delta_{i,l}\delta_{j,k}\neq0$ entails that $\text{col}(i)=\text{col}(l)$ and $\text{col}(j)=\text{col}(k)$.
Therefore, for the last term $\delta_{i,l}\delta_{j,k}((-1)^{|i|}c_{\text{col}(i)}-(-1)^{(|i|+|j|)(|k|+|l|)+|k|}c_{\text{col}(k)})$ as in \eqref{etac}, we just need to consider the situation when $\text{col}(i)=\text{col}(j)=\text{col}(k)=\text{col}(l)$, and also $i=l, j=k$.  
Such a term is equal to $(-1)^{|i|}c_{\text{col}(i)}-(-1)^{(|i|+|j|)(|j|+|i|)+|j|}c_{\text{col}(j)}
=(-1)^{|i|}(c_{\text{col}(i)}-c_{\text{col}(j)})=0$. From all above we conclude  that $\eta_\bc([e_{i,j},e_{k,l}])
=[\eta_\bc(e_{i,j}),\eta_\bc(e_{k,l})]$ for any $i,j,k,l\in I$, accomplishing the proof.

Denote by  $V_\bc^{\otimes d}$ the graded $U(\ppp)$-module which is equal to $V^{\otimes d}$ as a graded vector space, and endowed with the action obtained  by  twisting the natural one by the automorphism  $\eta_\bc$. This is to say, $u\cdot v=\eta_{\bc}(u)v$  for $u\in U(\ppp)$ and $v\in V^{\otimes d}$.

Obviously, we can identify $\bbc_\bc\otimes V^{\otimes d}$ with $V_\bc^{\otimes d}$ so that $1_\bc\otimes v=v$ for any $v\in V^{\otimes d}$.
The above $\bbc_\bc$ and $V_\bc^{\otimes d}$ can be naturally considered as $W_\chi$-modules by restriction.

It follows from Lemma \ref{lem: basis} that $V_\bc^{\otimes d}$ becomes a $\ggg_e$-module by restriction. So we have the following tensor representation of $U(\ggg_e)$ on $V_\bc^{\otimes d}$  with
\begin{align}\label{def: phi d c}
\phi_{d,\bc}: U(\ggg_e)\rightarrow \End_\bbc (V_\bc^{\otimes d}),
\end{align}
which is the composition of $\phi_d$ (see \S\ref{2.5}
) and $\eta_\bc$. Note that the image of $\phi_{d,\bc}$ is  the same as the image of $\phi_d$. So  the super Vust theorem  is still true for the case with $V^{\otimes d}$ being replaced by $V_\bc^{\otimes d}$.
Correspondingly, we have the following representation of
$W_\chi$ with
\begin{align}\label{def: bigphi d c}
\Phi_{d,\bc}: W_\chi\rightarrow  \End_\bbc(V_\bc^{\otimes d}).
\end{align}

Recall that the grading on  $V$ is defined  by setting  $v_i$ to  be of degree
$(n-\text{col}(i))$ (see \S \ref{2.5}). Then both  $V_\bc^{\otimes d}$  and
$\text{End}_{\mathbb{C}}(V_\bc^{\otimes d})$ also have gradations. Therefore, $\text
{End}_{\mathbb{C}}(V_\bc^{\otimes d})$ can be considered as a filtered superalgebra
via
\begin{align}\label{5.1}
\operatorname{F}_r \text{End}_{\mathbb{C}}(V_\bc^{\otimes d}) = \bigoplus_{s \leq
	r}\text{End}_{\mathbb{C}}(V_\bc^{\otimes d})_s,
\end{align}
and then the homomorphism $\Phi_{d,\bc}$  is a homomorphism of filtered
superalgebras. Moreover,
if we identify the associated graded algebra  $\grsf\operatorname{End}_{\mathbb C}(V_\bc^{\otimes d}) $ with
$\operatorname{End}_{\mathbb C}(V_\bc^{\otimes d})$
in the obvious way, Theorem \ref{cdd1} along with Remark \ref{pbw1} shows that the associated graded map
$$\grsf\Phi_{d,\bc}:\grsf W_{\chi}\rightarrow
\operatorname{End}_{\mathbb C}(V_\bc^{\otimes d})$$
coincides with the map
$\phi_{d}$ as in \S\ref{2.5}.

\subsection{Whittaker functor and  Skryabin's equivalence}
In this subsection,   we turn back to the case of $\ggg$ being a general basic classical Lie superalgebra as in \S\ref{sec 5}. We will recall the connection between the super module category of the finite $W$-superalgebra $W_\chi$ associated with $\ggg$ and a nilpotent element in $\ggg_\bz$ and  the Whittaker super module category of $\ggg$. As usual, we make an appointment that for each category, the homomorphisms of the objects are even. In the subsequent arguments, supermodules will be simply called modules, and the notations in \S\ref{sec 5} are maintained.

\begin{definition}\label{Whittaker}
	A  ${\ggg}$-module $L$  is called  a Whittaker  module, if for any $a\in{\mathfrak{m}}$, the element $a-\chi(a)$ acts on $L$ locally nilpotently. A vector  $v\in L$  is called a Whittaker vector if $\forall a\in{\mathfrak{m}}$, we have $(a-\chi(a))v=0$.

\end{definition}

Let ${\ggg}\text{-}W\text{mod}^\chi$  denote the category of all the finitely generated
Whittaker ${\ggg}$-module. 
For any $L\in{\ggg}\text{-}W\text{mod}^\chi$,
let
$\text{Wh(L)}=\{v\in L\mid (a-\chi(a))v=0,\forall a\in{\mathfrak{m}}\}$ be
spanned by the  Whittaker  vectors in  $L$.
%
%
Recall that $W_\chi\cong(U({\ggg})/I_\chi)^{\text{ad}\,{\mmm}}$, and denote by $\bar{y}\in U({\ggg})/I_\chi$ the coset associated with $y\in U({\ggg})$. Then we have
\begin{theorem}(\cite[Theorem 2.16]{ZS})
	\begin{itemize}
\item[(1)] Given a Whittaker ${\ggg}$-module $L$ with an action map $\rho$, $\text{Wh}(L)$ is naturally a $W_\chi$-module by letting $$\bar{y}.v=\rho(y)v$$ for $v\in\text{Wh}(L)$ and $\bar{y}\in U({\ggg})/I_\chi$.
\item[(2)]	 For $M\in W_\chi\text{-mod}$, $Q_\chi\otimes_{W_\chi}M$ is a Whittaker ${\ggg}$-module by letting $$y.(q\otimes v)=(y.q)\otimes v$$ for $y\in U({\ggg})$ and $q\in Q_\chi,~v\in M$.\end{itemize}
\end{theorem}Moreover, the following theorem shows that there exist  category equivalences between
  ${\ggg}\text{-}W\text{mod}^\chi$  and  $W_\chi\text{-mod}$, which is the super version of  Skryabin's equivalence for finite $W$-algebras as in \cite{Skr}.
\begin{theorem}(\cite[Theorem 2.17]{ZS})\label{skr}
	The functor $Q_\chi\otimes_{W_\chi}-:W_\chi\text{-mod}\longrightarrow
	{\ggg}\text{-}W\text{mod}^\chi$  is a category equivalence, with
$\text{Wh}:{\ggg}\text{-}W\text{mod}^\chi\longrightarrow W_\chi\text{-mod}$ being
its quasi-inverse.
\end{theorem}
\subsection{Functor $-\circledast X$}
From now on, we will be concerned with $\ggg=\gl(m|n)$ and keep the same notations as in \S\ref{sec 3} and \S\ref{sec 5}.

As an analogue of the Lie algebra case (see \cite[\S8]{BKl2}),  associated with a finite-dimensional
$\ggg$-module $X$, in the following we will introduce a functor from the category of $W_\chi$-modules to
itself with$$-\circledast X: W_\chi\text{-mod}
\rightarrow W_\chi\text{-mod}.$$
Set $1_\chi$ to be the image of $1$ in $Q_\chi$ where $Q_\chi$  is identified with $U(\ggg)\slash I_\chi$, and 
define  the dot action of  $u \in U(\mathfrak{p})$  on  $Q_\chi$ by
$u\cdot u'1_\chi:= \eta_\bc(u)u'1_\chi$ for all $u'\in U(\ggg)$.
Let  $\{z'_1,\ldots, z'_r\}$  and  $\{z'_{r+1},\ldots, z'_\iota\}$  be the bases of the
even part and odd part of $\mmm$, respectively. Further assume that
$z'_i\in\ggg(-d_i)$ for $1 \leq i \leq \iota$. Then the elements  $[z'_i,e]$ are
linearly independent, and  $[z'_i,e]\in\ggg(1-d_i)$. There exist elements
$z_1,\ldots, z_r$  and  $z_{r+1},\ldots,
z_\iota\in \mathfrak{p}$ such that   $z_i\in\ggg(d_i-1)$ with
\[
(z'_i,[z_j,
e])=([z'_i,z_j], e)=\delta_{i,j}.
\]

\begin{lemma}\label{lem:QchifreeWchi}
	
	The right  $W_\chi\text{-}$module $Q_\chi$ is free with a basis	
	\[
	\{z_1^{i_1} \cdots
	z_r^{i_r}z_{r+1}^{\epsilon_{r+1}}\cdots z_\iota^{\epsilon_\iota}\cdot 1_\chi
	\vert\; i_1,\ldots, i_r \in{\bbz}_+ \text{ and } \epsilon_{r+1},\ldots,
	\epsilon_\iota \in \mathbb Z_2 \}.
	\]
\end{lemma}
\begin{proof}	
The lemma can be proved in the same manner as in \cite{Skr} for the Lie algebra case. So we omit the details here.
\end{proof}

It follows from Lemma  \ref{lem:QchifreeWchi} that there is a unique right
$W_\chi$-module homomorphism
$p: Q_\chi \twoheadrightarrow
W_\chi$ with
\begin{align}\label{eq: def of p}
p(z_1^{i_1} \cdots
z_r^{i_r}z_{r+1}^{\epsilon_{r+1}}\cdots z_\iota^{\epsilon_\iota}\cdot
1_\chi)=\delta_{i_1,0}\cdots
\delta_{i_r,0}\delta_{\epsilon_{r+1},0}\cdots \delta_{\epsilon_\iota,0}
\end{align}
for all  $i_1,\ldots, i_r \in\bbz_+$, $\epsilon_{r+1},\ldots, \epsilon_\iota
\in \mathbb Z_2$. In particular, $p(1_\chi)=1$. Also note that $p$ is an even map.

\subsection{Tensor identities}
Now let $X$ be a finite-dimensional $\ggg$-module with a fixed basis  $w_1,\ldots, w_{r_0}\in X_{\bz}$, $w_{r_0+ 1},\ldots, w_{t}\in X_{\bo}$. For any $u\in
U(\ggg)$, define the coefficient functions  $c_{i,j}\in U(\ggg)^*$ by the
equality
\begin{align}\label{eq: def of func c}
u\cdot w_j =\sum_{i=1}^t c_{i,j}(u)w_{i}.
\end{align}

Given any $M\in{\ggg}\text{-}W\text{mod}^\chi$, 
it is clear that the usual $\ggg$-module tensor product $M \otimes X$
also belongs to the category $\ggg\text{-}W\text{mod}^\chi$. This is because $M$ is already assumed to be in $\ggg\text{-}W\text{mod}^\chi$, by definition for any $a \in \mmm$, the element $(a-\chi(a))$ acts  on $M$ locally nilpotent. On the other hand, $X$ is a finite-dimensional $\ggg$-module, then the nilpotent Lie algebra $\mmm$ acts on $X$ nilpotently in the sense of twisting through the automorphism $U(\mmm)\rightarrow U(\mmm)$ with $e_{ij}\mapsto e_{ij}-\chi(e_{ij})$. Hence  $M \otimes X$
belongs to the category $\ggg\text{-}W\text{mod}^\chi$.

%
%

With aid of the Skryabin's
equivalence, we define a functor of the category $W_\chi$-mod 
to itself
as below:
for a  $W_\chi$-module $M$, set $M \circledast-$ by
\begin{align}\label{star operator}
(M \circledast-)(X)=M \circledast X:= \text{Wh}((Q_\chi \otimes_{W_{\chi}}M)\otimes X),
\end{align}
  thereafter whose image will be directly denoted by $M\circledast X$.  This is an exact functor in $W_\chi$-mod.

 In the following, we introduce an important result which is a super version of \cite[Theorem 8.1]{BKl2} which will be used later.
  For this, firstly we fix an even right $W_\chi$-module homomorphism $p: Q_\chi\twoheadrightarrow W_\chi$ with $p(1_\chi)=1$ (see  \eqref{eq: def of p}).   By Skryabin's equivalence (Theorem \ref{skr}), the same  arguments as in \cite[Theorem 8.1]{BKl2} entail that the restriction
	of the map $(Q_\chi \otimes_{W_{\chi}}M)\otimes X \rightarrow M
	\otimes X$ sending  $(u1_\chi \otimes m)\otimes w$ to $p(u1_\chi)m
	\otimes w$  defines a natural even isomorphism of vector superspaces
	\[
	\chi_{M,X}: M\circledast X \rightarrow M \otimes X.
	\]
 When considering the regular $W_\chi$-module $M$ which is $W_\chi$ itself, the inverse image of $1\otimes w_j$ under the isomorphism $\chi_{W_\chi,X}$ can be written as
 \begin{align}\label{eq: inverse image}
 \chi_{W_\chi,X}^{-1}(1\otimes w_j)=\sum_{i=1}^t(x_{i,j}\cdot 1_\chi\otimes 1)\otimes w_i \;\mbox{ for unique elements } x_{i,j}\in U(\ppp).
 \end{align}
Since the map $\chi_{W_\chi,X}$ is even, all $x_{i,j}$ are even elements.
Furthermore, for $x\in \mathfrak{m}$ we have
	\begin{align}\label{eq: the second cond}
	&(x-\chi(x))\cdot \sum_{i=1}^{t}(x_{i,j}\cdot 1_{\chi}\otimes 1)\otimes w_{i}\cr
=&
\sum_{i=1}^{ t}([x, x_{i,j}]\cdot 1_{\chi}\otimes 1)\otimes w_{i}
	+\sum_{i=1}^{t} (x_{i,j}\cdot1_\chi\otimes 1)\otimes
x\cdot
	w_{i} \cr
{\overset{\eqref{eq: def of func c}}{=}}&\sum_{i=1}^{t}\Big([x, x_{i,j}]+\sum_{s=1}^{t}
c_{i,s}(x)x_{s,j}\Big)\cdot 1_{\chi}\otimes 1\otimes w_{i}.
	\end{align}
Note that $\chi_{W_\chi,X}^{-1}(1\otimes w_j)$ lies in $\text{Wh}((Q_\chi\otimes_{W_{\chi}} 1)\otimes X)$.  By definition, \eqref{eq: inverse image} and \eqref{eq: the second cond} give rise to the following inclusion
 $$[x, \eta_\bc(x_{i,j})] + \sum_{s=1}^{t} c_{i,s}(x)\eta_\bc(x_{s,j}) \in I_\chi.$$
  Taking the parities into account as above, by the same arguments as in the proof of \cite[Theorem 8.1]{BKl2} we have the following result.

\begin{theorem}\label{thm:chi-isom}
Let $M$ be any left  $W_\chi$-module, and $X$ be a finite-dimensional
$\ggg$-module  as above,  with super dimension $\underline{\dim}X=(r_0|t-r_0)$. For the natural isomorphism of  superspaces
	\[
	\chi_{M,X}: M\circledast X \rightarrow M \otimes X,
	\]
	its inverse map of $\chi_{M,X}$ can be described  by mapping
	 $m \otimes w_j$  to
	$\sum_{i=1}^{t}
	(x_{i,j} \cdot 1_\chi \otimes m ) \otimes
	w_i$, where   $(x_{i,j})_{t\times t}$ is an invertible supermatrix with entries
$x_{i,j}\in U(\mathfrak{p})$ being uniquely determined by the properties:
	\begin{itemize}
\item[(i)]  $|x_{i,j}|=\bz$ and $p(x_{i,j}\cdot 1_\chi)=\delta_{i,j}$.
\item[(ii)]	 $[x, \eta_\bc(x_{i,j})] + \sum_{s=1}^{t} c_{i,s}(x)\eta_\bc(x_{s,j}) \in I_\chi$ for all $x\in \mathfrak{m}$ and all $i,j$.
	\end{itemize}
\end{theorem}

\begin{proof} With the arguments prior to the theorem, what remains to do is to repeat verbatim the proof of \cite[Theorem 8.1]{BKl2}. In particular, all $x_{i,j}$ are even because $\chi\in\ggg^*$ is even, and the isomorphism $\chi_{M,X}$ is also even.
\end{proof}

Then we have the following corollary.

\begin{corollary}\label{cor:mu-isom} Keep the notations as above. Additionally, suppose $M$ is  a $\ppp$-module,  and $X$ is a finite-dimensional
$\ggg$-module as defined above.
For any $u \in U(\mathfrak{p})$, $m\in M$ and
$w\in X$, the restriction of the map
\[\begin{array}{lcll}
\hat{\mu}_{M,X}:&(Q_\chi \otimes_ {W_\chi}M)\otimes X&\rightarrow& M\otimes
	X\\&(u\cdot 1_\chi \otimes m )\otimes w&\mapsto&um \otimes
	w
\end{array}
\]
defines an isomorphism $\mu_{M,X}$ of  $W_\chi$-modules by
	\[
	\mu_{M,X}: M\circledast X \cong M \otimes X.
	\]
Here, the $U(\ppp)$-modules $M$ and  $M\otimes X$ are regarded the left and right hand sides as $W_\chi$-modules by restriction.  The inverse map sends  $m \otimes w_k$  to $\sum_{i,j=1}^t(x_{i,j} \cdot 1_\chi \otimes
	y_{j,k}m)\otimes w_i$, where
	$(y_{i,j})_{t\times t}$ is the inverse matrix of  $(x_{i,j})_{t\times t}$ as defined in Theorem
\ref{thm:chi-isom}.
\end{corollary}
\begin{proof}
	The proof is similar to that in \cite[Corollary 8.2]{BKl2}. 
\end{proof}

    The composition of the functor  $\circledast$  can be defined  in the same way as in \cite[(8.8)--(8.10)]{BKl2}. In particular, for a finite-dimensional $\ggg$-module
$Y$, there is a
natural isomorphism:
\begin{align}\label{aVV}
\alpha_{M, X,Y}: (M\circledast X)\circledast Y \cong M \circledast (X \otimes
Y).
\end{align}
One can further define $M\circledast Y^{\circledast d}$ and has a natural isomorphism  $M\circledast Y^{\circledast d}\cong M \circledast Y^{\otimes d}$.  In the following we will consider tensor products involving $\circledast$ powers.

\subsection{}\label{sec: apply for Cc}
From now on, we focus on  $M = \bbc_\bc$ which is regarded as a $U(\ppp)$-module. Then we define the $(W_{\chi}, {\SHd})$-bimodule
structure on $\bbc_\bc \circledast V^{\otimes d}$. By the arguments as in \cite[\S3.3]{BKl}, the induction starting from (\ref{aVV}) gives rise to the following isomorphism of $W_\chi$-modules:
$$\mu_{\bbc_\bc, V^{\otimes d}}:\bbc_\bc\circledast V^{\circledast d}{\overset{\sim}{\longrightarrow}} \bbc_\bc\otimes V^{\otimes d}=V_\bc^{\otimes d}.$$
 Now we have the following corollary to Theorem  \ref{thm:chi-isom} and Corollary \ref{cor:mu-isom}, for the case with $M=\bbc_\bc$ and $X=V^{\otimes d}$.

\begin{corollary}\label{cor: chimu-Vtensord} Keep the notation $I:=I(m|n)$. The following hold:
\begin{itemize}
		\item[(1)] For all $\bfr=(r_1,\ldots,r_d), \bft=(t_1,\ldots,t_d)\in I^d$, there exit elements  $x_{\bfr,\bft}:=x_{r_1,t_1}\cdots x_{r_d,t_d} \in
	U(\mathfrak{p})$
satisfying
	\begin{itemize}
		\item[(i)] $[e_{i,j}, \eta_\bc(x_{\bfr,\bft})] + \sum_{\bfs\in I^d}
		\eta_\bc(x_{\bfs, \bft})
		\in I_\chi$ as long as $e_{i,j} \in \mathfrak{m}$, where the sum is over all $\bfs\in I^d$ obtained from $\bfr$ by replacing an entry equal to $i$ by $j$;	
		\item[(ii)] The action of $x_{\bfi,\bfj}$ on $\bbc_\bc$ is given by
		$x_{{\bfi},{\bfj}}.1_\bc=\delta_{{\bfi},{\bfj}}1_\bc$;
		\end{itemize}
\item[(2)] Under the isomorphism $\mu_{\bbc_\bc, V^{\otimes d}}$, the inverse
element of $v_\bfj\in V_\bc^{\otimes d}$ is
	$\sum_{\bfi\in I^d}
	(\eta_\bc(x_{{\bfi},{\bfj}})1_\chi \otimes
	1_\bc)\otimes v_{{\bfi}}\in\bbc_\bc \circledast V^{\otimes d}\subseteq (Q_\chi\otimes_{W_\chi}\bbc_\bc)\otimes V^{\otimes d}$.\end{itemize}
\end{corollary}
\begin{proof} This corollary follows from Theorem  \ref{thm:chi-isom}  and Corollary \ref{cor:mu-isom}, along with the computation as below.

 Note that  $e_{i,j}(v_t)=\delta_{j,t}v_i$, and we have $c_{r,s}(e_{i,j})=\delta_{r,i}\delta_{s,j}$ for $i,j,r,s\in I=I(m|n)$. By Theorem \ref{thm:chi-isom}, if $e_{i,j}\in \mmm$, then
 \begin{align}\label{eq: single comp}
[e_{i,j}, \eta_\bc(x_{r,t})] + \sum_{s\in I} c_{r,s}(e_{i,j})\eta_\bc(x_{s,t})= [e_{i,j}, \eta_\bc(x_{r,t})] + \delta_{r,i}\eta_\bc(x_{j,t})\in I_\chi. \end{align}
Note that for $\bfr'=(r_1,\ldots,r_{d-1}), \bft'=(t_1,\ldots,t_{d-1})\in I^{d-1}$, we have
$$[e_{i,j}, \eta_\bc(x_{\bfr,\bft})]=[e_{i,j}, \eta_\bc(x_{\bfr',\bft'})]\eta_\bc(x_{r_d,t_d})+
(-1)^{|e_{i,j}||\eta_\bc(x_{\bfr',\bft'})|}\eta_\bc(x_{\bfr',\bft'})[e_{i,j},\eta_\bc(x_{r_d,t_d})].$$
Taking Theorem  \ref{thm:chi-isom}(i) and \eqref{eq: single comp} into consideration, we have the statement (1)-(i). The remaining statements follow from the above arguments.
\end{proof}

\section{Degenerate affine Hecke algebras and their double centralizers
with $W_\chi$}\label{6}
In this concluding  section, on the basis of \S\ref{sec: skryabin} we exploit the arguments of \cite{BKl} on finite  $W$-algebras to the super case, and finally obtain the duality presented in Theorem \ref{thm:dcthm}.

\subsection{The natural transformations of tensor functors and the dAHAs}

Brundan-Kleshchev introduced the degenerate affine Hecke algebra into the study of tensor representations of finite $W$-algebras in \cite{BKl} and \cite{BKl2}. Their arguments are in principle available to the super case.
Let $\ggg=\gl(V)$ for $V=\bbc^{m|n}$. For a given $W_\chi$-module $M$ and a finite-dimensional $\ggg$-module $V$, define a $W_\chi$-module endomorphism $\Omega:M \otimes V \rightarrow M \otimes V$ by
$$ \Omega =\sum_{i,j\in I(m|n)} (-1)^{|j|} e_{i,j} \otimes e_{j,i}.$$
  Recall that $V=\bbc^{m|n}$ admits a standard basis
 $v_{\bar1},\ldots,v_{\overline{m}}, v_{1},\ldots, v_{n}$.
 Consider the functor $-\circledast
V$ of the category $W_\chi$-mod. Define a natural transformation $x$ from the functor $-\circledast V$ to itself:
for any $W_{\chi}$-module $M$, $x_M$  is an endomorphism of $(- \circledast V)(M)=M\circledast V$, defined by left multiplication by $\Omega$, precisely
\begin{align}\label{eq: omega sign}
\Omega((u 1_\chi \otimes m)
\otimes v) = \sum_{i,j\in I(m|n)} (-1)^{|{j}|+(|i|+|j|)(|u|+|m|)}(e_{i,j} u 1_\chi \otimes m) \otimes e_{j,i}v,
\end{align}
where the $\ggg$-module $Q_\chi\otimes_{W_\chi}M$ is regarded as the first tensor position, and $V$ as the second.
Furthermore, keeping the map $\alpha_{M, X,Y}$ (see \eqref{aVV}) in mind we can define another natural transformation  $s$ from the functor $(-\circledast V)
\circledast V$ to itself as below.
Take
$$s_M: (M \circledast V) \circledast V
\rightarrow (M \circledast V) \circledast V$$
to be the composite map $\alpha_{M, V, V}^{-1} \circ
\widehat s_M \circ \alpha_{M, V, V}$, where  $\widehat s_M$  is the endomorphism of $M
\circledast (V \otimes V)=
\operatorname{Wh}((Q_\chi \otimes_{W_{\chi}} M) \otimes V \otimes V)$  arising from left multiplication by $\Omega^{[2,3]}$, which means
$$\Omega^{[2,3]} ((u 1_\chi \otimes m)\otimes v \otimes v') := \sum_{i,j\in I(m|n)}(-1)^{|j|+(|i|+|j|)|v|}
(u1_\chi \otimes m) \otimes e_{i,j} v \otimes e_{j,i} v'.$$
By a direct check (see \cite[Proposition 5.1]{CW1}), the above can be reformulated as below
\begin{align}\label{eq: reform [23]}
\Omega^{[2,3]} ((u 1_\chi \otimes m)\otimes v \otimes v')=
(-1)^{|v||v'|}(u1_\chi \otimes m) \otimes v'\otimes v.
\end{align}

More generally, for any  $d \geq 1$ we can introduce the following natural endomorphisms of the functor $-\circledast V^{\circledast d}$:
\begin{equation}\label{xxii}
x_i:= \bfone^{d-i} x \bfone^{i-1},\qquad
s_j := \bfone^{d-j-1} s \bfone^{j-1}
\end{equation}
for  $1 \leq i \leq d$ with $d\geq 1$ and   $1 \leq j \leq d-1$ for all $d\geq 2$.
By the arguments parallel to \cite[\S3]{BKl} and \cite[\S8]{BKl2}, all these give rise to  a well-defined right action of the dAHA $\SHd$ on $M\circledast V^{\circledast d}$ for any $W_\chi$-module $M$.

 There is a more convenient way to describe them. To achieve this,
we exploit the natural isomorphism
\begin{equation}
a_d:(-\circledast V)^d
\stackrel{\sim}{\longrightarrow}-\circledast V^{\otimes d},
\end{equation}
which is obtained by iterating the map in \eqref{aVV}.
For  $1 \leq i \leq d$ and $1 \leq j \leq d-1$, let  $(\widehat x_i)_M$  and
$(\widehat s_j)_M$  denote the endomorphisms of the image $(-\circledast
V^{\otimes d})(M)$ defined
by the left multiplication of the elements  $\sum_{h=1}^i \Omega^{[h,i+1]}$  and
$\Omega^{[j+1,j+2]}$, respectively. Here and thereafter, for $s>r$ set
$$\Omega^{[r,s]}=\sum_{i,j\in I(m|n)}(-1)^{|j|}\bfone^{\otimes(r-1)}\otimes e_{i,j}\otimes\bfone^{\otimes(s-r-1)}\otimes e_{j,i}\otimes \bfone^{\otimes(d+1-s)}.$$
Then we have the following reformulation of $x_i$ and $s_j$:
\begin{equation}\label{xisj}
x_i = a_d^{-1} \circ \widehat x_i \circ a_d,
\qquad
s_j = a_d^{-1} \circ \widehat s_j \circ a_d.
\end{equation}

As in the ordinary case (see \cite{BK0} and \cite{BKl}), we have the following important property for those  natural endomorphisms  $x_i,s_j$ of the functor $-\circledast V^{\circledast d}$.
\begin{theorem}\label{thm: real DAHA}
The above $x_i$ and $s_j$ satisfy the relations in Definition \ref{def: dasha} for $\SHd$.
\end{theorem}

\begin{proof}  In the super case, by \eqref{eq: reform [23]} a relation crucially making potential difference is still true as in the ordinary case \cite[\S4.4 and \S8.3]{BKl2}. By a direct check as in the ordinary case (see \cite[\S3.4]{BKl} and  \cite[\S8.3]{BKl2}), all relations in Definition \ref{def: dasha} of $\SHd$ are satisfied for the above $x_i$ and $s_j$. The theorem follows.
\end{proof}

By the above theorem, for any $W_{\chi}$-module $M$ there is a natural representation of the degenerate affine Hecke algebra   ${\SHd}$ on  $M \circledast
V^{\circledast d}$.

\begin{remark} In the case of  the functor $-\otimes V^{\otimes d}$, the corresponding natural transformations $x_i$ and $s_j$ can be similarly defined. They also satisfy the relations of $\SHd$ (see \cite[\S5.1]{CW}).
\end{remark}

\subsection{} Turn to the case $M=\bbc_\bc$. In \S\ref{sec: apply for Cc} 
there is already an isomorphism of $W_\chi$-modules: $$\bbc_\bc\circledast V^{\circledast d}\cong \bbc_\bc\otimes V^{\otimes d}=V_\bc^{\otimes d}.$$
The
${\SHd}$-module $\mathbb C_\bc \circledast V^{\circledast d}$ can be lifted to
$\bbc_\bc \circledast V^{\otimes d}$,  and
then to  $V_\bc^{\otimes d}$. This makes them into $(W_{\chi}, {\SHd})$-bimodules.
Let us describe the actions of the generators of  ${\SHd}$  on $\bbc_\bc
\circledast V^{\otimes d}$  and  $V_\bc^{\otimes d}$ explicitly.
Let $x_i$  and  $s_j$   act on
$$
\bbc_\bc \circledast V^{\otimes d} \subseteq
(Q_\chi \otimes_{W_\chi} \bbc_\bc) \otimes V^{\otimes d}
$$
in the same way as the endomorphisms of these modules defined by the
multiplication of $\sum_{h=1}^i \Omega^{[h,i+1]}$ and $\Omega^{[j+1,j+2]}$,
respectively. Regard $Q_\chi \otimes_{W_{\chi}} \bbc_\bc$ as the first tensor
position and the copies
of $V$ as positions $2,3,\dots,d+1$.

We now describe the action of  ${\SHd}$   on $V_\bc^{\otimes d}$ in this special case. Each $s_i$ acts by permuting the $i$th and $(i+1)$th tensor positions, arising from the morphism  $\mu_{\bbc_\bc, V^{\otimes d}}: \bbc_\bc \circledast V^{\otimes
 d}
 \rightarrow V_\bc^{\otimes d}$ as described in Corollary \ref{cor:mu-isom}. The action of each  $x_s$ is described by the
 following lemma.
\begin{lemma}\label{action} Set $I=I(m|n)$ as before.
	For $\mathbf{i}\in I^d$ with $1 \leq s \leq d$, the action of $x_s$ on
$V_\bc^{\otimes d}$ satisfies
\begin{align*}
	v_{{\mathbf{i}}} x_s =c_{\text{col}(i_s)}v_{{\mathbf{i}}}+ (-1)^{|i_s|} v_{\bfi-\imath_s}&+\sum_{\overset{1\leq t< s}{\text{col}(i_t)\geq\text{col}(i_s)}}(-1)^{|i_t|} v_{\bfi(t\; s)}\cr
&-\sum_{\overset{s<t\leq d}{\text{col}(i_t)<\text{col}(i_s)}} (-1)^{|i_t|}v_{\bfi(s\; t)},
\end{align*}
where $\bfi(t\; s)$ stands for the transposition interchanging $t$ and $s$ in $S_d$, and the meaning of $v_{\bfi-\imath_s}$ is the same as in (\ref{eq: imath}).
\end{lemma}

\begin{proof}
First note that $x_{j+1}=s_jx_{j}s_j+s_j$ for $j=1,2,\ldots,d-1$. Carrying induction on  $j$, we just
need to check the situation with $j=1$. Let  $\mu$ denote the map  $\mu_{\mathbb
C_\bc,V^{\otimes d}}$ for simplicity. For $\bfi=(i_1,\ldots,i_d), \bfj=(j_1,\ldots,j_d)\in I^d$, take $x_{\bfi,\bfj}\in U(\ppp)$ as in  Corollary \ref{cor: chimu-Vtensord}. Recall that $\Omega^{[1,2]}=\sum_{i,j\in I}(-1)^{|j|}e_{i,j}\otimes e_{j,i}\otimes\bfone^{\otimes d-1}$.  Keeping \eqref{eq: omega sign} in mind, by definition we then have
\begin{align*}
			v_{\mathbf{i}} x_1 &=
			\mu\bigg(
			\Omega^{[1,2]}\Big(\sum_{\mathbf{j} \in I^d} (\eta_\bc(x_{\mathbf{j},\mathbf{i}}) 1_\chi
\otimes 1_\bc) \otimes v_{\mathbf{j}}\Big)
			\bigg)\cr
			&=
			\sum_{\mathbf{j},\mathbf{k}\in I^d}(-1)^{|k_1|+(|j_1|+|k_1|)|x_{\bfj,\bfi}|}\mu((e_{j_1,k_1} \eta_\bc(x_{\mathbf{j},\mathbf{i}}) 1_\chi \otimes 1_\bc)
			\otimes v_{\mathbf{k}})\cr
&=
			\sum_{\mathbf{j},\mathbf{k}\in I^d}(-1)^{|k_1|}\mu((e_{j_1,k_1} \eta_\bc(x_{\mathbf{j},\mathbf{i}}) 1_\chi \otimes 1_\bc)
			\otimes v_{\mathbf{k}})
		\end{align*}
where in the last summands, $\bfj,\bfk\in I^d$ satisfy the condition that $j_2=k_2,\ldots,j_d=k_d$; and the last equality is due to the fact $|x_{\bfj,\bfi}|=\bz$ (see  Theorem \ref{thm:chi-isom}(i)). In the following, we proceed with the arguments in different cases.
	
(i) First assume that $\operatorname{col}(j_1) \leq \operatorname{col}(k_1)$.
According to Corollary \ref{cor: chimu-Vtensord}(1)-(ii)  and the definitions of
$\mu$ and $\bbc_\bc$, we have $\mu((e_{j_1,k_1} \eta_\bc(x_{\mathbf{j},\mathbf{i}}) 1_\chi \otimes 1_\bc)
	\otimes v_{\mathbf{k}}) = 0$ except $\mathbf{i}=\mathbf{j} =\mathbf{k}$, and in this exceptional case we have
	\begin{align*}
	(-1)^{|k_1|}\mu((e_{j_1,k_1} \eta_\bc(x_{\mathbf{j},\mathbf{i}}) 1_\chi \otimes 1_\bc)\otimes v_{\mathbf{k}})
&=(-1)^{|i_1|} \mu((e_{i_{1},i_{1}}\eta_\bc(x_{\mathbf{i},
\mathbf{i}})1_{\chi}\otimes 1_\bc)\otimes v_{\mathbf{i}})\cr
	&=c_{\text{col}(i_1)}v_{\bfi}.
\end{align*}

(ii)	Now consider the terms with $\operatorname{col}(k_1) <
\operatorname{col}(j_1)$. With the identification of $W_\chi$ as in Lemma \ref{isoWchi},  $e_{j_{1}, k_{1}}$ becomes an element of $\mathfrak{m}$ (see \eqref{eq: p h m}). Note that
$$e_{j_1,k_1}
	\eta_\bc(x_{\mathbf{j},\mathbf{i}}) 1_\chi
	=\eta_\bc(x_{\mathbf{j},\mathbf{i}}) e_{j_1,k_1} 1_\chi +[e_{j_1,k_1},
\eta_\bc(x_{\mathbf{j},\mathbf{i}})]1_{\chi}.$$
So we have
\begin{align}\label{eq: mu app}
&(-1)^{|k_1|}\mu((e_{j_1,k_1} \eta_\bc(x_{\mathbf{j},\mathbf{i}}) 1_\chi\otimes 1_\bc)
			\otimes v_{\mathbf{k}})\cr
=&(-1)^{|k_1|} \mu ((\eta_\bc(x_{\mathbf{j},\mathbf{i}}) e_{j_1,k_1} 1_\chi\otimes 1_\bc)\otimes v_\bfk) +(-1)^{|k_1|}\mu(([e_{j_1,k_1},
\eta_\bc(x_{\mathbf{j},\mathbf{i}})]1_{\chi}\otimes 1_\bc)\otimes v_\bfk).
\end{align}
In the above equality, for the first summand on the right-hand side we know that
$\chi(e_{j_1,k_1})$ is zero unless $k_1 = j_{1}-1$ because the corresponding nilpotent element $e$ is  principal of the form (\ref{eq: e's expression}).   In the case, Corollary \ref{cor: chimu-Vtensord}(2) entails
$$\sum_{\bfj,\bfk\in I^d}
(-1)^{|k_1|} \mu((\eta_\bc(x_{\mathbf{j},\mathbf{i}}) e_{j_1,k_1} 1_\chi\otimes 1_\bc)\otimes
v_{\mathbf{k}}) = (-1)^{|i_1|} v_{\bfi-\imath_1},$$
where the meaning of $v_{\bfi-\imath_1}$ is the same as in (\ref{eq: imath}).

	For the second summand on the right-hand side of \eqref{eq: mu app}, thanks to
Corollary  \ref{cor: chimu-Vtensord}(1)-(i), by calculation we can conclude that
	\begin{align*}
&\sum_{\bfj,\bfk\in I^d}(-1)^{|k_1|} \mu(([e_{j_1,k_1},	\eta_\bc(x_{\mathbf{j},\mathbf{i}})]1_{\chi}\otimes 1_\bc) \otimes
v_{\mathbf{k}}) \cr
=& \sum_{\bfk\in I^d}- (-1)^{|k_1|}\bigg(\sum_{\bs\in I^d} \mu((\eta_\bc(x_{\bfs,\bfi})1_{\chi}\otimes
1_\bc) \otimes v_{\mathbf{k}})\bigg),
\end{align*}
where $\bfs\in I^d$ is obtained from $\bfj$ by replacing an entry equal to $j_1$ by $k_1$. Keeping in mind that $k_t=j_t$ for $t=2,\ldots,d$, by Corollary \ref{cor:mu-isom} and  $x_{\bfs,\bfi}1_{\bc}=\delta_{\bfs,\bfi}1_\bc$ (Corollary \ref{cor: chimu-Vtensord}(1)-(ii)) we have  that $\bfi$ is obtained from $\bfj$ by replacing an entry equal to $j_1$ by $k_1$, and all possible $\bfj$ are of the form $(i_1, i_2,\ldots, i_{t-1},i_1,i_{t+1},\ldots, i_d)$ for  $t\in\{2,\ldots,d\}$ as long as $\text{col}(i_1)>\text{col}(i_t)$. This implies that $k_1=i_t$. Correspondingly, all possible $\bfk$ are of the following form
\begin{align}\label{eq: bfk}
(i_t, i_2,\ldots, i_{t-1},i_1,i_{t+1},\ldots, i_d) \;\;\text{ with }  \text{col}(i_t)<\text{col}(i_1).
\end{align}
  By summing up, we have
\begin{align}
&\sum_{\bfj,\bfk\in I^d}(-1)^{|k_1|} \mu(([e_{j_1,k_1},	\eta_\bc(x_{\bfj,\bfi})]1_{\chi}\otimes 1_\bc)\otimes
v_{\bfk}) \cr
=& -\sum_{\overset{1<t\leq d}{\text{col}(i_t)<\text{col}(i_1)}} (-1)^{|i_t|} \mu((1_{\chi}\otimes
1_\bc)\otimes v_{\bfi(1\; t)})\cr
=& - \sum_{\overset{1<t\leq d}{\text{col}(i_t)<\text{col}(i_1)}}(-1)^{|i_t|} v_{\bfi(1\; t)}.
\end{align}

Therefore, we have
	\begin{align}\label{eq: action of x_1}
	v_{{\mathbf{i}}} x_1 =c_{\text{col}(i_1)}  v_{{\mathbf{i}}}+(-1)^{|i_1|} v_{\bfi-\imath_1}
-\sum_{\overset{1<t\leq d}{\text{col}(i_t)<\text{col}(i_1)}}(-1)^{|i_t|}  v_{\bfi(1\; t)}.
	\end{align}
\end{proof}
\subsection{Degenerate cyclotomic Hecke algebra}\label{map for deg sk}
\begin{lemma}\label{minimal} Keep the notations as before, in particular  $\ggg=\gl(m|n)$ ($m\leq n$). For $d \geq 2$, the minimal polynomial of the
	endomorphism
	of  $V_\bc^{\otimes d}$  defined by the action of
	$x_1$ is  $\prod_{i=1}^n(x - c_i)$.	
\end{lemma}
\begin{proof}
Keeping in mind \eqref{eq: imath} and \eqref{eq: bfk},
 by \eqref{eq: action of x_1} we have that the minimal polynomial of
 $V_\bc^{\otimes d}$ divides  $\prod_{i=1}^n(x-c_i)$ for any  $d\geq 1$. Recall the filtration of
 $\text{End}_{\mathbb{C}}(V_\bc^{\otimes d})$ defined as in \eqref{5.1}. Thanks to Lemma
 \ref{action}, the endomorphism of  $V_\bc^{\otimes d}$  defined by
	$x_1$ belongs to  $\operatorname{F}_1\operatorname{End}_{\mathbb
C}(V_\bc^{\otimes d})$, and the associated graded endomorphism of $V_\bc^{\otimes d}$ is
equal to  $\wp\circ e\otimes \bfone^{\otimes(d-1)}$ where $\wp\in \End_{\bbc}(V)$ is defined via $\wp(v)=(-1)^{|v|}v$ for any $\bbz_2$-homogeneous vector $v\in V$.
Since $e$ is principal nilpotent with
Jordan block of size $n$,   the corresponding minimal polynomial of $(\wp\circ e \otimes \bfone^{\otimes(d-1)})$ is exactly
equal to  $x^{n}$.
This implies that the degree of minimal polynomial of the
endomorphism defined by the action of $x_1$
cannot be of smaller than  $n$, completing the proof.
\end{proof}
Let $\Lambda_\bc =\sum_{i=1}^n\Lambda_{c_i}$ be an element of the free abelian group generated by symbols $\{\Lambda_a\mid a\in\bbc\}$. The corresponding degenerate cyclotomic
Hecke algebra ${\SHd}(\Lambda_\bc)$
is the quotient
 of ${\SHd}$ by the two-sided ideal
generated
by  $\prod_{i=1}^n(x - c_i)$.
By view of Theorem \ref{thm: real DAHA} and Lemma \ref{minimal}, the right action of ${\SHd}$ on
 $V_\bc^{\otimes d}$  factors through the quotient  ${\SHd}(\Lambda_\bc)$, and we can
 obtain
a homomorphism
\begin{equation}\label{psid}
\Psi_{d,\bc}: {\SHd}(\Lambda_\bc) \rightarrow \operatorname{End}_{\mathbb C}(V_\bc^{\otimes
d})^{\operatorname{\op}}.
\end{equation}

Define a filtration
$\operatorname{F}_0{\SHd}(\Lambda_\bc) \subseteq \operatorname{F}_1{\SHd}(\Lambda_\bc)
\subseteq \cdots$ by declaring that
$\operatorname{F}_r {\SHd}(\Lambda_\bc)$ is the span of all
$x_1^{i_1} \cdots x_d^{i_d} w$ for $i_1,\dots,i_d \geq 0$ and $w \in {\fraksd}$
with $i_1+\cdots+i_d \leq r$.
Then
there is a well-defined surjective homomorphism of graded superalgebras
$$
\zeta_d:
\mathbb C_n[x_1,\dots,x_d]\rrtimes \mathbb C {\fraksd}
\twoheadrightarrow
\grsf {\SHd}(\Lambda_\bc)
$$
such that $x_i \mapsto \grsf_1 x_i$, $s_j \mapsto \grsf_0
s_j$ for each  $i$, $j$.
By Lemma  \ref{action} we see that
the map  $\Psi_{d,\bc}$ defined in (\ref{psid}) is a homomorphism of filtered algebras
and
\begin{equation}\label{bag}
(\grsf \Psi_{d,\bc}) \circ \zeta_d = \bar\psi_d.
\end{equation}
Moreover, we have
\begin{lemma}\label{ze}
	the map  $\zeta_d:
	\mathbb C_n[x_1,\dots,x_d]
	\rrtimes \mathbb C {\fraksd}
	\rightarrow
	\grsf {\SHd}(\Lambda_\bc)$  is an isomorphism of graded
superalgebras.
\end{lemma}
\begin{proof}
Note that the monomial vectors
$v_{\mathbf{t}}=v_{t_{1}}\otimes
v_{t_{2}}\otimes\cdots\otimes v_{t_{d}}\in V^{\otimes d}$ for any $\bft=(t_1,\ldots,t_d)\in I^d$ constitute a basis of $V^{\otimes d}$. One can check directly that the map $\bar\psi_d$ in Theorem \ref{doublecentralizerin} is injective.
Hence by \eqref{bag} the map $\zeta_d$ is injective too.
\end{proof}

\subsection{Proof of Theorem \ref{thm:dcthm}: a higher level Schur-Sergeev
duality} \label{pf last}

The following lemma is a generalization of \cite[Lemma  3.6]{BKl}.

\begin{lemma}\label{thetrick}
	
	Let
	$\Phi:B \rightarrow A$ and $\Psi:C \rightarrow A$ be homomorphisms of
	filtered superalgebras such that
	$\Phi(B) \subseteq Z_A(\Psi(C))$, where
	$Z_A(\Psi(C))$ is the centralizer of  $\Psi(C)$  in  $A$. View the
subsuperalgebras $\Phi(B)$, $\Psi(C)$ and $Z_A(\Psi(C))$
	of $A$ as filtered superalgebras with filtrations
	induced by the one on $A$, so that the associated graded
	superalgebras are naturally subalgebras of  $\grsf A$. Then
	$$
	(\grsf \Phi)(\grsf B) \subseteq \grsf
\Phi(B)
	\subseteq \grsf Z_A(\Psi(C))
	\subseteq Z_{\grsf A} (\grsf \Psi(C))
	\subseteq Z_{\grsf A}( (\grsf \Psi)(\grsf
C)).
	$$
\end{lemma}

Now we consider the following maps:
\begin{equation*}\label{co2}
W_{\chi} \stackrel{\Phi_{d,\bc}}{\longrightarrow}
\operatorname{End}_{\mathbb C}(V_\bc^{\otimes d})
\stackrel{\Psi_{d,\bc}}{\longleftarrow}
{\SHd}(\Lambda_\bc).
\end{equation*}
Thanks to Theorem \ref{cdd1} along with Remark \ref{pbw1}, we can identify  the associated
graded map
 $\grsf \Phi_{d,\bc}$
with the map  $\phi_{d,\bc}$  in \S\ref{sec: bc and filtered}. Moreover, Lemma
 \ref{ze} and  (\ref{bag}) enable us to identify
$\grsf \Psi_{d,\bc}$  with  $\bar\psi_d$. Taking
$A=\operatorname{End}_{\mathbb C}(V_\bc^{\otimes d})$ and
 $(B,C, \Phi,\Psi)=
(W_{\chi}, {\SHd}(\Lambda_\bc)^{\operatorname{\op}}, \Phi_{d,\bc}, \Psi_{d,\bc})$ or $({\SHd}(
\Lambda_\bc)^{\operatorname{\op}}, W_{\chi}, \Psi_{d,\bc}, \Phi_{d,\bc})$,  it follows  from
Theorem \ref{doublecentralizerin} that
$$(\grsf \Phi)(\grsf B) = Z_{\grsf
A}((\grsf \Psi)(\grsf C)).$$

By virtue of Lemma  \ref{thetrick}, we have the following theorem.

\begin{theorem}\label{thm: concluding} Keep the notations as previously, in particular, $\ggg=\gl(m|n)$ ($m\leq n$), and $e\in\ggg_\bz$ is a principal nilpotent element. Then the following double centralizer property hold:
		\begin{align*}
		\Phi_{d,\bc}(W_{\chi}) &= \operatorname{End}_{{\SHd}(\Lambda_\bc)}(V_\bc^{\otimes
d}),\cr
		\operatorname{End}_{W_{\chi}}(V_\bc^{\otimes d})^{\operatorname{\op}}&=
\Psi_{d,\bc}({\SHd}(\Lambda_\bc)).
	\end{align*}	
\end{theorem}		

As a special case when $\bc=(0,\ldots,0)$, Theorem
\ref{thm:dcthm} follows from the above general statement.


\begin{remark} (1) One may naturally expect that the super Vust theorem is true for all nilpotent elements.

(2)  Furthermore, one may expect that the Schur-Sergeev duality for principal $W$-superalgebras in the present paper could be  established for all finite $W$-superalgebras associated with $\gl(m|n)$ and any nilpotent elements in $\gl(m|n)_\bz$.

(3) Related to the above expectations, there is a work \cite{Pe2} worth being mentioned where the author explicitly gave an isomorphism of $\bbc$-algebras between the finite $W$-superalgebra associated with an arbitrary nilpotent $e\in\gl(m|n)_\bz$ and a quotient of a certain subalgebra of the super Yangian $Y(m|n)$. This result  generalizes the main result of \cite{BK0} for $\gl(m|n)$.
\end{remark}

\end{document}